\documentclass[10pt,a4paper]{amsart}
\usepackage[utf8]{inputenc}
\usepackage[T1]{fontenc}
\usepackage{amsmath}
\usepackage{amsfonts}
\usepackage{amssymb}
\usepackage{amsaddr}
\usepackage{enumitem}
\usepackage{mathtools}

\usepackage{amsthm}
\usepackage{color}
\usepackage{bbm}
\usepackage{fullpage}
\usepackage{comment}
\author{Malo Jézéquel}
\date{}
\address{
CNRS, Sorbonne Université, Université Paris Diderot, Institut de Mathématiques de Jussieu-Paris Rive Gauche, IMJ-PRG, F-75005, Paris, France. \\  
Current adress: \\
Laboratoires de Probabolité, Statistiques et Modélisation (LPSM), CNRS, Sorbonne Université, Université de Paris, 4, Place Jussieu, 75005, Paris.
\\(email: jezequel@lpsm.paris)}
\title{Local and global trace formulae for smooth hyperbolic diffeomorphisms}
\newcommand{\B}{\mathcal{B}}
\newcommand{\p}[1]{\left( #1 \right)}

\newcommand{\bin}[2]{\left(\begin{array}{c}
                       #2 \\
                       #1
                       \end{array} \right)}

\newcommand{\F}{\mathbb{F}}
\newcommand{\N}{\mathbb{N}}
\newcommand{\R}{\mathbb{R}}
\newcommand{\T}{\mathbb{T}}
\newcommand{\Z}{\mathbb{Z}}
\newcommand{\C}{\mathbb{C}}
\newcommand{\M}{\mathcal{M}}

\newcommand{\A}{\mathcal{A}}

\newcommand{\G}{\mathcal{G}}
\newcommand{\U}{\mathcal{U}}
\renewcommand{\L}{\mathcal{L}}
\newcommand{\h}{\mathcal{H}}

\newcommand{\s}[3]{\int_{#1} #2 \mathrm{d}#3}
\newcommand{\set}[1]{\left\{#1\right\}}

\newcommand{\bul}[1]{\stackrel{\circ}{#1}}
\renewcommand{\phi}{\varphi}
\renewcommand{\epsilon}{\varepsilon}

\renewcommand{\bar}[1]{\overline{#1}}
\newcommand{\n}[1]{\left\|#1\right\|}
\newcommand{\hra}{\hookrightarrow}
\newcommand{\nhra}{\not\hookrightarrow}
\renewcommand{\b}[1]{\left| #1 \right|}
\newcommand{\tf}[1]{\textup{tr}^\flat\p{#1}}

\renewcommand{\sqsubset}{\Subset}

\newtheorem{lm}{Lemma}[section]
\newtheorem{prop}[lm]{Proposition}
\newtheorem{cor}[lm]{Corollary}
\newtheorem{thm}[lm]{Theorem}

\newtheorem*{conj}{Conjecture}
\theoremstyle{definition}
\newtheorem{rmq}[lm]{Remark}
\newtheorem{df}[lm]{Definition}
\newtheorem{ex}[lm]{Example}

\usepackage{pgf,tikz}
\usetikzlibrary{arrows}

\begin{document}

\begin{abstract}
We define and study local and global trace formulae for discrete-time uniformly hyperbolic weighted dynamics. We explain first why dynamical determinants are particularly convenient tools to tackle this question. Then we construct counter-examples that highlight that the situation is much less well-behaved for smooth dynamics than for real-analytic ones. This suggests to study this question for Gevrey dynamics. We do so by constructing an anisotropic space of ultradistributions on which a transfer operator acts as a trace class operator. From this construction, we deduce trace formulae for Gevrey dynamics, as well as bounds on the growth of their dynamical determinants and the asymptotics of their Ruelle resonances.
\end{abstract}

\maketitle

\section*{Introduction}

The aim of this paper is to discuss \emph{trace formulae} for discrete-time uniformly hyperbolic weighted dynamics. This work is thought as a first step toward the following conjecture of Dyatlov and Zworski.

\begin{conj}[Dyatlov-Zworski \cite{DZdet}]
Let $M$ be a compact manifold and $\phi_t : M \to M$ a $\mathcal{C}^\infty$ Anosov flow. Then the trace formula
\begin{equation}\label{cdz}
\sum_{\mu \in \textrm{ Res}\p{P}} e^{-i \mu t} = \sum_{\gamma} \frac{T_\gamma^{\#} \delta\p{t - T_\gamma}}{\b{\det\p{I - \mathcal{P}_\gamma}}}, t > 0
\end{equation}
holds in the sense of distributions on $\R_+^*$. The sum on the left-hand side ranges over resonances of the operator $P = -i V$ where $V$ is the generator of the flow $\phi_t$, the sum of the right-hand side ranges over periodic orbits $\gamma$ of $\phi_t$. If $\gamma$ is a periodic orbit, $T_\gamma$ denotes its period, $T_\gamma^{\#}$ its primitive period and $\mathcal{P}_\gamma$ the associated linearized Poincaré map. See \cite{DZdet} for definitions.
\end{conj}

This conjecture is known to be true when $\phi_t$ is the geodesic flow on the unit tangent bundle of a compact Riemannian surface with constant negative curvature, as a consequence of Guillemin trace formula and Selberg trace formula. A local version of \eqref{cdz} has been proven by Jin and Zworski in \cite{ltfzwor}. The formula \eqref{cdz} is called a trace formula because it intuitively expresses that the trace of a transfer operator is the sum of its eigenvalues. A common way to prove the validity of \eqref{cdz} is to establish the finiteness of the order of the dynamical determinant (see Definition \ref{order} for the definition of the order of an entire function)
\begin{equation}\label{detcont}
d_{\phi_t}\p{\lambda} = \exp\p{-\sum_{\gamma} \frac{T_\gamma^{\#} e^{i \lambda T_\gamma}}{T_\gamma \b{\det\p{I - \mathcal{P}_\gamma}}}},
\end{equation}
which is known to be entire when the stable and unstable bundles of $\phi_t$ are orientable (see \cite{DZdet}). The formula \eqref{cdz} then follows from results of pure complex analysis (see \cite[Theorem 17]{MunozMarco} and \cite[Theorem 3.1]{MunozMarco2}). Consequently, classical results of Ruelle, Rugh and Fried \cite{Ruelle1,R2,R2,fried,friedzeta} imply that the conjecture of Dyatlov and Zworski is true under some assumptions of analyticity of $\phi_t$.

Our point here is to study an analogue of the conjecture of Dyatlov and Zworski for discrete-time uniformly hyperbolic weighted dynamics that are $\mathcal{C}^\infty$ but \emph{a priori} non-analytic. Our results suggest that \eqref{cdz} may not hold for all $\mathcal{C}^\infty$ uniformly hyperbolic flows. However, we introduce tools in \S \ref{locspace} and \S \ref{lto} that could presumably be improved to prove \eqref{cdz} for flows that are Gevrey (Gevrey differentiability is a class of regularity intermediate between $\mathcal{C}^\infty$ and real-analyticity, the needed definitions are recalled in \S \ref{gevsec}).

The main tool to formulate an analogue of the conjecture of Dyatlov and Zworski in our discrete time context is the following transfer operator ($T : M \to M$ is the dynamics and $g : M \to \C$ a weight)
\begin{equation}
L : u \to g.\p{u \circ T}.
\end{equation}
The operator $L$ can for instance be defined on $\mathcal{C}^\infty$ and then be extended to an operator $\L$ on some carefully designed anisotropic Banach spaces. Associated to this operator are the Ruelle resonances of the system $\p{T,g}$, which are in a sense the relevant eigenvalues of $L$ (see Definition \ref{RR}) and describe statistical properties of the dynamics. If $g >0$, the Ruelle resonances allow to describe accurately the asymptotics of the correlations for the Gibbs measure $\mu_{ g}$ associated to the observable $\log g - J_u$ where $J_u$ is the unstable jacobian of $T$. That is, for all $\phi,\psi : M \to \C$ smooth, we can give an asymptotic development for 
\begin{equation}\label{fusil}
\s{}{\phi . \psi \circ T^n}{\mu_{ g}},
\end{equation}
when $n$ tends to $+ \infty$, with a geometric error term of arbitrarily small ratio. A natural candidate for the trace of the transfer operator is the "flat" trace which is defined for $n \geqslant 1$ by
\begin{equation}
\textrm{tr}^\flat \p{\L_g^n} = \sum_{ T^n x = x} \frac{\prod_{k=0}^{n-1}g\p{T^k x}}{\b{\det\p{I-D_x T^n}}}.
\end{equation}
This is actually the trace of the transfer operator in certain (real-analytic) contexts in which the transfer operator can be made nuclear (or trace class, see \cite{BanNaud} for instance). See \cite[Remark 3.1]{Bal2} for a heuristic justifying the definition of the flat trace. Consequently, a natural analogue of \eqref{cdz} in our context is
\begin{equation}\label{itr}
\textrm{tr}^\flat \p{\L_g^n} = \sum_{\lambda \textrm{ resonance}} \lambda^n.
\end{equation}
Notice that we can then ask for which values of $n$ the equality holds, as well as the sense in which the right-hand side is defined (in particular, whether the convergence is absolute or not). The analogue of \eqref{detcont} in our discret-time context is the following dynamical determinant 
\begin{equation}\label{idet}
d_{T,g}\p{z} = \exp\p{-\sum_{n=1}^{+ \infty} \frac{1}{n} \sum_{T^n x = x} \frac{\prod_{k=0}^{n-1}g\p{T^k x}}{\b{\det\p{I-D_x T^n}}}z^n}
\end{equation}
which has been widely studied, see for instance \cite{Tsu,Bal2,livdet,LivTsu}. Our first task will be to establish the link between trace formulae and the dynamical determinant \eqref{idet}, which happens to be quite strong : trace formulae can be rephrased merely as properties of the dynamical determinant that imply its order (see Theorems \ref{carac} and \ref{ac} and Remark \ref{application}).

The second class of results of the paper consists in counter-examples in the $\mathcal{C}^\infty$ case: constructing systems $\p{T,g}$ with explicit dynamical determinants, we see that the situation is much more intricate for smooth systems than for real-analytic ones, in particular trace formulae need not hold (see Proposition \ref{rapp1} and \ref{rapp2}, Corollaries \ref{mechant}, \ref{varie}, \ref{tout},\ref{asbad} and \ref{rien}, and remarks below them).

These important differences between smooth and real-analytic systems suggest to study the behaviour of an intermediate class of regularity. This is the third focus of our paper. A natural candidate is the class of Gevrey functions, that have been introduced by Maurice Gevrey in \cite{gevorg} (see \S \ref{gevsec} for basic definitions). This class shares many features with the class of smooth function (in particular, there are Gevrey bump functions) but also has some properties that recalls those of real-analytic functions, in particular the Fourier transform of a Gevrey function decreases faster than a stretched exponential (see Lemma \ref{decroi}). To study Gevrey systems we introduce a new family of anisotropic Hilbert spaces, whose elements are ultradistributions (that is continuous linear forms on Gevrey functions). The main feature of this new spaces is that the transfer operator act on them as a trace class operator, which will in particular imply that trace formula \eqref{itr} holds (for every $n$). We also get bounds on the asymptotic of the number of Ruelle resonances and on the growth of the dynamical determinant \eqref{idet}. All these results are summed up in Theorem \ref{main} (see also Remark \ref{ord0}). Apart from these new Hilbert spaces, the main new tools that we introduce are a Paley--Littlewood-like decomposition on thinner bands than usual that are adapted to our spaces of ultradistributions (see \S \ref{locspace}) and a notion of generalized cone-hyperbolicity (see Definition \ref{ch}) without which we should require that our systems have enough hyperbolicity in Theorem \ref{main}. Notice that, as far as we know, there was no known constructions of anisotropic Banach spaces on which the transfer operator acts compactly for non-analytic dynamics.

The paper is structured as follows.

In \S \ref{settings} we give precise definitions for the objects we are going to study and state our main results.

In \S \ref{tac}, we highlight the relations between trace formulae and dynamical determinants (this is the content of Theorem \ref{ac} and Remark \ref{application}). We also construct entire functions with particular properties in prevision of \S \ref{explicit}.

In \S \ref{explicit}, we construct systems with explicit dynamical determinants. We first construct symbolic systems (as in \cite{port} or \cite[Example 1 p.165]{PP}) with explicit zeta functions (defined by \eqref{defzeta}) in \S \ref{symbexp}. We then conjugate these symbolic systems to smooth ones in \S \ref{smoothexp}, using Whitney's extension theorem \cite{whitney} as in \cite{sabowen}. This construction is summed up in Proposition \ref{construction}, from which we deduce, using results from \S \ref{tac}, some corollaries that illustrate some behaviours of smooth dynamics that are impossible for real-analytic ones.

Finally, \S \ref{gevsec}, \S \ref{locspace}, \S \ref{lto} and \S \ref{preuve} are dedicated to the proof of our main result, Theorem \ref{main}, by constructing anisotropic Hilbert spaces of ultradistributions adapted to uniformly hyperbolic dynamics. We begin by recalling some elementary facts about Gevrey functions and ultradistributions in \S \ref{gevsec}. Then we construct "local" anisotropic spaces in \S \ref{locspace}, and we study the action of a "local" transfer operator on these spaces in \S \ref{lto}. Then, we glue our local spaces together to construct a global one and prove Theorem \ref{main} in \S \ref{preuve}.

\tableofcontents

\section{Settings and statement of results}\label{settings}

Let $M$ be a smooth $d$-dimensional manifold (smooth means $\mathcal{C}^\infty$ throughout the paper), let $T : M \to M$ be a smooth diffeomorphism and recall the following definition.

\begin{df}[Hyperbolic basic set]\label{bashyp}
A $T$-invariant compact subset $K$ of $M$ is said to be a \emph{hyperbolic basic set} for $T$ if the following properties are fulfilled :
\begin{itemize}
\item $K$ is \emph{hyperbolic}, i.e. there is a Riemannian metric on a neighbourhood of $K$ and for all $x \in K$ a decomposition $T_x M = E_x^u \oplus E_x^s$ such that $D_x T\p{E_x^i} = E_{Tx}^i$ for $i \in \set{u,s}$, and there exist $\lambda > 1$ and $c > 0$ such that, for all $n \in \N$ and $x \in K$, we have $ \n{\left. D_x T^n \right|_{E_x^s}} \leqslant c \lambda^{-n}$ and $ \n{\left. D_x T^{-n} \right|_{E_x^s}} \leqslant c \lambda^{-n}$;
\item $K$ is \emph{isolated}, i.e. there is a neighbourhood $V$ of $K$ in $M$ such that $K = \bigcap_{n \in \Z} T^n\p{V}$;
\item $\left. T \right|_{K}$ is transitive.
\end{itemize}  
\end{df}

Let us fix a hyperbolic basic set $K$ for $T$ and a smooth weight $g : M \to \C$. Below, we will need $g$ to be supported in small enough isolating neighbourhood of $K$. We are interested in the transfer operator associated to the system $\p{T,g}$ which is defined by
\begin{equation}\label{transT}
L : u \mapsto g.\p{u \circ T}.
\end{equation}
In particular, we shall investigate the Ruelle resonances of the system $\p{T,g}$ which are in a sense the relevant eigenvalues of $L$. To define these resonances, we first recall the notion of the essential spectral radius of an operator.

\begin{df}[Essential spectral radius]
Let $\B$ be a Banach space\footnote{All the vector spaces that we consider are vector spaces over $\C$.} and let $\L : \B \to \B$ be a bounded operator. The \emph{essential spectral radius} of $\L$ is the smallest non-negative real number $\rho$ such that the intersection of the spectrum of $\L$ with $\set{z \in \C :\b{z} >\rho}$ consists of isolated eigenvalues of finite multiplicity.
\end{df}

Since $T$ and $g$ are smooth, \cite[Theorem 1.1]{Tsu},\cite{GLK} or \cite{Bal2} yield the following.

\begin{thm}\label{fond}
For every small enough compact isolating neighbourhood $V$ for $K$ and for each $\epsilon >0$ there is a Banach space $\B$ such that
\begin{itemize}
\item $\B$ is contained in the space $\mathcal{D}'\p{V}$ of distributions on $M$ whose support is contained in $V$ and the inclusion is continuous;
\item $\mathcal{C}^\infty\p{V}$, the space of smooth functions on $M$ that are supported in $V$, is contained in $\B$ and the inclusion is continuous with dense image;
\item if $g$ is supported in $V$ then the operator $L$ defined on $\mathcal{C}^\infty\p{V}$ by \eqref{transT} extends to a bounded operator $\L : \B \to \B$ whose essential spectral radius is smaller than $\epsilon$.
\end{itemize}
\end{thm}

\begin{rmq}
Using an abstract functional analytic lemma, for instance \cite[Lemma A.3]{Tsu}, we can see that if $\B$ and $\widetilde{\B}$ are two Banach spaces that satisfy the conclusion of Theorem \ref{fond}, then the intersection of $\set{z \in \C : \b{z} > \epsilon}$ with the spectrum of $\L$ is the same when $\L$ acts on $\B$ or on $\widetilde{\B}$. For the Banach spaces that are constructed in \cite{Tsu}, the discrete part of the spectrum of $\L$ can be described in terms of a dynamical determinant, in particular it does not depend on $V$. In fact, the formulae \eqref{defdet} and \eqref{flat} below makes clear that this spectrum only depend of the values of $g$ on $K$. However, we need to $g$ to be supported in a small enough isolating neighbourhood of $K$ in order to make sense of this spectrum.
\end{rmq}

Consequently, we can give the following definition of the Ruelle resonances.

\begin{df}[Ruelle resonances and resonant states]\label{RR}
Let $\lambda \in \C$ and $m \in \N^*$. We say that $\lambda$ is a Ruelle resonance of multiplicity $m$ for $\p{T,g}$ if there are $\epsilon > 0$, a compact isolating neighbourhood $V$ for $K$ and a Banach space $\B$ satisfying the conclusion of Theorem \ref{fond} such that $\b{\lambda} > \epsilon$ and $\lambda$ is an eigenvalue of algebraic multiplicity $m$ of $\L$ acting on $\B$.

Equivalently, $\lambda$ is a Ruelle resonance of multiplicity $m$ for $\p{T,g}$ if for every $0<\epsilon < \b{\lambda}$, every small enough compact isolating neighbourhood $V$ for $K$ and every Banach space $\B$ satisfying the conclusion of Theorem \ref{fond}, the complex number $\lambda$ is an eigenvalue of algebraic multiplicity $m$ of $\L$ acting on $\B$.
\end{df}

To study Ruelle resonances, a convenient tool is the dynamical determinant \eqref{idet}, which can also be expressed as the following formal power series :
\begin{equation}\label{defdet}
d_{T,g}\p{z}= \exp\p{- \sum_{n=1}^{+ \infty} \frac{1}{n} \tf{\L^n} z^n},
\end{equation}
where the "flat trace" of the transfer operator is defined for $n \geqslant 1$ by
\begin{equation}\label{flat}
\tf{\L^n} = \sum_{\substack{x \in K \\ T^n x =x}} \frac{g^{\p{n}}\p{x}}{\b{\det\p{I-D_x T^n}}},
\end{equation}
where, for all $n \in \N$ and $x \in M$, we set $g^{\p{n}}\p{x} = \prod_{k=0}^{n-1} g\p{T^k x}$. While we take \eqref{flat} as a definition of the flat trace, we could give another definition which would make the flat trace an actual function of the transfer operator $\L$, see for instance \cite{Bal2}. In this smooth setting, \cite[Theorem 1.5]{Tsu} expresses the link between dynamical determinant and Ruelle resonances in the following way.

\begin{thm}
$T,g$ and $M$ being smooth, the formal power series \eqref{defdet} converges to an entire function $d_{T,g}$ and the Ruelle resonances for $\p{T,g}$ are exactly the inverses of the zeros of $d_{T,g}$ (multiplicity taken into account).
\end{thm}

Consequently, a control on the growth of the dynamical determinant $d_{T,g}$ implies through Jensen's formula a control on the number of Ruelle resonances outside of a small disc centered at zero (see for instance the proof of \eqref{nombre} in Lemma \ref{abstrait} below). We are particularly interested in trace formulae and we will see in \S \ref{tac} that the dynamical determinant is a particularly convenient tool to study those formulae. We shall say that the \emph{global trace formula} for $\p{T,g}$ holds at step $n \geqslant 1$ if 
\begin{equation}\label{tracefor}
\tf{\L^n} = \sum_{\lambda \textrm{ resonances}} \lambda^n,
\end{equation}
where the series in the right hand side converges absolutely. As in \cite{ltfzwor}, we also define a local counterpart to this notion. We say that $\p{T,g}$ satisfies \emph{local trace formula} if for all $r >0$ such that $\p{T,g}$ has no resonances of modulus $r$ we have
\begin{equation}
\tf{\L^n} \underset{n \to + \infty}{=} \sum_{\substack{ \lambda \textrm{ resonances} \\ \b{\lambda} > r}} \lambda^n + o\p{r^n}
.\end{equation}

\begin{rmq}
The formula \eqref{tracefor} always holds for every $n \geqslant 1$ when $\L$ is replaced by a trace class operator $A$ acting on a Hilbert space and the Ruelle resonances by its eigenvalues thanks to Lidskii trace theorem \cite[Theorem 6.1 p.63]{Gohb}. If $A$ is only supposed to be a nuclear operator acting on a Banach space, \eqref{tracefor} holds for $n \geqslant 2$ but not necessarily for $n =1$ (see \cite[Theorem 4.1 p.106]{Gohb} and the remark preceding it, see also \cite[Corollary 2 p.17, second part of the book]{Groth} and the proof of Theorem \ref{ac}). Thus, it seems natural not to consider \eqref{tracefor} only for $n=1$ but to wonder for which $n$ this formula holds.
\end{rmq}

\subsection{General results for smooth systems}

We will see in \S \ref{tac} that local and global trace formulae can be expressed merely as properties of the dynamical determinant $d_{T,g}$. Then, a result of elementary complex analysis, Theorem \ref{ac}, yields the following theorem (see also Remark \ref{application}, see Definiton \ref{order} or \cite[Definitions 2.1.1 and 2.7.3]{Boas} for the definitions of the order and the genus of an entire function). 

\begin{thm}\label{carac}
$T,g$ and $M$ being smooth, the local trace formula always holds. Furthermore, the following properties are equivalent:
\begin{enumerate}[label=(\roman*)]
\item there is an integer $n_0 \geqslant 1$ such that the global trace formula \eqref{tracefor} holds at step $n$ for all $n \geqslant n_0 + 1$;
\item the dynamical determinant $d_{T,g}$ has finite order. 
\end{enumerate}
In addition, when this holds, the optimal $n_0$ is the genus of $d_{T,g}$.
\end{thm}

Using results from \cite{Ruelle1,R1,R2,fried,friedzeta}, Theorem \ref{carac} implies that if $T,g$ and $M$ are real analytic then the global trace formula always holds at all steps. However, we shall see in \S \ref{explicit} that if $T$ and $g$ are only $\mathcal{C}^\infty$ the situation is very different. Indeed, we construct in Proposition \ref{construction} systems $\p{T,g}$ with explicit dynamical determinants. With the results from \S \ref{tac}, we deduce a bunch of consequences that suggests that the smooth category is not the best context to study trace formulae. Here are two examples, the first one is about trace formulae and the second about asymptotics of Ruelle resonances. 

\begin{prop}\label{rapp1}
Let $E$ be a subset of $\N^*$. Then there exists a system $\p{T,g}$ such that for all $n \in \N^*$ the global trace formula holds at step $n$ if and only if $n \in E$. Moreover, $g$ may be chosen positive on $K$.
\end{prop}

\begin{prop}\label{rapp2}
Let $N_0 : \R_+^* \to \R_+$ be a locally bounded function. Then there is a system $\p{T,g}$ such that if $N\p{r}$ is the number of Ruelle resonances for $\p{T,g}$ outside of the closed disc of center $0$ and radius $r$ (counted with multiplicity) we have
\begin{equation*}
N_0\p{r} \underset{r \to 0}{=} o\p{N\p{r}}
.\end{equation*}
Moreover, $g$ may be chosen positive on $K$.
\end{prop}

Proposition \ref{rapp2} shows that that the asymptotics of the Ruelle resonances for smooth systems may be arbitrarily bad, recall that that for real-analytic systems we have $N\p{r} \underset{r \to 0}{=} O\p{\b{\log r}^{1+d}} $ where $d$ is the dimension of $M$. In Propositions \ref{rapp1} and \ref{rapp2}, the fact that the weights $g$ may be chosen to be positive on $K$ implies that they are associated to physically meaningful Gibbs measures whose asymptotics of correlations are described by the Ruelle resonances for $\p{T,g}$, see \cite{GouLiv2}.

Notice that Proposition \ref{construction} also encompasses some counter-examples in finite differentiability that allows for instance to discuss sharpness of \cite[Theorem 1.5]{Tsu} (see Remark \ref{sharpness}, in particular, for any $r \in \N^*$, there is a $\mathcal{C}^r$ system $\p{T,g}$ for which the dynamical determinant $d_{T,g}$ does not extend to the whole complex plane).

\subsection{Results for Gevrey systems}

Since we observe very different behaviors for real-analytic and $\mathcal{C}^\infty$ systems, it seems natural to investigate the case of Gevrey systems, and that is the object of \S \ref{locspace}, \ref{lto} and \ref{preuve}. We develop there an approach based on the construction of a new Hilbert space on which $\L$ acts as a trace class operator (in particular $\L$ is nuclear). This approach is legitimate thanks to the following short lemma that implies that we can investigate Ruelle resonances by working with spaces of ultradistributions\footnote{In fact, the main point is that we can work with a space that does not contain $\mathcal{C}^\infty\p{V}$.}. The required definitions and properties of Gevrey functions and ultradistributions are given in \S \ref{gevsec}. A proof is given in Appendix \ref{proofequiva}.

\begin{lm}\label{equiva}
Let $\sigma >1$ and assume that $M$, $g$, and $T$ are $\sigma$-Gevrey (see Definition \ref{gfum}). Let $V$ be a compact isolating neighbourhood for $K$, let $\epsilon >0$, and let $\B$ be a Banach space such that ($\U_\sigma\p{V}$ and $\G_\sigma\p{V}$ are defined in Definition \ref{gfum})
\begin{itemize}
\item $\B$ is a subspace of $\U_\sigma\p{V}$ and the injection is continuous;
\item $\G_\sigma\p{V}$ is contained in $\B$ and the injection is continuous with dense image;
\item $L$ extends to a bounded operator $\L :\B \to \B$ whose essential spectral radius is less than $\epsilon$.
\end{itemize}
Then the intersection of $\set{z \in \C : \b{z} > \epsilon}$ with the spectrum of $\L$ on $\B$ coincides with the intersection of $\set{z \in \C : \b{z} > \epsilon}$ with the set of Ruelle resonances of $\p{T,g}$ (multiplicity taken into account) and the corresponding eigenvectors are the resonant states (in particular, the eigenvectors are in fact distributions). 
\end{lm}

We can then state our main theorem, whose proof is given in \S \ref{preuve}, using tools from \S \ref{locspace} and \S \ref{lto}

\begin{thm}\label{main}
Let $\sigma >1$ and assume that $M$, $T$ and $g$ are $\sigma$-Gevrey. Then there exist a compact isolating neighbourhood $V$ for $K$ and a separable Hilbert space $\h$ such that the following holds ($\G_{\sigma}\p{V}$ and $\U_{\sigma}\p{V}$ are defined in Definition \ref{gfum} ):
\begin{enumerate}[label=(\roman*)]
\item the Hilbert space $\h$ is contained in $\U_{\sigma}\p{V}$ and the inclusion is continuous;
\item the Hilbert space $\h$ contains $\G_{\sigma}\p{V}$ and the inclusion is continuous with dense image;
\item if $g$ is supported in $V$, the transfer operator $L$ defined by \eqref{transT} extends to a trace class operator $\L : \h \to \h$;
\item for all $n \in \N^*$ we have $ \textup{tr}\p{\L^n} = \tf{\L^n} $, in particular the dynamical determinant $d_{T,g}$ is the Fredholm determinant of $\L$ (i.e. $d_{T,g}\p{z} = \det\p{I - z \L}$), and the global trace formula \eqref{tracefor} holds at all steps for $\p{T,g}$;
\item\label{better} for all $\beta > 2 + \p{\sigma + 1} d$ we have
\begin{equation*}
\log_+ \b{d_{T,g}\p{z}} \underset{\b{z} \to + \infty}{=} O\p{\p{\log\b{z}}^{1 + \beta}},
\end{equation*}
in particular $d_{T,g}$ has order zero;
\item\label{butter} if $N\p{r}$ is the number of Ruelle resonances for $\p{T,g}$ outside of the closed disc of center $0$ and radius $r$ (counted with multiplicity) we have for all $\beta > 2 +\p{\sigma + 1}d$
\begin{equation*}
N\p{r} \underset{r \to 0}{=} O\p{\b{\log r}^{1 + \beta}}.
\end{equation*}
\end{enumerate}
\end{thm}

Consequently, Gevrey dynamics are much more well-behaved than smooth ones. The proof of points \ref{better} and \ref{butter} in Theorem \ref{main} is based on a particular representation of the transfer operator (obtained via Paley--Littlewood-like decomposition). This representation allows to use the following abstract functional analytic lemma, whose proof, given in Appendix \ref{proofabstrait}, is almost totally taken out from \cite{fried}. If $l \in \B'$ and $e \in \A$, for some Banach spaces $\A$ and $\B$, we write $e \otimes l$ for the rank $1$ operator $u \mapsto l\p{u}. e$.

\begin{lm}\label{abstrait}
Let $\B$ be a Banach space. Let $\L : \B \to \B$ be a nuclear operator that may be written as
\begin{equation}\label{volrep}
\L = \sum_{m \in \N} \lambda_m e_m \otimes l_m
\end{equation}
where $l_m \in \B'$ and $e_m \in \B$ have unit norm, and such that, for all $m \geqslant 0$ and some constants $C <0$, $\beta>0$, and $0 <\theta < 1$,
\begin{equation}\label{dif}
\b{\lambda_m} \leqslant C \theta^{m^{\frac{1}{\beta}}}.
\end{equation}
Write\footnote{Since $\L$ is nuclear of order $0$, its Fredholm determinant is well-defined, see \cite[Corollary 4 p.18 of the second part]{Groth}.}
\begin{equation}\label{defan}
\det\p{I-z\L}  = \sum_{n = 0}^{+ \infty} a_n z^n,
\end{equation}
then there are constants $M,D >0$ such that for all $n \in \N$ we have
\begin{equation*}
\b{a_n} \leqslant M \exp\p{-D n^{1 + \frac{1}{\beta}}}
.\end{equation*}
Furthermore, we have
\begin{equation}\label{crois}
\log_+ \b{\det\p{I-z\L}} \underset{\b{z} \to + \infty}{=} O\p{\p{\log \b{z}}^{1+ \beta}},
\end{equation}
and
\begin{equation}\label{nombre}
N\p{r} \underset{r \to 0}{=} O\p{\b{\log r}^{1 + \beta}},
\end{equation}
where $N\p{r}$ is the number of eigenvalues of $\L$ of modulus greater than $r$ (counted with multiplicity).
\end{lm}

\begin{rmq}\label{ord0}
It will appear in the proof of Theorem \ref{main} that the transfer operator $\L$ admits a representation of the type \eqref{volrep} with the estimate \eqref{dif}. This implies in particular that the transfer operator $\L$ is nuclear of order $0$ in the sense of Grothendieck and, for instance, the use of the Lidskii trace theorem below will never be essential (see \cite{Groth}). Another consequence of this representation is that the use of dynamical determinants for numerical experiments would be very efficient in the Gevrey category, as it has been shown in the real-analytic category (see for instance \cite{Poll} or \cite{jenk}).
\end{rmq}

\begin{rmq}
Proposition \ref{rapp1} makes clear that some assumption has to be made in addition to the smoothness of $M,T$ and $g$ in order to get trace formulae. Asking for $M,T$ and $g$ to be Gevrey is enough, according to Theorem \ref{main}. However, one could imagine requirements of different natures, for instance that $T$ is Anosov (i.e. $K =M$) or that $g$ is a natural weight (e.g. the inverse of the jacobian of $T$). Indeed, our counter-examples do not satisfy such assumptions.
\end{rmq}

\begin{rmq}
The Gevrey assumption is used in the proof of Theorem \ref{main} to get Lemma \ref{gauw}. Consequently, the crucial property of Gevrey functions to get Theorem \ref{main} is Lemma \ref{decroi}: the Fourier transform of a rapidly decreasing Gevrey function decreases faster than a stretched exponential. Thus, we could presumably replace in Theorem \ref{main} Gevrey functions by a larger Denjoy--Carleman class (see \cite{nqa} for a definition of the Denjoy--Carleman classes, one could maybe also define classes of functions by imposing a rate of decay of their Fourier transforms), but this would deteriorate the estimates in \ref{better} and \ref{butter}. Provided that the transfer operator remains nuclear, global trace formulae would still hold. In this sense, Theorem \ref{main} is not optimal if we are merely interested in trace formulae for discrete-time dynamics. However, the spectral picture for discrete-time and continuous-time dynamics differ by an exponential and thus the estimates in \ref{better} and \ref{butter} are the kind of properties that we expect in the perspective of the conjecture of Dyatlov and Zworski (see for instance \cite{friedzeta,fried} in which such estimates imply finiteness of the order of the dynamical zeta function for a continuous-time dynamics).
\end{rmq}

\begin{rmq}
The limit case $\sigma = 1$ in the hierarchy of Gevrey functions corresponds to real-analytic functions. Consequently, it is tempting to replace $\sigma$ by $1$ in Theorem \ref{main}, in particular in \ref{better} and \ref{butter}. This suggests that our result may not be optimal, indeed it is known that in the real-analytic case we can take $\beta = d$ (see \cite{Ruelle1,R1,R2,fried,friedzeta}), why replacing $\sigma$ by $1$ in our result gives $\beta > 2 + \p{d+1}$.
\end{rmq}

\begin{rmq}
As pointed out by the anonymous referee, our construction should allow to generalize some results known in the real-analytic setting to Gevrey hyperbolic diffeomorphisms. It is very likely for instance that one can implement the method of Adam \cite{adam} to prove that near a linear Anosov diffeomorphism of the torus there is a generic set, in some Gevrey topology, of Gevrey diffeomorphisms which have non-trivial resonances (since the method of Adam relies mostly on trace formulae). One could also use similar spaces to study linear response (see \cite{linresp}): indeed, it wouldn't be surprising if the dependence of the operator $\L$ on the dynamics $T$ could be made smooth (when $T$ varies in a space of Gevrey maps).
\end{rmq}

\section{Trace formulae and order of the dynamical determinant}\label{tac}

We shall now explain the link between trace formulae and finite order of the dynamical determinant. We will need the following definition.

\begin{df}\label{ordering}
If $f$ is an entire function, we say that $\p{z_m}_{m \geqslant 0}$ is an ordering of the zeroes of $f$ if $z_0,z_1,\dots,z_m,\dots$ are the zeroes of $f$ counted with multiplicity and the sequence $\p{\b{z_m}}_{m\geqslant 0}$ is non-decreasing.
\end{df}

We shall always order the zeroes of an entire function in this way. Recall the following definitions.

\begin{df}\label{order}
Let $f$ be an entire function. The \emph{order} of $f$ may be defined as
\begin{equation*}
\limsup_{r \to + \infty} \frac{\log_+ \log_+ \p{\sup_{\b{z} \leqslant r} \b{f\p{z}}}}{\log r}
\end{equation*}
where $\log_+ x = \log \max\p{1,x}$. If $f$ is non-zero and has finite order, let $p$ be the smallest natural integer such that 
\begin{equation}\label{defp}
\sum_{m\geqslant 0} \frac{1}{\b{z_m}^{p+1}} < + \infty
\end{equation}
where $\p{z_m}_{m \geqslant 0}$ is an ordering of the zeroes of $f$ (the integer $p$ is well-defined thanks to Jensen's formula). By Hadamard's Factorization Theorem \cite[2.7.1]{Boas} there is a polynomial $Q$ such that for all $z \in \C$
\begin{equation}\label{Hadamard}
f\p{z} = e^{Q\p{z}} \prod_{m \geqslant 0} E\p{\frac{z}{z_m},p}
\end{equation}
where the function $E$ is the Weierstrass primary factor defined by
\begin{equation}\label{defE}
E\p{u,p} = \p{1-u}\exp\p{\sum_{k=1}^p \frac{1}{k}u^k} = \exp\p{-\sum_{k=p+1}^{+ \infty} \frac{1}{k}u^k}
.\end{equation}
The \emph{genus} of $f$ is then defined as $\max\p{\deg Q, p}$. We shall say that the genus of an entire function of infinite order is infinite.
\end{df}

\begin{rmq}\label{genord}
It may be deduced from Hadamard's Factorization theorem that if $o$ and $\gamma$ denote respectively the order and the genus of some entire function then $\gamma \leqslant o \leqslant \gamma + 1$ (see \cite{Boas} for details).
\end{rmq}

As explained in Remark \ref{application}, the following theorem is an abstract way to express the link between the order of the dynamical determinant and trace formulae.

\begin{thm}\label{ac}
Let $f$ be an entire function such that $f\p{0} = 1$. Let $G$ be a holomorphic function defined on a neighbourhood of $0$ such that $G\p{0}= 0$ and $f\p{z} = e^{G\p{z}}$ for $z$ in a neighbourhood of zero. Write
\begin{equation}\label{defanv}
G\p{z} = - \sum_{n=1}^{+\infty} \frac{1}{n} a_n z^n
\end{equation}
and denote by $\p{z_m}_{m \geqslant 0}$ an ordering of the zeroes of $f$. Then for all $r >0$ such that $f$ has no zero of modulus $r$, we have
\begin{equation}\label{ltf}
a_n \underset{n \to + \infty}{=} \sum_{\b{z_m} < r} \frac{1}{z_m^n} + o\p{\frac{1}{r^n}} 
.\end{equation}
Furthermore, the following properties are equivalent :
\begin{enumerate}[label=(\roman*)]
\item \label{o} the order of $f$ is finite;
\item \label{no} there is a natural integer $n_0$ such that for all integers $n \geqslant n_0 + 1$ the series
\begin{equation}\label{tr?}
\sum_{m\geqslant 0} \frac{1}{z_m^n}
\end{equation}
converges absolutely and its sum is $a_n$.
\end{enumerate}
If \ref{o} or \ref{no} holds then the minimal value of $n_0$ so that \ref{no} holds is the genus of $f$.
\end{thm}

\begin{rmq}\label{application}
Taking $f = d_{T,g}$, we have $a_n = \tf{\L_g^n}$ and it appears that global trace formula \eqref{tracefor} holds with absolute convergence of the left hand side if $n \geqslant n_0 + 1$, where $n_0$ denotes the genus of $d_{T,g}$, and local trace formula always holds according to \eqref{ltf} (hence Theorem \ref{carac}). In \S \ref{explicit}, we shall construct dynamical determinants with arbitrary (finite or infinite) genus, so that all the behaviours described in Theorem \ref{ac} may be realised by dynamical determinants.
\end{rmq}

\begin{rmq}
As we shall see in Proposition \ref{ce} below, the absoluteness of the convergence in \ref{no} is essential to get an equivalence. This is quite unfortunate especially as we shall realise the counter-examples of Proposition \ref{ce} as dynamical determinants in section \S \ref{explicit}. On the other hand, it is very easy to construct an example for which the series \eqref{tr?} converges to a sum different from $a_n$ (for any chosen values of $n$): just multiply $f$ by the exponential of an entire function. 
\end{rmq}

\begin{proof}[Proof of Theorem \ref{ac}]
\begin{itemize}
\item To prove \eqref{ltf}, one only needs to notice that the holomorphic function 
\begin{equation*}
z \mapsto \frac{f\p{z}}{\prod_{\substack{i \in \N \\ |z_i| < r}} \p{1-\frac{z}{z_i}}} = \exp\p{ - \sum_{n=1}^{+ \infty} \frac{1}{n} \p{ a_n - \sum_{\b{z_m}<r} \frac{1}{z_m^n}} z^n}
\end{equation*}
does not vanish on a disc of center $0$ and radius a little bigger than $r$, and so admits a holomorphic logarithm there.
\item Suppose \ref{o}. Recall $p$ from Definition \ref{order} and notice that the series \eqref{tr?} converges absolutely for $n \geqslant p+1$. Let $r$ be a positive real number such that $r \leqslant |z_m|$ for all $m$. Then define for $|z| \leqslant \frac{r}{2}$ and $m \geqslant 0$
\begin{equation}\label{defF}
f_m\p{z} = - \sum_{k=p+1}^{+ \infty} \frac{1}{k} \p{\frac{z}{z_m}}^k,
\end{equation}
and notice that $\b{f_m\p{z}} \leqslant \frac{r^{p+1}}{2^p} \frac{1}{\b{z_m}^{p+1}}$. Then, recalling \eqref{defp}, the series $\sum_{m \geqslant 0} f_m$ converges on the disc of center $0$ and radius $\frac{r}{2}$ to a holomorphic function $F$ and, recalling \eqref{Hadamard}, we have for $z$ close enough to $0$
\begin{equation*}
e^{G\p{z}} = f\p{z} = e^{Q\p{z} + F\p{z}}
.\end{equation*}
Thus we may identify the coefficients of order greater than $\deg Q$ in the expansions in power series of $F$ and $G$, which ends the proof of \ref{no} recalling \eqref{defanv} and \eqref{defF}.

\item Suppose \ref{no}. Using the hypothesis for $n = n_0+1$, the infinite product
\begin{equation*}
P\p{z} = \prod_{m \geqslant 0} E\p{ \frac{z}{z_m} , n_0},
\end{equation*} 
converges on $\C$ to a holomorphic function of finite order smaller than $n_0+1$ and genus $n_0$ (see \cite[Theorem 2.6.5]{Boas}). But since $a_n = \sum_{m \geqslant 0} \frac{1}{z_m^n}$ for $n \geqslant n_0+1$, we have, recalling the definition \eqref{defE} of $E$, for $z$ close enough to $0$,
\begin{equation*}
P\p{z} = \exp\p{ - \sum_{n = n_0+1}^{+ \infty} \frac{1}{n} a_n z^n}
\end{equation*}
and consequently
\begin{equation*}
f\p{z} = \exp\p{ - \sum_{n=1}^{n_0} \frac{1}{n} a_n z^n} P\p{z}
.\end{equation*}
Thus, $f$ has finite order smaller than $n_0+1$ (and genus smaller than $n_0$).
\end{itemize}
\end{proof}

We now give two counter-examples that highlight the necessity to ask for absolute convergence in \ref{no}.

\begin{prop}\label{ce}
\begin{enumerate}[label=(\alph*)]
\item \label{zd} There exists an entire function $f$ with $f\p{0}=1$ such that if $\p{z_m}_{m\geqslant0}$ is an ordering of the zeroes of $f$ (as defined in Definition \ref{ordering}) then for all $n \geqslant 1$ the series $\sum_{m \geqslant 0} \frac{1}{z_m^n}$ converges with sum $a_n$ (defined in \eqref{defanv}) but the convergence is not absolute.
\item \label{mp} There exists an entire function $f$ with $f\p{0}=1$, an ordering $\p{z_m}_{m \geqslant 0}$ of the zeroes of $f$ and a permutation $\sigma$ of $\N$ such that $\p{z_{\sigma\p{n}}}_{n \in \N}$ is an ordering of the zeroes of $f$ and, for all $n \geqslant 1$, the series $\sum_{m \geqslant 0} \frac{1}{z_m^n}$ converges with sum $a_n$, but $\sum_{m \geqslant 0} \frac{1}{z_{\sigma\p{m}}^n}$ does not converge.
\end{enumerate}
\end{prop}

\begin{rmq}
Theorem \ref{ac} implies that the functions constructed by Proposition \ref{ce} have infinite order, while global trace formulae hold in some weak sense.
\end{rmq}

To prove Proposition \ref{ce}, we shall need the following lemma, whose proof is straightforward using an Abel transform.

\begin{prop}\label{abel}
Let $\p{b_m}_{m \geqslant 0}$ be a sequence of complex numbers such that there is a constant $M$ such that for all $\ell \in \N$ we have $\left| \sum_{m=0}^{\ell} b_m \right| \leqslant M$. Let $\p{c_m}_{m \geqslant 0}$ be a decreasing sequence of positive real numbers with null limit. Then the series $\Sigma_{m \geqslant 0} b_m c_m$ converges, and we have the estimates
\begin{equation*}
\left| \sum_{m=0}^{+ \infty} b_m c_m \right| \leqslant 2 M c_0
.\end{equation*}
\end{prop}

\begin{proof}[Proof of Proposition \ref{ce}]
\begin{enumerate}[label=(\alph*)]
\item \label{a} Choose an irrational real number $\theta$ for which there is a constant $c > 0$ such that for all $n \in \N^*$ we have 
$\left|1-e^{2i \pi n \theta}\right| \geqslant \frac{c}{n^2}$ (almost any real number may be chosen thanks to Borel--Cantelli's lemma). For every integer $n$, set
\begin{equation*}
a_n = \sum_{m =2}^{+ \infty} \p{\frac{e^{2i \pi m\theta}}{\ln\p{m}}}^n,
\end{equation*}
which is well-defined thanks to Lemma \ref{abel}, but the convergence is clearly not absolute. Furthermore for all integers $m_0 \geqslant 2$ we have
\begin{equation}\label{bidoum}
\b{a_n - \sum_{m=2}^{m_0-1} \p{\frac{e^{2i \pi m\theta}}{\ln\p{m}}}^n} \leqslant \frac{4}{c} \frac{n^2}{\ln\p{m_0}^n}
\end{equation}
(take $b_m = e^{2i \pi n \p{m + m_0}}$ and $c_m = \p{\ln\p{m+m_0}}^{-n}$ in Lemma \ref{abel}). Now \eqref{bidoum} with
\begin{equation*}
\exp\p{- \sum_{n=1}^{+ \infty} \frac{1}{n} \p{\sum_{m=2}^{m_0-1} \p{\frac{e^{2i \pi m\theta}}{\ln\p{m}}}^n} z^n} = \prod_{m=2}^{m_0-1} \p{1- \frac{e^{2i \pi m\theta}}{\ln\p{m}}z}
\end{equation*}
implies that the function $f$ defined by $f\p{z} = \exp\p{- \sum_{n=1}^{+ \infty} \frac{1}{n} a_n z^n}$, for $z$ in a neighbourhood of zero, extends to an entire function whose zeroes are exactly the $\p{\frac{e^{2i \pi m\theta}}{\ln\p{m}}}^{-1}$. Since there is only one way to order the zeroes of $f$ with increasing moduli, point \ref{zd} is proven.

\item \label{b} Choose $\theta$ as in \ref{a} and denote by $\p{n_k}_{k \geqslant 0}$ the sequence of integers defined by $n_0 = 0$ and  $n_k = k ! $ for $k \geqslant 1$. Define $I_0= \set{0}$ and $I_k = [\![ n_k + 1, n_{k+1} ]\!]$ for $k \geqslant 1$. For all $n \in \N$, denote by $k\p{n}$ the unique integer such that $n \in I_{k\p{n}}$. Then set for all integers $n \geqslant 1$
\begin{equation*}
a_n = \sum_{m=0}^{+ \infty} \p{\frac{e^{2i \pi m \theta}}{\ln\p{k\p{m} + 2}}}^n
.\end{equation*}
Then, as in \ref{a}, we may use Lemma \ref{abel} to show that $f\p{z} =\exp\p{- \sum_{n=1}^{+ \infty} \frac{1}{n} a_n z^n}$  extends to an entire function whose zeroes are exactly the $z_m = \p{ \frac{e^{2i \pi m \theta}}{\ln\p{k\p{m} + 2}}}^{-1}$ for $m \in \N$. We shall see that there is another way to order the zeroes of $f$, which breaks the convergence of the series \eqref{tr?} for all $n \geqslant 1$, but preserves the monotonicity of the sequence of moduli.

Choose $0 < \epsilon < 1$ such that for all $x \in \left[0,\epsilon\right]$ we have $\Re \p{ e^{2i \pi x} } \geqslant \frac{1}{2}$. Then for all $k \in \N$ and $n \geqslant 1$, denote by $N_k^{\p{n}}$ the number of those $m \in I_k$ such that $ m n \theta \in \left[0,\epsilon\right] \p{\textup{mod}} 1$, and choose a permutation $\sigma_{k}^{\p{n}}$ of $I_k$ which puts these elements first. Equidistribution of the $mn \theta$, for $n$ fixed and $m \geqslant 0$, implies that
\begin{equation*}
\frac{\sum_{\ell=0}^k N_{\ell}^{\p{n}}}{n_{k+1}} \underset{k \to + \infty}{\to} \epsilon,
\end{equation*}
but $\sum_{\ell=0}^{k-1} N_\ell^{\p{n}} \leqslant n_{k} + 1 \underset{k \to + \infty}{=} o\p{n_{k+1}}$, and thus
\begin{equation*}
\frac{N_k^{\p{n}}}{\ln\p{k+2}}  \underset{k \to + \infty}{\to} + \infty
.\end{equation*}

Now choose $\phi : \N \to \N^*$ such that for all $n \geqslant 1$ the reciprocal image $\phi^{-1}\p{\set{n}}$ is infinite (for instance $1,1,2,1,2,3,\dots$) and set 
\begin{equation*}
\sigma = \bigcup_{k=0}^{+ \infty} \sigma_k^{\p{\phi\p{k}}}
.\end{equation*}  

Now, if $n \geqslant 1$ the series $\sum_{m \geqslant 0} \frac{1}{z_{\sigma\p{m}}^n}$ does not converge. Indeed, for all $k$ such that $\phi\p{k} = n$, we have
\begin{equation*}
\Re \p{\tilde{S}_{n_{k-1} + N_k^{\p{n}} }}  \geqslant \Re\p{ S_{n_{k-1}}} + \frac{1}{2} \frac{N_k^{\p{n}}}{\ln\p{k+2}},
\end{equation*}
where $\p{S_m}_{m \geqslant 0}$ is the sequence of partial sums of the series $\sum_{m \geqslant 0} \frac{1}{z_{m}^n}$, and $\p{\tilde{S}_m}_{m \geqslant 0}$ is the sequence of partial sums of the series $\sum_{m \geqslant 0} \frac{1}{z_{\sigma\p{m}}^n}$. We let $k$ tend to $+ \infty$ with $\phi\p{k} = n$, which is possible thanks to our choice of $\phi$. By the first paragraph of part \ref{b}, the first term of the right hand side converges but the second one tends to $+ \infty$, and thus the left hand side does not converge.
\end{enumerate}
\end{proof}

In order to realise the counter-examples of Proposition \ref{ce} as dynamical determinants in \S \ref{explicit}, we shall need the two following, merely technical, lemmas.

\begin{lm}\label{tech}
For all $\epsilon >0$ and $\rho >0$, the counter-examples of Proposition \ref{ce} may be realised as entire functions $f$ of the form $f : z \mapsto 1 - 2z -z\p{1-z} h\p{z}$, where $h$ is an entire function such that for all $z \in \C$ we have $h\p{z} = \sum_{\ell = 0}^{+ \infty} \alpha_\ell z^\ell$, with $\alpha_\ell \in \left[- \frac{\epsilon}{\rho^\ell}, \frac{\epsilon}{\rho^\ell}\right]$ for all integers $\ell$. 
\end{lm}

\begin{proof}
For all $k \geqslant 2$ and $n \geqslant 1$ set either
\begin{equation*}
a_n^{\p{k}} = \sum_{m =k}^{+ \infty} \p{\frac{e^{2i \pi m\theta}}{\ln\p{m}}}^n
\end{equation*}
or
\begin{equation*}
a_n^{\p{k}} = \sum_{m= k - 2}^{+ \infty} \p{\frac{e^{2i \pi m \theta}}{\ln\p{k\p{m} + 2}}}^n,
\end{equation*}
depending on whether you want to get a counter-example of type \ref{zd} or \ref{mp}. Then set for all $k \geqslant 1$ 
\begin{equation*}
\tilde{f}_k\p{z} = \exp\p{ - \sum_{n=1}^{+\infty} \frac{1}{n}\p{a_n^{\p{k}} + \bar{a_n^{\p{k}}}  } z^n}
.\end{equation*}
Estimate \eqref{bidoum} (and its analogue for the case \ref{mp} of Proposition \ref{ce}) implies that $\tilde{f}_k$ converges to $1$ uniformly on all compact subsets of $\C$ as $k$ goes to $+ \infty$. Then set $f_k : z \to \p{1-\frac{z}{\lambda_k}} \tilde{f}_k\p{z}$, where $\lambda_k = \frac{\tilde{f}_k\p{1}}{1 + \tilde{f}_k\p{1}}$. Thus we have $f_k\p{0} = 1$, $f_k\p{1} = -1$, and it is easy to check that $f_k$ is a counter-example of type \ref{zd} or \ref{mp}, according to the way the $a_n^{\p{k}}$ have been defined. We shall see that that for large enough $k$ the function $f_k$ satisfies the conditions of Lemma \ref{tech}. Let $h_k$ be the entire function defined by $h_k\p{z} = - \frac{f_k\p{z} -1 + 2z}{z\p{1-z}}$. We shall also need the auxiliary function $H_k\p{z} = \frac{\tilde{f}_k\p{z} - 1}{z} - \p{\tilde{f}_k\p{1} - 1}$ which vanishes at $z=1$ and tends to $0$ uniformly on all compact subsets of $\C$ when $k$ tend to $+\infty$. Then notice that
\begin{equation*}
- h_k\p{z} = \frac{H_k\p{z}}{1-z} \p{1-\frac{z}{\lambda_k}} + \frac{\tilde{f}_k\p{1}^2 - 1}{\tilde{f}_k\p{1}}
\end{equation*}
and write
\begin{equation*}
\frac{H_k\p{z}}{1-z} = \sum_{\ell=0}^{+ \infty} \beta_\ell z^\ell 
\end{equation*}
then we have
\begin{equation*}
\alpha_0 = - \p{\beta_0 + \frac{\tilde{f}_k\p{1}^2 - 1}{\tilde{f}_k\p{1}}} \textrm{ and } \alpha_{\ell+1} = -\p{\beta_{\ell+1} + \frac{\beta_\ell}{\lambda_k}} \textrm{ if } \ell \geqslant 0.
\end{equation*}
But the sequence $\p{\frac{1}{\lambda_k}}_{k \geqslant 0}$ converges to $\frac{1}{2}$ and (we may suppose $\rho \geqslant 2$)
\begin{equation*}
\b{\beta_\ell} = \frac{1}{2\pi} \b{\s{D\p{0,\rho}}{\frac{H_k\p{z}}{\p{1-z} z^{\ell+1}}}{z} } \leqslant \frac{2}{\rho^\ell} \sup_{\b{z} \leqslant r} \b{\tilde{f}_k\p{z} -1}
,\end{equation*}
which ends the proof, recalling that $\tilde{f}_k$ converges to $1$ uniformly on all compact subsets of $\C$ as $k$ goes to $+ \infty$.
\end{proof}

\begin{lm}\label{tech2}
Let $f$ be an entire function such that $f\p{0} = 1$ and $\p{c_k}_{k\geqslant 0}$ be a sequence of positive real numbers such that $\sum_{k \geqslant 0} c_k < + \infty$. Then the infinite product
\begin{equation}\label{infpro}
\prod_{k \geqslant 0} \p{z \mapsto f\p{c_k z}}
\end{equation}
converges uniformly on all compact subsets of $\C$ to an entire function $d$ that has same genus\footnote{If $f$ has non integral order $\delta$ and $\sum c_k^\delta < + \infty$, then one may show using \cite[2.9.1]{Boas} that $d$ has also same order than $f$. \label{etlordre}} than $f$. Furthermore, if $f$ is one of the counter-examples of type \ref{a} or \ref{b} constructed in Lemma \ref{tech}, then $d$ also satisfies point \ref{zd} or \ref{mp} respectively of Proposition \ref{ce}.
\end{lm}

\begin{proof}
If $K$ is a compact subset of $\C$ then, since $f\p{0} = 1$, there is a constant $C >0$ such that for all $z \in K$ we have $\b{f\p{z} -1} \leqslant C \b{z}$. Thus for all $z \in K$ and $k \geqslant 0$, we have $\b{f\p{c_k z} -1}\leqslant C \b{c_k} \b{z}$. Thus the infinite product \eqref{infpro} does converge uniformly on all compact subset of $\C$ to an entire function $d$. That $d$ has same genus than $f$ is straightforward from the Definition \ref{order} and Hadamard's factorization Theorem (we use the positivity of the $c_k$'s to ensure that no unwanted cancellation happens). Let us point out that if $f\p{z} = \exp\p{- \sum_{n=1}^{+ \infty} \frac{1}{n}a_n z^n}$ then $d\p{z} = \exp\p{- \sum_{n=1}^{+ \infty} \frac{1}{n}a_n \p{\sum_{k=0}^{+ \infty} c_k^n} z^n}$.

Suppose now that $f$ is the counter-example of type \ref{zd} constructed in Lemma \ref{tech} and denote by $\p{z_m}_{m \geqslant 0}$ an ordering of its zeroes. Let $\p{w_m}_{m \geqslant 0}$ be an ordering of the zeroes of $d$ , then there is a bijection $\p{\phi,\psi} : \N \to \N^2$ such that for all $ m \in \N$ we have $w_m = \frac{z_{\phi\p{m}}}{c_{\psi\p{m}}}$, and for all $k \in \N$ the sequence $\p{z_{\phi\p{m}}}_{m \in \psi^{-1} \p{\set{k}}}$ is an ordering of the zeroes of $f$. Let $n \geqslant 1$. It is clear from our construction that $f$ has no more than two zeroes of a given modulus\footnote{That's where we use that $f$ is precisely the counter-example constructed above. We shall not need it for the case \ref{mp}}, and so there is a constant $M$ such that for all $k \in \N$ and $m_0 \in \N$ we have
\begin{equation}\label{rusev}
\b{\sum_{ \substack{m \leqslant m_0 \\ \psi\p{m} = k}} \frac{1}{z_{\phi\p{m}}^n}} \leqslant M.
\end{equation}
Now, for all $k \in \N$, let $u_k$ be the sequence $\p{ \sum_{ \substack{m \leqslant m_0 \\ \psi\p{m} = k}} \frac{1}{w_m^n}}_{m_0 \geqslant 0}$ whose limit is $c_k^n a_n$ by construction of $f$. From \eqref{rusev}, the sup norm of $u_k$ is smaller than $c_k^n M$, and thus the series $\sum_{k \geqslant 0} u_k$ converges in the space of converging sequence equipped with the sup norm. But its sum is clearly the sequence of partial sums of $\sum_{ m \geqslant 0} \frac{1}{w_m^n}$. Thus this series converges, and its sum is $a_n \sum_{k = 0}^{+ \infty} c_k^n$, as wanted.

We suppose that $f$ is a counter-example of type \ref{mp}. There are two natural partitions of the zeroes of $d$ : the partition $Z_0,Z_1,\dots,Z_k,\dots$ by modulus ($Z_0$ contains the element of minimal modulus, the following are in $Z_1$, etc) and the partition $Z'_1,Z'_2,\dots$ defined by
\begin{equation*}
Z'_k = \set{\frac{z}{c_k} : z \textrm{ is a zero of } f}
.\end{equation*}
Both partitions are endowed with the natural notion of multiplicity. Now, we get an ordering for which the trace formulae hold in the following way : we put first the element of $Z_0 \cap Z'_0$ in the order which gave trace formulae for $f$, then we put the element of $Z_0 \cap Z'_1$ (according to the same order) ,then $Z_0 \cap Z'_2$, etc, when we are done with $Z_0$ (which happens in a finite number of steps), we do the same with $Z_1$, then $Z_2$, etc. The proof that trace formulae hold in this case is similar as in case \ref{zd} (in fact a bit easier). To get an ordering for which there is divergence of the inverse of the zeroes of $d$ at any power, we do exactly the same, except that at each step we put the elements of $Z'_0$ in the order which gave the divergence for $f$.
\end{proof}

We end this section with the two following lemmas, that shall be used to prove Corollaries \ref{tout} and \ref{asbad}.

\begin{lm}\label{tech3}
Let $E$ be a subset of $\N^*$. Then there is an entire function $Q$ such that $Q\p{0} = Q\p{1} = 0$ and if $Q : z \mapsto \sum_{n = 1}^{+ \infty} \beta_n z^n$ then $\beta_n = 0$ if and only if $n \in E$, and $\beta_n \in \R$ for all $n \in \N^*$. Moreover, for all $\epsilon >0$ and $\rho > 0$, if $\alpha > 0$ is sufficiently small, then there is an entire function $h : z \mapsto \sum_{n=0}^{+ \infty} \alpha_n z^n$ such that $\p{1-2z} e^{ \alpha Q\p{z}}= 1-2z-z\p{1-z}h\p{z}$, for all $z \in \C$, and $\alpha_n \in \left[-\frac{\epsilon}{\rho^n},\frac{\epsilon}{\rho^n}\right]$ for all $n \in \N$.
\end{lm}

\begin{proof}
We shall construct $Q$ of the form $Q : z \mapsto z \p{1-z} \sum_{n=0}^{+ \infty} b_n z^n$. Then we have $\beta_1 = b_0$ and $\beta_{n+1} = b_n - b_{n-1}$, for all $n \geqslant 1$. If $E$ contains a final segment of $\N^*$ then it is easy to see that there is a polynomial $Q$ with real coefficients that satisfies the first part of Lemma \ref{tech3}. If $E$ does not contain a final segment of $\N^*$ then the sequence $\p{b_n}_{n \in \N}$ may be recursively defined by
\begin{align*}
& b_0 = 1 \textrm{ if } 1 \notin E, 0 \textrm{ otherwise}; \\
& b_n = b_{n-1} \textrm{ if } n \geqslant 1 \textrm{ and } n + 1 \in E; \\
& b_n = \frac{1}{\min \set{ \ell \geqslant n , \ell + 2 \notin E}!} \textrm{ if } n \geqslant 1 \textrm{ and } n+1 \notin E
.\end{align*}
The second part of Lemma \ref{tech3} may be proven in a similar way than Lemma \ref{tech}.
\end{proof}

\begin{lm}\label{tech4}
Let $N_0 : \R_+^* \to \R$ be a locally bounded function. Then for all $\epsilon > 0$ and $\rho > 0$ there is an entire function $h : z \mapsto \sum_{k = 0}^{+ \infty} \alpha_k z^k$ such that for all $k \in \N$ we have $\alpha_k \in \left[-\frac{\epsilon}{\rho^k},\frac{\epsilon}{\rho^k}\right]$, and if $f : z \to 1 - 2z -z\p{1-z} h\p{z}$ and $\p{z_m}_{m \in \N}$ is an ordering of the zeroes of $z$ then
\begin{equation}
N_0\p{r} \underset{r \to 0}{=} o\p{\#\set{m \in \N : \b{z_m} < \frac{1}{16r}}}.
\end{equation}
\end{lm}

\begin{proof}
Choose a sequence $\p{z_m}_{m \in \N}$ of non-zero positive real numbers such that $z_m \underset{m \to + \infty}{\to}  +\infty$ and
\begin{equation}\label{existe}
N_0\p{r} \underset{r \to 0}{=} o\p{\# \set{m \in \N : z_m \leqslant \frac{1}{16 r}}}.
\end{equation}
Such a sequence exists since $N_0$ is locally bounded. 

Notice that if $\b{z} \leqslant \frac{1}{2}$ and $p \in \N$ then
\begin{equation*}
\b{E\p{z,p} - 1} \leqslant \frac{1}{2} \sup_{\b{w} \leqslant \frac{1}{2}} \b{E'\p{w,p}} \leqslant \frac{1}{2^p}
\end{equation*}
where $E$ is the Weierstrass primary factor from \eqref{defE}. Consequently, we can choose an increasing sequence of integer $\p{p_m}_{m \in \N}$ such that $p_m \underset{m \to + \infty}{\to} + \infty$ and the infinite product $\prod_{m \geqslant 0} E\p{\frac{z}{z_m}, p_m}$ converges uniformly on all compact subsets of $\C$. For all $m_0$ large enough define an entire function $f_{m_0}$ by
\begin{equation*}
\begin{split}
f_{m_0}\p{z}  & = \p{1 - \lambda_{m_0}z} \prod_{m \geq m_0} E\p{\frac{z}{z_m},p_m} \\
              & = 1 - 2z - z(1-z)h_{m_0}(z)
\end{split}
\end{equation*}
where
\begin{equation*}
\lambda_{m_0} = 1 + \frac{1}{\prod_{m \geq m_0} E\p{\frac{1}{z_m},p_m}}.
\end{equation*}
Using Cauchy's formula, it is easy to see that $h_{m_0}$ converges to $1$ uniformly on all compact subsets of $\C$ when $m_0 \to + \infty$. Thus $h = h_{m_0}$ satisfy the first condition when $m_0$ is large enough. Moreover, we have for all $r > 0$
\begin{equation}\label{pres}
\# \set{m \in \N : z_m \leqslant \frac{1}{16 r}} = \# \set{m \geqslant m_0 : z_m \leqslant \frac{1}{16 r}} + m_0,
\end{equation}
and thus
\begin{equation*}
N_0\p{r} \underset{r \to 0}{=} o\p{\# \set{m \geqslant m_0 : z_m \leqslant \frac{1}{16 r}} }
\end{equation*}
with \eqref{existe}, and since the right-hand side of \eqref{pres} tends to $+ \infty$ when $r$ tends to $0$. This ends the proof because the $z_m$'s are zeros of $f$.
\end{proof}

\section{Hyperbolic dynamics with explicit dynamical determinants and corollaries}\label{explicit}

In this section, we realise a wide class of entire functions as dynamical determinants. In particular, we shall materialise all the possibilities considered in Theorem \ref{ac} as well as the counter-examples of Proposition \ref{ce}. We shall also construct dynamical determinants, associated with finitely differentiable weights, which cannot be holomorphically continued to the whole complex plane. Our construction will be based on a well-known example of zeta functions for hyperbolic flows which cannot be continued meromorphically to the whole complex plane (see \cite{port} and \cite[Example 1 p.165]{PP}). The strategy is the following: we first construct a subshift of finite type and a weight for which the zeta function is explicit, then we use Whitney's extension theorem \cite{whitney} as in \cite{sabowen} to get a hyperbolic dynamics on a manifold with the same dynamical zeta function, and finally we show that in this particular case the dynamical determinant may be obtained from the dynamical zeta function.

\subsection{Symbolic dynamics with explicit weighted zeta functions}\label{symbexp}

Denote by $\p{\Sigma,\sigma}$ the full (two-sided) shift on two symbols that is
\begin{equation*}
\Sigma = \set{0,1}^{\Z} \textrm{ and } \sigma : \p{x_i}_{i \in \Z} \mapsto \p{x_{i+1}}_{i \in \Z}
.\end{equation*}

For all $\theta \in \left]0,1\right[$ define a distance on $\Sigma$ by $d_\theta \p{x,y} = \theta^{k}$ where $k = \inf\set{ i \in \N : x_i \neq y_i \textrm{ or } x_{-i} \neq y_{-i}}$ (with the convention $\theta^\infty = 0$).

Recall that if $G : \Sigma \to \C$ is a function, the weighted zeta function associated to $\p{\sigma,G}$ is the formal power series defined by
\begin{equation}\label{defzeta}
\zeta_{\sigma,G}\p{z} = \exp\p{\sum_{n=1}^{+ \infty} \frac{1}{n} \p{\sum_{\sigma^n x = x} \prod_{k=0}^{n-1} G\p{\sigma^k x}} z^n}. 
\end{equation}
Notice that $\zeta_{\sigma,1}$ is the well-known Artin-Mazur zeta function, and that the radius of convergence of $\zeta_{\sigma,G}$ is non-zero as soon as $G$ is bounded. We are going to construct weights $G$ for which $\zeta_{\sigma, G}$ is given by \eqref{exprzeta}, adapting a construction from \cite{port} and  \cite[Example 1 p.165]{PP}.

\begin{prop}\label{symbo}
Let $h$ be a holomorphic function defined on a neighbourhood of $0$ and whose expansions in power series at zero is $h\p{z} = \sum_{k = 0}^{+ \infty} \alpha_k z^k$. Denote by $\rho$ its convergence radius, and assume that for all $k \in \N$ we have $\alpha_k \neq -1$. Then there is a function $G : \Sigma \to \C$ such that
\begin{equation}\label{exprzeta}
\zeta_{\sigma,G}\p{z}^{-1} = 1 - 2 z - z\p{1-z} h\p{z}
.\end{equation}
Moreover for all $\theta \in \left]\frac{1}{\rho}, 1 \right[$, the function $G$ is Lipschitz for the distance $d_{\theta}$ and if $\alpha_k \in \left]-1,+\infty\right[$ for all $k \in \N$ then $G$ is strictly positive.
\end{prop}

\begin{proof}
Set $\beta_m = \frac{1+\alpha_m}{1+\alpha_{m-1}}$ if $m \geqslant 1$ and $\beta_0 = 1+ \alpha_0$ and define $G : \Sigma \to \C$ by
\begin{equation*}
G\p{x} = \left\{\begin{array}{cc}
   \beta_m & \textrm{ if } x_0 = \dots = x_{m-1} = 0 \textrm{ and } x_m = 1 \\
   1 & \textrm{ if } x_0 = \dots = x_i= \dots = 0,
\end{array} \right.
\end{equation*}
where $x = \p{x_i}_{i \in \Z}$. An easy computation shows that $G$ is Lipschitz for the distance $d_\theta$ provided that $\theta \in \left] \frac{1}{\rho}, 1 \right[$. For all $N >0$ define a $N+1 \times N+1$ matrix $P_N$ by
\begin{equation*}
\left\{ \begin{array}{cc}
\p{P_N}_{0,i} = \beta_i & \textrm{ if } 0 \leqslant i \leqslant N-1 \\
\p{P_N}_{0,N} = 1 & \\
\p{P_N}_{i+1,i} = \beta_i & \textrm{ if } 0 \leqslant i \leqslant N-1 \\
\p{P_N}_{N,N} = 1 & \\
\textrm{ the other entries are zero},
\end{array}\right. 
\end{equation*}
that is,
\begin{equation*}
P_N =
\left[
\begin{array}{cccccc}
\beta_0 & \beta_1 & \beta_2 & \dots & \beta_{N-1} & 1 \\
\beta_0 & 0 & & \dots & & 0 \\
0 & \beta_1 & 0 & \dots & & \dots\\
\dots & 0 & \beta_2 &  \dots & & \\
 & & &\dots & & 0 \\
 & \dots & & 0 & \beta_{N-1} & 1 
\end{array}
\right]
.\end{equation*}
Then an elementary graph-theoretic argument provides that, for all integers $k \geqslant 1$ and all $N > k$, we have
\begin{equation}\label{comptage}
\sum_{\substack{ x \in \Sigma \\ \sigma^k x =x}} \prod_{i=0}^{k-1} G\p{\sigma^i x} = \textrm{tr}\p{P_N^k}
.\end{equation}

Using an argument of dominated convergence (it is easy to show that $\b{\textrm{tr}\p{P_N^k}} \leqslant 2^k \left\|G\right\|_{\infty}^k$ by reducing to the positive case), one may then show that, for positive small enough $z$ 
\begin{equation*}
\zeta_{\sigma,G}\p{z}^{-1} = \lim_{N \to + \infty} \det\p{I - z P_N}
.\end{equation*}
A computation provides
\begin{align*}
\det\p{I - z P_N} & = \p{1-z} \p{ 1 - \sum_{k=0}^{N-1} \p{\prod_{i=0}^{k} \beta_i} z^{k+1}} - z^{N+1} \prod_{i = 0}^{N-1} \beta_i \\ & = \p{1-z} \p{1- \sum_{k=0}^{N-1} \p{1+\alpha_k} z^{k+1}} - z^{N+1} \p{1+\alpha_{N-1}}
\end{align*}
and thus
\begin{equation*}
\zeta_{\sigma,G}\p{z}^{-1} = \p{1-z} \p{1 - \sum_{k=0}^{+ \infty} \p{1+\alpha_k} z^{k+1}} = 1 - 2 z - z\p{1-z} h\p{z} 
.\end{equation*}
\end{proof}

\begin{rmq}
We could get a more general expression for \eqref{exprzeta}, for instance by allowing more than two symbols. However, we shall not need this here.
\end{rmq}

\subsection{Smooth hyperbolic dynamics with explicit dynamical determinants}\label{smoothexp}

We want now to conjugate our symbolic example to a smooth one. To do so, we use a method of Bowen \cite{sabowen} to conjugate a subshift of finite type to a piecewise affine horseshoe.

\begin{prop}\label{construction}
There is a smooth diffeomorphism $T$ of the sphere $S^4$ and a hyperbolic basic set $K$ for $T$ such that if $h$ is as in Proposition \ref{symbo} with in addition that $\rho > 1$, then there is a function $g : S^4 \to \C$ such that
\begin{equation}\label{zetexp}
\zeta_{T,g}\p{z}^{-1} = 1 - 2 z - z\p{1-z} h\p{z}
\end{equation}
and\footnote{The infinite product converges for the same reason as in Lemma \ref{tech2}.}
\begin{equation}\label{detexp}
d_{T,g}\p{z} = \prod_{k=0}^{+ \infty} \p{\zeta_{T,g}\p{\frac{z}{4^{k+2}}}^{-1}}^{\frac{\p{k+1}\p{k+2}\p{k+3}}{6}}
.\end{equation}
Moreover $g$ is $\mathcal{C}^r$ for all integers $r$ strictly smaller than $\frac{\ln \rho}{\ln 4}$, and, if $\alpha_k \in \left]-1,+\infty\right[$ for all integers $k$, then $g$ is strictly positive on $K$.
\end{prop}

\begin{proof}
Let $G : \Sigma \to \C$ be the function given by Proposition \ref{symbo}. We next recall a construction due to Bowen \cite{sabowen}, in order to check that it has some extra properties that suit us. 

Let $\p{e_i}_{0 \leqslant i \leqslant 3}$ be the standard basis in $\R^4$. Set $R\p{k} = 0$ if $k \geqslant 0$ and $R\p{k} = 1$ if $k<0$. Then, for $x = \p{x_i}_{i \in \Z} \in \Sigma$, define
\begin{equation*}
I\p{x} = \sum_{k \in \Z} 4^{-\b{k}} e_{2x_k + R\p{k}}
.\end{equation*} 
Then one easily checks that for $x,y \in \Sigma$ we have
\begin{equation}\label{equivv}
\frac{5}{6} d_{\frac{1}{4}}\p{x,y} \leqslant d\p{I\p{x},I\p{y}} \leqslant  \frac{8}{3} d_{\frac{1}{4}}\p{x,y}
,\end{equation}
where $d$ is the euclidean distance on $\R^4$. Thus $I$ induces a homeomorphism on its image $K$, which is a compact subset of $\R^4$. Define  $V_i = \set{ \p{x_0,x_1,x_2,x_3} \in \R^4 : 1 \leqslant x_{2i} \leqslant \frac{3}{2}, 0 \leqslant x_k \leqslant \frac{1}{2} \textrm{ for } k \neq 2i} $ and $F_i= I\p{\set{ x \in \Sigma : x_0 = i}}$ for $i= 0,1$. It is easy to check that $F_i$ is contained in $V_i$.

Define
\begin{equation*}
L = \left[ \begin{array}{cccc}
4 & 0 & 0 & 0 \\
0 & \frac{1}{4} & 0 & 0 \\
0 & 0 & 4 & 0 \\
0 & 0 & 0 & \frac{1}{4}
\end{array} \right]
.\end{equation*} 

For $x \in V_i$ set $G_i x = L x -4 e_{2i} + \frac{1}{4} e_{2i+1}$ (for $i=0,1$). Then define $G$ on $V_0 \cup V_1$ by $\left. G \right|_{V_i} = G_i$. One easily checks that $ G \circ I = I \circ \sigma$.  Viewing $\R^4$ as embedded in $S^4$, one may extends $G$ to a diffeomorphism $T$ of $S^4$, that coincides with $G_i$ on a neighbourhood $U_i$ of $V_i$ (see for instance \cite{Palais}). Setting $U = U_1 \cup U_2$ one has $\bigcap_{k \in \Z} T^k\p{U} = K$. Thus $K$ is a hyperbolic basic set for $T$ with isolating neighbourhood $U$.

Now define $\tilde{g}$ on $K$ by $\tilde{g} = G \circ I^{-1}$. Let $r$ be an integer strictly smaller than $\frac{\ln \rho}{\ln 4}$. Choose $\theta \in \left] \frac{1}{\rho} , 4^{-r} \right[$. Next, recalling \eqref{equivv} and that $G$ is Lipschitz for the distance $d_{\theta}$, there exists a constant $C$ such that for all $x,y \in K$ we have
\begin{equation*}
\frac{\b{\tilde{g}\p{x} - \tilde{g}\p{y}}}{d\p{x,y}^r} \leqslant C \frac{d_\theta\p{I^{-1}\p{x}, I^{-1}\p{y}}}{d_{\frac{1}{4}}\p{I^{-1}\p{x},I^{-1}\p{y}}^r} = C \p{4^r \theta}^{m\p{x,y}}  
\end{equation*} 
where $m\p{x,y}$ is the smallest integer such that $I^{-1}\p{x}$ and $I^{-1}\p{y}$ do not coincide at the position $m\p{x,y}$ or $-m\p{x,y}$. Using \eqref{equivv} again, one gets 
\begin{equation*}
m\p{x,y} = - \frac{\ln\p{d_{\frac{1}{4}}\p{I^{-1}\p{x},I^{-1}\p{y}}}}{\ln 4} \geqslant - \frac{\ln\p{\frac{6}{5}d\p{x,y}}}{\ln 4},
\end{equation*}
and thus
\begin{equation*}
\frac{\b{\tilde{g}\p{x} - \tilde{g}\p{y}}}{d\p{x,y}^r} \leqslant \tilde{C} d\p{x,y}^{-\frac{\ln\p{4^r \theta}}{\ln 4}}
.\end{equation*}
Consequently, Whitney extension's theorem \cite{whitney} ensures that $\tilde{g}$ may be extended to a $\mathcal{C}^r$ function $g$ on $S^4$. If $\rho = + \infty$, then $g$ may be chosen $\mathcal{C}^{\infty} $. Moreover, up to multiplying $g$ by a bump function, one may assume that $g$ is supported in $U$ and, if $G$ is positice, that $g$ is positive on $K$.

Since $I$ conjugates $\p{\sigma,G}$ and $\p{\left. T \right|_{K},g}$, one has
\begin{equation*}
\zeta_{T,g}\p{z}^{-1} = \zeta_{\sigma,G}\p{z}^{-1} = 1-2z - z\p{1-z}h\p{z}
.\end{equation*}

Notice that for all $n \in \N^*$ we have
\begin{equation*}
\frac{1}{\b{\det\p{I-L^n}}} = \frac{1}{16^n} \sum_{k =0}^{+\infty} \p{-1}^k \bin{k}{-4} \frac{1}{4^{nk}},
\end{equation*}
and recall that $\p{-1}^k \bin{k}{-4} = \frac{\p{k+1}\p{k+2}\p{k+3}}{6}$ is an integer. Fubini's theorem gives
\begin{align*}
d_{T,g}\p{z} & = \exp\p{- \sum_{n=1}^{+ \infty} \frac{1}{n} \sum_{\substack{x \in K \\ T^n x = x}} \frac{g^{\p{n}}\p{x}}{\b{\det\p{I-D_x T^n}}} z^n} \\
             & = \exp\p{- \sum_{n=1}^{+ \infty} \frac{1}{n} \frac{1}{\det\p{I -L^n}}\sum_{\substack{x \in K \\ T^n x = x}} g^{\p{n}}\p{x} z^n} \\
             & = \exp\p{- \sum_{n=1}^{+ \infty} \sum_{k=0}^{+\infty} \frac{\p{-1}^k \bin{k}{-4}}{n} \sum_{\substack{x \in K \\T^n x = x}} g^{\p{n}}\p{x} \p{\frac{z}{4^{k+2}}}^n} \\
             & = \prod_{k=0}^{+\infty} \p{\zeta_{T,g}\p{\frac{z}{4^{k+2}}}^{-1}}^{\frac{\p{k+1}\p{k+2}\p{k+3}}{6}}.
\end{align*}
\end{proof}

As an immediate consequence of Proposition \ref{construction} and Lemmas \ref{tech}, \ref{tech2} , \ref{tech3} and \ref{tech4}, we get the four following corollaries.

\begin{cor}\label{mechant}
The counter-examples of Proposition \ref{ce} may be produced as dynamical determinants. Namely : 
\begin{enumerate}[label=\alph*)]
\item There are a smooth diffeomorphism $T$ of $S^4$, a hyperbolic basic set $K$ for $T$, and a smooth function $g : S^4 \to \R$, strictly positive on $K$, such that for any ordering $\p{\lambda_m}_{m \geqslant 0}$ of the resonances of $\p{T,g}$ (see Definition \ref{ordering}) we have for all $n \geqslant 1$ the trace formula
\begin{equation}\label{globtf}
\tf{\L_g^n} = \sum_{\substack{ T^n x = x \\ x \in K}} \frac{g^{\p{n}}\p{x}}{\b{\det\p{I - D_x T^n}}} = \sum_{m \geqslant 0} \lambda_m^n
\end{equation}
but the convergence of the right hand side is never absolute.
\item There are a smooth diffeomorphism $T$ of $S^4$, a hyperbolic basic set $K$ for $T$, a smooth function $g : S^4 \to \R$ strictly positive on $K$, an ordering $\p{\lambda_m}_{m \geqslant 0}$ of the resonances of $\p{T,g}$, and a permutation $\sigma$ of $\N$ such that $\p{\lambda_{\sigma \p{m}}}_{m \geqslant 0}$ is an ordering of the resonances of $\p{T,g}$ and, for all $n \geqslant 1$,  the trace formula \eqref{globtf} holds but the series $\sum_{m \geqslant 0} \lambda_{\sigma \p{m}}^n$ does not converge.
\end{enumerate}
\end{cor}

Notice that for the examples of Corollary \ref{mechant} global trace formulae never hold in the sense defined in \S \ref{settings} (we required absolute convergence of the right-hand side of \eqref{tracefor}).

\begin{cor}\label{varie}
The dynamical determinant of a smooth diffeomorphism with smooth weight on a hyperbolic basic set may be of any (finite or infinite) genus\footnote{And even of any non-integral order according to footnote \ref{etlordre}.}.
\end{cor}

Recall that Theorem \ref{ac} gives a characterization of the genus of the dynamical determinant in terms of global trace formulae. Moreover from Corollary \ref{varie} and \cite[Corollary 1 p.17, second part of the book]{Groth}, we deduce that there are dynamical determinants that are not Fredholm determinants of any nuclear operators (and so there is no "good" Banach space on which the associated transfer operators are nuclear).

\begin{cor}\label{tout}
Let $E$ be a subset of $\N^*$. Then there are a smooth diffeomorphism $T$ of $S^4$, a hyperbolic basic set $K$ for $T$, and a smooth function $g : S^4 \to \R$, strictly positive on $K$, such that, for any ordering $\p{\lambda_m}_{m \geqslant 0}$ of the resonances of $\p{T,g}$, and for all $n \in \N^*$, the series 
\begin{equation*}
\sum_{m \geqslant 0} \lambda_m^n
\end{equation*}
converges absolutely and its sum is $\tf{\L_g^n}$ if and only if $n \in E$.
\end{cor}

Roughly speaking, Corollary \ref{tout} asserts that global trace formula \eqref{itr} may hold on any fixed subset of $\N^*$, this is a more precise statement than Proposition \ref{rapp1}. We can also give a more precise version of Proposition \ref{rapp2}.

\begin{cor}\label{asbad}
Let $N_0 : \R_+^* \to \N$ be a locally bounded function. Then there are a smooth diffeomorphism $T$ of $S^4$, a hyperbolic basic set $K$ for $T$, and a smooth function $g : S^4 \to \R$, strictly positive on $K$, such that if $N\p{r}$ is the number of Ruelle resonances for $\p{T,g}$ outside of the closed disc of center $0$ and radius $r$ (counted with multiplicity) we have
\begin{equation*}
N_0\p{r} \underset{r \to 0}{=} o\p{N\p{r}}
.\end{equation*}
\end{cor}

Finally, we notice that Proposition \ref{construction} can also be used to construct systems without any resonances.

\begin{cor}\label{rien}
There are a smooth diffeomorphism $T$ of $S^4$, a hyperbolic basic set $K$ for $T$, and a smooth function $g : S^4 \to \C$, such that the system $\p{T,g}$ has no resonances.
\end{cor}

In Corollary \ref{rien}, it is fundamental that $g$ takes value in $\C$: if $g$ was positive then $e^{P_{top}\p{T,\log g - J_u}}$ would be a resonance, where $P_{top}\p{T,\log g -J_u}$ is the topological pressure of $\log g-J_u$ (where $J_u$ is the unstable jacobian) with respect to the dynamics $T$ (see Remark \ref{sharpness}).

\begin{proof}
The function $h : z \mapsto \frac{e^{i\pi z} - 1 + 2z}{z\p{z-1}}$ continues holomorphically to $\C$ and may be written as $h\p{z} = \sum_{k = 0}^{+ \infty} \alpha_k z^k$ with $\alpha_k = - \sum_{l=1}^{k+1} \frac{\p{i \pi}^l}{l !} - 2$. Thus $\alpha_k \neq -1$ for every $n \in \N$ (this is a consequence of the transcendality of $\pi$). Thus applying Proposition \ref{construction}, we find $\p{T,g}$ such that
\begin{equation*}
d_{T,g}\p{z} = \prod_{k=0}^{+ \infty} \p{e^{ \frac{i \pi z}{4^{k+2}}}}^{\frac{\p{k+1}\p{k+2}\p{k+3}}{6}}
\end{equation*}
does not vanish. Thus $\p{T,g}$ has no resonances.
\end{proof}

\begin{rmq}
The weight $g$ produced by Corollaries \ref{mechant}, \ref{varie} or \ref{tout} being strictly positive on $K$, it is associated to some physically meaningful Gibbs measure $\mu_{g}$ (the one appearing in \eqref{fusil}, see \cite[Chapter 7]{Bal2} for details). For example if $g=1$ or $g = \b{\det\p{\left. DT \right|_{E^u}}} $ ( $= 16$ in our case ), $\mu_{ g}$ is respectively the physical measure or the measure of maximal entropy for $\left. T \right|_{K}$ (for the $T$ we constructed these measures coincide). It may be noticed that the weights produced by Corollaries \ref{mechant}, \ref{varie} and \ref{tout} may be chosen arbitrary close to $1$ in the $\mathcal{C}^{\infty}$ topology on a neighbourhood of $K$. The proof of this relies on the fact that, according to Lemmas \ref{tech} and \ref{tech3}, $h$ may be taken arbitrarily close to $0$ in the topology of the uniform convergence on all compact subsets of $\C$ (but, to actually prove it, an investigation of a proof of Whitney's extension theorem is needed).
\end{rmq}

\begin{rmq}\label{autres}
Proposition \ref{construction} realises a lot of entire functions as inverses of dynamical zeta functions, thus we could have stated many variations on Corollaries \ref{mechant}, \ref{varie} and \ref{tout}. For instance, one may construct a weight $g$ for which the trace formula \eqref{globtf} always holds but the convergence is absolute only when $n$ is bigger than some fixed integer (replace $\frac{1}{\p{\ln m}^n}$ in the expression of $a_n$ in the proof of Proposition \ref{ce} by $\frac{1}{m^{\alpha n}}$ for some $\alpha >0$ and then state analogues of Lemma \ref{tech} and Lemma \ref{tech2}).
\end{rmq}

\begin{rmq}\label{sharpness}
If in Proposition \ref{construction} we take $h\p{z} = h_{a,\rho}\p{z} = a \ln\p{1 + \frac{z}{\rho}}$, where $\rho >1$ and $a >0$ is small, then we get weights $g = g_{a,\rho}$, strictly positive on $K$. From formulae \eqref{zetexp} and \eqref{detexp}, we know that the radius of convergence of $d_{T,g_{a,\rho}}$ is exactly\footnote{The dynamical determinant $d_{T,g_{a,\rho}}$ cannot even be continued meromorphically outside the disc of center $0$ and radius $16 \rho$.} $\rho_{eff} = 16 \rho$. Let $r \geqslant 2$ be an integer, and choose $\rho$ such that $r < \frac{\ln \rho}{\ln 4}$, then \cite[1.5]{Tsu} predicted a radius of convergence greater than $\rho_{pred} = \exp\p{-P_{top}\p{\log g_{a,\rho} - \log 16}} 4^{r-1}$ for $d_{T,g_{a,\rho}}$. However since $g$ is strictly positive, \cite[Theorem 6.2]{Bal2} and \cite[Theorem 7.5]{Bal2} imply that $\exp\p{-P_{top}\p{\log g_{a,\rho} - \log 16}} $ is the smallest zero of $d_{T,g_{a,\rho}}$, which can be made arbitrary close to $\frac{1}{32}$ by taking $a$ close enough to $0$. On the other hand, we may chose $\rho$ arbitrary close to $4^r$. Thus, for all $\epsilon > 0$, there is a choice of $a$ and $\rho$ such that
\begin{equation*}
\frac{\rho_{eff}}{\rho_{pred}} \leqslant 2048 + \epsilon
.\end{equation*}
This means that \cite[Theorem 1.5]{Tsu} described accurately the way the radius of convergence of the dynamical determinant grows when the regularity of the weight grows (up to a bounded multiplicative constant that could be made smaller than $2048$ by playing on the parameter of the construction of Proposition~\ref{construction}).
\end{rmq}

\begin{rmq}
The Ruelle resonances of the systems constructed in Proposition \ref{construction} comes as infinite families. In particular, Proposition \ref{construction} does not allow to construct a system with a finite non-zero number of resonances. As far as we know, the only known examples of systems with finitely many resonances have either one or zero resonance.
\end{rmq}

\section{Gevrey functions and ultradistributions}\label{gevsec}

The remaining of the paper is dedicated to the study of Gevrey hyperbolic dynamics and the proof of Theorem \ref{main}. We start by recalling some basic facts from the theory of Gevrey functions and ultradistributions.

Gevrey functions have been introduced by Gevrey in his seminal paper \cite{gevorg}. Ultradistributions are classically defined as the continuous linear functionals on Gevrey functions. We will only need very few facts from this classical theory, the interested reader can for instance refer to the work of Komatsu \cite{Ko1,Ko2,Ko3}. We start with definitions on $\R^d$.

\begin{df}[Gevrey functions]\label{gfunc}
Let $d$ be a positive integer and $U$ be an open subset of $\R^d$. Let $\sigma >1$. If $f : U \to \R$ is $\mathcal{C}^\infty$ and $K$ is a compact subset of $U$, we say that $f$ is $\sigma$-Gevrey on $K$ if there are constants $C,R >0$ such that for all $\alpha \in \N^d$ we have 
\begin{equation}\label{defgev}
\sup_{x \in K} \b{\partial^\alpha f\p{x}} \leqslant C R^{\b{\alpha}} \b{\alpha}^{\sigma \b{\alpha}}.
\end{equation}
We shall say that $f$ is $\sigma$-Gevrey on $U$ if it is $\sigma$-Gevrey on all compact subsets of $U$.
\end{df}

Notice that if we take $\sigma = 1$ in this defintion, the class of functions that we obtain is the class of real-analytic function on $U$ (this is a consequence of Taylor's formula). Since we want to use the Fourier transform, it is convenient to introduce a definition of rapidly decreasing Gevrey functions.

For all $R \geqslant 1, \sigma > 1$, and $f \in \mathcal{C}^\infty \p{\R^d}$ define
\begin{equation}\label{normegev}
\n{f}_{R,\sigma} = \sup_{\substack{x \in \R^d \\ \alpha \in \N^d \\ m \in \N}} \frac{\p{1+\b{x}}^m \b{\partial^\alpha f\p{x}}}{R^{\b{\alpha}+m} \p{\b{\alpha} + m}^{\sigma \p{\b{\alpha}+m}}}.
\end{equation}
Then define
\begin{equation*}\label{grandA}
A_{R,\sigma} = \set{f \in \mathcal{C}^\infty : \n{f}_{R,\sigma} < + \infty},
\end{equation*}
which is a Banach space when endowed with the norm $\n{\cdot}_{R,\sigma}$. We can now set
\begin{equation}\label{gsig}
\G_\sigma = \bigcup_{R \geqslant 1} A_{R,\sigma},
\end{equation}
and we endow $\G_\sigma$  with the final topology of the inclusions of the $A_{R,\sigma}$ (this makes of $\G_\sigma$ a topological vector space which is presumably not locally convex). 

\begin{rmq}\label{comp}
Notice that if a function $f : \R^d \to \R$ is $\sigma$-Gevrey and compactly supported then $f \in \G_\sigma$.
\end{rmq}

We list now some basic properties of the space $\G_\sigma$ of rapidly decreasing Gevrey functions.
\begin{prop}\label{base}
Let $\sigma >1$.
\begin{enumerate}[label=(\roman*)]
\item The multiplication from $\G_\sigma \times \G_\sigma$ to $\G_\sigma$ is continuous;
\item the Fourier transform from $\G_\sigma$ to itself is a continuous isomorphism.\label{deux}
\end{enumerate}
\end{prop}

\begin{proof}
The first point is an exercice that the cautious reader would easily solve. We focus on the second one which is more crucial for our purpose.
 
We only need to prove that the Fourier transform sends $\G_\sigma$ continuously into itself, the result then follows by the Fourier inversion formula since $\G_\sigma$ is contained in $\mathcal{S}\p{\R^d}$. Let $f \in A_{R,\sigma}$. We shall use the following convention for the Fourier transform
\begin{equation*}
\F\p{f}\p{\xi} = \hat{f}\p{\xi} = \s{\R^d}{e^{-i x \xi} f\p{x}}{x}.
\end{equation*}
Recall that
\begin{equation*}
\xi^{\alpha} \partial^\beta \hat{f} \p{\xi} = \p{-i}^{\b{\alpha}+\b{\beta}} \s{\R^d}{e^{-i x \xi} x^\beta \partial^\alpha f\p{x}}{x}
\end{equation*}
and thus
\begin{equation}\label{fourier}
\b{\xi^{\alpha} \partial^\beta \hat{f} \p{\xi}} \leqslant C\p{d} \n{f}_{R,\sigma} R^{d+1+\b{\beta} + \b{\alpha}} \p{d+1+\b{\beta} + \b{\alpha}}^{\sigma\p{d+1+\b{\beta} + \b{\alpha}}}
.\end{equation}
From this it easily follows that the Fourier transform is continuous from $A_{R,\sigma}$ to $A_{R',\sigma}$ for some $R' \geqslant R$ depending on $R$ and $d$, which implies \ref{deux}.
\end{proof}

\begin{df}[Tempered ultradistributions]\label{tu}
We define $\U_\sigma$ as the space of continuous linear forms on $\G_\sigma$, endowed with the weak-star topology, this will be our space of tempered ultradistributions.
\end{df}

Beware that $\U_\sigma$ does not coincide with the space of tempered distributions defined in \cite{pilipotemp,pilipoconv}, since we defined $\G_\sigma$ as a union instead of an intersection in \eqref{gsig}. This is of no harm since ultradistributions are used here merely as a tool to get information on Ruelle resonances in a very pedestrian way. We can define multiplication $\U_\sigma \times \G_\sigma \to \U_\sigma$ and the Fourier transform $\U_\sigma \to \U_\sigma$ in the usual way, as well as the support of an ultradistribution. We shall say that a measurable function $f : \R^d \to \C$ is an ultradistribution and write $f \in \U_\sigma$ if for all $g \in \G_\sigma$ the function $fg$ is integrable and the linear form $g \mapsto \s{\R^d}{fg}{x}$ is continuous on $\G_\sigma$ (we then identify $f$ with this functional).

We can now give a definition of Fourier multiplier in our ultradistributional context. Notice that we shall always apply it to ultradistributions whose Fourier transform is locally square integrable and multipliers that are compactly supported, and we could consequently have given an \emph{ad hoc} definition that bypasses the notion of ultradistribution. However, it is costless to give now the definition in the following way.

\begin{df}[Fourier multiplier]
If $ \psi : \R^d \to \C$ is in $\G_\sigma$ we define $\psi\p{D} : \U_\sigma \to \U_\sigma$ by
\begin{equation*}
\psi\p{D} u = \F^{-1} \p{ \psi . \F u}
\end{equation*}
where $\F : \U_\sigma \to \U_\sigma$ denotes the Fourier transform.
\end{df}

The following lemma is classical, but crucial for the construction of our local space in \S \ref{locspace}, and so we provide the elementary proof. Moreover, the idea of the proof will be reused in the proof of Lemma \ref{gauw}, which is where the Gevrey assumption is needed in our proof of Theorem \ref{main}.

\begin{lm}\label{decroi}
Let $R \geqslant 1$ and $\sigma >1$. Recall that $A_{R,\sigma}$ is defined by \eqref{grandA}. There exists a constant $C > 0$ such that :
\begin{enumerate}
\item for all $f \in A_{R,\sigma}$ and $\xi \in \R^d$ we have
\begin{equation*}
\b{\hat{f}\p{\xi}} \leqslant C \n{f}_{R,\sigma} e^{-C \b{\xi}^{\frac{1}{\sigma}}} ;
\end{equation*}
\item for all $f \in A_{R,\sigma}$ and $x \in \R^d$ we have
\begin{equation*}
\b{f\p{x}} \leqslant C \n{f}_{R,\sigma} e^{-C \b{x}^{\frac{1}{\sigma}}}.
\end{equation*}
\end{enumerate}
\end{lm}

\begin{proof}
We have seen in the proof of Proposition \ref{base} that the Fourier transform sends $A_{R,\sigma}$ continuously into $A_{R',\sigma}$ for some $R'\geqslant R$. Thus, using an inverse Fourier transform, the second point follows immediately from the first, that we shall prove now.

Using the fact that
\begin{equation*}
\b{\xi}^m \leqslant \p{\sum_{i=1}^d \b{\xi_i}}^m = \sum_{\substack{\alpha \in \N^d \\ \b{\alpha} = m}} c_\alpha \b{\xi^\alpha} \textrm{ with } \sum_{\substack{\alpha \in \N^d \\ \b{\alpha} = m}} c_\alpha = d^m
\end{equation*}
and recalling \eqref{fourier} we get
\begin{equation}\label{maj}
\b{\xi}^m \b{\hat{f}\p{\xi}} \leqslant C\p{d,R} \n{f}_{R,\sigma} K^m m^{\sigma m}
\end{equation}
for some constants $K,C\p{d,R} \geqslant 1$ and all $m \in \N$. Now taking $m = \max\p{ \left\lfloor \frac{1}{e} \frac{\b{\xi}^{\frac{1}{\sigma}}}{K^{\frac{1}{\sigma}}} \right\rfloor , 0}$ we get the result.
\end{proof}

If $U$ is an open subset of $\R^d$ and $f :  U \to \R^N$ is a function, for some integer $N$, we say that $f$ is $\sigma$-Gevrey if its components are $\sigma$-Gevrey. With this definition the class of $\sigma$-Gevrey is closed under composition (a proof was already present in Gevrey's original paper \cite{gevorg}) and inversion. Moreover, there are $\sigma$-Gevrey partitions of unity (see for instance \cite{nqa} where this topic is dealt within the context of non-quasianalytic Denjoy--Carleman classes). There is even a version of Whitney's extension theorem for Gevrey functions \cite{whitnqa}, which could be a way to prove for instance that points \ref{better} and \ref{butter} in Theorem \ref{machinw} are sharp using a method similar to the one of \S \ref{explicit}.
\begin{rmq}\label{ose}
If $\theta$ is a $\sigma$-Gevrey function compactly supported in $\R^d$ and $\kappa : \R^d \to \R^d$ is a $\sigma$-Gevrey diffeomorphism, then it comes from the proof of the closure of the class of $\sigma$-Gevrey functions in \cite{gevorg} that the map $ f \to \theta . f \circ \kappa$ is continuous from $\G_\sigma$ to itself. Thus, if $u \in \U_\sigma$ we may define $\theta . u \circ \kappa$ by the formula
\begin{equation*}
\langle \theta . u \circ \kappa , f \rangle = \langle u , \theta \circ \kappa^{-1} f\circ \kappa^{-1} \b{\det D\kappa^{-1}} \rangle
.\end{equation*} 
Furthermore, the map $u \mapsto \theta . u \circ \kappa$ is continuous from $\U_\sigma$ to $\U_\sigma$. This result extends to the case of a local diffeomorphism $\kappa : U \to \R^d$ such that $\kappa^{-1}\p{\textup{supp } \theta}$ is a compact subset of $U$. Indeed, we can locally extend $\kappa$ to a global diffeomorphism (see \eqref{exten} for instance) and then use a partition of unity. Consequently, this result will also extend to the case of ultradistributions on manifolds as soon as we have a proper definition.
\end{rmq}

We end this section with the definitions of Gevrey functions and ultradistributions on manifolds.

\begin{df}[Gevrey manifold]\label{gevman}
Let $\sigma> 1$. We shall call $\sigma$-Gevrey manifold a $\mathcal{C}^{\infty}$ manifold $M$ endowed with a maximal atlas $\mathcal{A}$ of charts such that the changes of charts are $\sigma$-Gevrey.
\end{df}

\begin{ex}\label{exa}
A real-analytic manifold has a natural structure of $\sigma$-Gevrey manifold for all $\sigma >1$. If $\sigma \leqslant \sigma'$, any $\sigma$-Gevrey manifold has a natural structure of $\sigma'$-Gevrey manifold. If $M$ is a $\sigma$-Gevrey manifold then all the usual bundles over $M$ have a natural structure of $\sigma$-Gevrey manifold (which make projections and trivialisations $\sigma$-Gevrey).
\end{ex}

\begin{df}[Gevrey maps]\label{gevmap}
Let $M$ and $N$ be $\sigma$-Gevrey manifolds. Denote by $\mathcal{A}$ and $\mathcal{A}'$ their respective $\sigma$-Gevrey maximal atlases. A $\mathcal{C}^\infty$ map $f : M \to N$ is said to be $\sigma$-Gevrey if for all $\p{\psi,U} \in \mathcal{A}$ and $\p{\phi,V} \in \mathcal{A}'$ the map
\begin{equation*}
\phi \circ f \circ \psi^{-1} : \psi\p{f^{-1}\p{V}} \to \R^d
\end{equation*}
is $\sigma$-Gevrey. We say that $f$ is Gevrey if it is $\sigma$-Gevrey for some $\sigma >1$.
\end{df}

It is clear that with these definitions, the closure of the class of Gevrey maps under composition, inversion and multiplication (for maps valued in $\R$ or $\C$) can be extended from the case of open subsets of $\R^d$ to the case of Gevrey manifolds. For the same reason, there are Gevrey partitions of unity on Gevrey manifolds. Anyway, we shall always use charts and the main purpose of Definitions \ref{gevman} and \ref{gevmap} is to be able to state Theorem \ref{main} in a compact and natural way. 

\begin{df}[Gevrey functions and ultradistributions on manifolds]\label{gfum}

Let $\sigma > 1$. If $M$ is a $\sigma$-Gevrey manifold, we denote by $\G_\sigma\p{M}$ the space of Gevrey functions from $M$ to $\C$. If $V \subseteq M$ we denote by $\mathcal{G}_\sigma\p{V}$ the space of Gevrey functions on $M$ supported in $V$. If in addition $V$ has compact closure, we endow $\G_{\sigma}\p{V}$ with a topology as in the euclidean case by covering it with a finite number of domain of $\sigma$-Gevrey charts.

A $\sigma$-Gevrey density on $M$ is a measure absolutely continuous with respect to Lebesgue on $M$ whose density in $\sigma$-Gevrey charts is $\sigma$-Gevrey we can endow this space with a topology in the same way than $\G_\sigma\p{M}$ (provided $M$ is compact, in fact the choice of a $\sigma$-Gevrey volume identifies these two spaces).

If $M$ is compact, we denote by $\U_\sigma\p{M}$ the space of linear functionals on the $\sigma$-Gevrey densities on $M$, and we endow it with the weak-star topology. If $V$ is a subset of $M$ then $\U_\sigma\p{V}$ denotes the elements of $\U_\sigma\p{M}$ supported in $V$ (the definition of the support is the same as for distributions, if $M$ is non-compact but the closure of $V$ may be covered by a finite numbers of domain of $\sigma$-Gevrey charts, we can still define $\U_\sigma\p{V}$). Notice that when $V$ is closed in $M$ then $\U_\sigma\p{V}$ is closed in $\U_\sigma\p{M}$.
\end{df}

\section{Local space of anisotropic ultradistributions}\label{locspace}

This subsection is dedicated to the definition of the local space $\h_{\Theta,\alpha,\bar{t}}$ and the introduction of Paley--Littlewood tools. It is possible to give a simpler definition of the local spaces $\h_{\Theta,\alpha,\bar{t}}$, using the notion of polarization from \cite{Tsu} instead of the notion of generalized polarization stated below. However, this simpler construction imposes to add in Theorem \ref{main} the requirement that $\lambda$ from Definition \ref{bashyp} may be chosen stricly greater than $\sqrt{2}$. If $C$ and $C'$ are closed cones in an euclidean space we write $C \Subset C'$ for $\overline{C} \subset \bul{C'} \cup \set{0}$.

\begin{df}[Generalized polarization]\label{gp}
Let $r \geqslant 2$ be an integer. A generalized polarization with $r$ cones is a family $\Theta = \p{C_i,\phi_i}_{0\leqslant i \leqslant r}$ where
\begin{enumerate}[label=(\roman*)]
\item $C_0 = \R^d$ and $C_1,\dots,C_r$ are closed $d_u$-dimensional cones in $\R^d$ (for some fixed $d_u >0$);
\item if $i \in \set{0,\dots,r-1}$ then $C_{i+1} \Subset C_i$;
\item for all $i \in \set{0,\dots,r}$ the function $\phi_i : S^{d-1} \to \left[0,1\right]$ is $\sigma$-Gevrey, supported in the interior of $C_i \cap S^{d-1}$ and, if in addition $i \leqslant r- 2$, then $\phi_i$ vanishes on a neighborhood of $C_{i+2} \cap S^{d-1}$;
\item for all $x \in S^{d-1}$ we have $\sum_{i=1}^r \phi_i\p{x} = 1$.
\end{enumerate}
\end{df}

The notion of polarization is borrowed and adapted from \cite{Bal2,Tsu}. Notice that with this definition we may have $\phi_r =0$, but this wouldn't make the proofs much easier. Choose a Gevrey function $\chi : \R \to \left[0,1\right]$ such that $\chi\p{x} = 1 $ if $x \leqslant \frac{1}{2}$ and $\chi\p{x} = 0$ if $x \geqslant 1$. Fix $\alpha >1$ (that will have to be chosen large enough, we need $\alpha > \sigma$ in Proposition \ref{loc} and $\alpha > \sigma + 1$ in \S\ref{lto}) and then define for all $n \geqslant 1 $ and $\xi \in \R^d$, $\chi_n\p{\xi} = \chi\p{\b{\xi}-n^\alpha}$, set also $\chi_n = 0$ if $n \leqslant 0$. Then set for $n \in \N$, $\psi_n\p{\xi} = \chi_{n+1}\p{\xi} - \chi_n\p{\xi}$. Thus we have
\begin{equation}\label{bande}
\textup{supp } \psi_n \subseteq \set{ \xi \in \R^d : n^\alpha \leqslant \b{\xi} \leqslant \p{n+1}^\alpha + 1 } \textrm{ for } n \geqslant 1
\end{equation}
and $\textup{supp } \psi_0 \subseteq \set{ \xi \in \R^d : \b{\xi} \leqslant 2}$.
Moreover, we have $\sum_{n \geqslant 0} \psi_n = 1 $. Set
\begin{equation*}
\Gamma = \N \times \set{0,\dots,r}.
\end{equation*}
Define then for $\p{n,i} \in \Gamma$ the function $\psi_{\Theta,n,i}$ by
\begin{equation*}
\psi_{\Theta,n,i}\p{\xi} = \left\{
\begin{array}{c}
\psi_n\p{\xi} \phi_i\p{\frac{\xi}{\b{\xi}}} \textrm{ if } n \geqslant 1 \\
\frac{\psi_0\p{\xi}}{r+1} \textrm{ if } n = 0,
\end{array}
\right.
\end{equation*}
so that we have
\begin{equation}\label{unite}
\sum_{\p{n,i} \in \Gamma} \psi_{\Theta,n,i} = 1.
\end{equation}

The space that we shall construct is of Sobolev type as for instance in chapter 4 of \cite{Bal2}. Choose $\bar{t} = \p{t_0,\dots,t_r} \in \R^{r+1}$ and define $w = w_{\Theta,\alpha, \bar{t}}$ on $\R^d$ by
\begin{equation*}
w\p{\xi} = w_{\Theta,\alpha,\bar{t}}\p{\xi} = \psi_0\p{\xi} + \p{1 - \psi_0\p{\xi}} \sum_{i=0}^r \phi_i\p{\frac{\xi}{\b{\xi}}} e^{t_i \b{\xi}^{\frac{1}{\alpha}}}.
\end{equation*}

\begin{df}[Local space $\h_{\Theta,\alpha,\bar{t}}$]
Let $\tau$ be any real number strictly between $1$ and $\alpha$ (we will see in Proposition \ref{loc} that the choice is inessential). Set (recalling Definition \ref{tu})
\begin{equation}\label{espgevult}
\h_{\Theta,\alpha,\bar{t}} = \set{u \in \U_\tau : \hat{u} \in L^2_{loc} \textrm{ and } \s{\R^d}{\b{\hat{u}\p{\xi}}^2 w_{\Theta,\alpha,\bar{t}}\p{\xi}^2}{\xi} < + \infty},
\end{equation}
endowed with the hermitian product
\begin{equation*}
\langle u , v \rangle_{\Theta,\alpha,\bar{t}} = \s{\R^d}{\bar{\hat{u}\p{\xi}} \hat{v}\p{\xi} w_{\Theta,\alpha,\bar{t}}\p{\xi}^2}{\xi}.
\end{equation*}
\end{df}

\begin{prop}\label{loc}
$\h_{\Theta,\alpha,\bar{t}}$ is a separable Hilbert space that does not depend on the choice of $\tau$ . For all $1 < \sigma < \alpha$, the space $\G_\sigma$ is continuously contained and dense in $\h_{\Theta,\alpha,\bar{t}}$, and $\h_{\Theta,\alpha,\bar{t}}$ is continuously contained in $\U_\sigma$. 
\end{prop}

\begin{proof}
The continuous inclusion of $\G_\sigma$ in $\h_{\Theta,\alpha,\bar{t}}$ for $1 < \sigma < \alpha$ immediately follows from Lemma \ref{decroi}. If $f \in A_{R,\sigma}$ with $\sigma < \tau$ and $R \geqslant 1$ and $u \in \h_{\Theta,\alpha,\bar{t}}$ we have
\begin{align}
\b{\langle u , \bar{f} \rangle} & = \b{\s{\R^d}{\hat{u}\p{\xi} \bar{\hat{f}}\p{\xi}}{\xi}} \leqslant \s{\R^d}{ \b{\hat{u}\p{\xi} w_{\Theta,\alpha,\bar{t}}\p{\xi}} \b{\frac{\hat{f}\p{\xi}}{w_{\Theta,\alpha,\bar{t}}\p{\xi}}}}{\xi} \nonumber \\
     & \leqslant \n{u}_{\Theta,\alpha,\bar{t}} \n{\frac{\hat{f}}{w_{\Theta,\alpha,\bar{t}}}}_2 \leqslant C \n{u}_{\Theta,\alpha,\bar{t}} \n{f}_{R,\sigma} \label{inclu}
\end{align}
where the last line follows from Lemma \ref{decroi} and $\sigma < r < \alpha$. Thus $\h_{\Theta,\alpha,\bar{t}}$ is continuously included in $\U_\sigma$.

We shall now prove that $\G_\sigma$ is dense in $\h_{\Theta,\alpha,\bar{t}}$. Let $u \in \h_{\Theta,\alpha,\bar{t}}$ be in the orthogonal space to $\G_\sigma$. Choose $\phi \in \G_\sigma$ compactly supported. We have then for all $f \in \G_\sigma$
\begin{equation*}
\s{\R^d}{\hat{u}\p{x}\phi\p{x} w_{\Theta,\alpha,\bar{t}}\p{x}^2 f\p{x}}{x} = 0.
\end{equation*}
In particular, the convolution of $\phi \hat{u} w_{\Theta,\alpha,\bar{t}}^2$ with any element of $\G_\sigma$ is null. Noticing that $\phi \hat{u} w_{\Theta,\alpha,\bar{t}}^2$ is in $L^1$, it vanishes almost everywhere, and thus $\hat{u} = 0$ and then $u=0$.

Let $\p{u_n}_{n \in \N}$ be a Cauchy sequence in $\h_{\Theta,\alpha,\bar{t}}$. Then $\p{\hat{u}_n. w_{\Theta,\alpha,\bar{t}}}_{n \in \N}$ is a Cauchy sequence in $L^2\p{\R^d}$ and thus has a limit $v.w_{\Theta,\alpha,\bar{t}}$. Reasoning in the same way as for \eqref{inclu}, we see that for all $R \geqslant 1$ there is a constant $C$ such that for all $f \in A_{R,r}$ we have
\begin{equation*}
\b{\langle v , f \rangle} \leqslant C \n{v . w_{\Theta,\alpha,\bar{t}}}_2 \n{f}_{R,r}.
\end{equation*}
Thus $v \in \U_\tau$, we may consequently define $u = \F^{-1} v \in \U_r$, and $u$ belongs to $\h_{\Theta,\alpha,\bar{t}}$ since its Fourier transform is $v$ and $v . w_{\Theta,\alpha,\bar{t}} \in L^2\p{\R^d}$. Finally, the convergence of $\p{\hat{u}_n. w_{\Theta,\alpha,\bar{t}}}_{n \in \N}$ to $\hat{u}. w_{\Theta,\alpha,\bar{t}}$ in $L^2\p{\R^d}$ readily implies the convergence of $\p{u_n}_{n \in \N}$ to $u$ in $\h_{\Theta,\alpha,\bar{t}}$.

To prove that $\h_{\Theta,\alpha,\bar{t}}$ does not depend on $\tau$, notice that if $1<\tilde{\tau} <\tau$ then the definition \eqref{espgevult} replacing $\tau$ by $\tilde{\tau}$ yields a bigger space $\widetilde{\h}_{\Theta,\alpha,\bar{t}}$. But $\h_{\Theta,\alpha,\bar{t}}$ is dense in $\widetilde{\h}_{\Theta,\alpha,\bar{t}}$ (it contains $\G_\sigma$ for $1 < \sigma < \tilde{\tau}$) and it is closed as well since it is a Hilbert space (the inclusion is obviously isometric). Thus $\h_{\Theta,\alpha,\bar{t}}$ does not depend on $\tau$ (in particular \eqref{inclu} holds for all $1 < \sigma < \alpha$). Finally, $\h_{\Theta,\alpha,\bar{t}}$ is separable since it is isomorphic to $L^2\p{\R^d}$.
\end{proof}

\begin{rmq}\label{densité}
Notice that the proof of Lemma \ref{loc} in fact proves that the set of elements of $\G_\sigma$ whose Fourier transform is compactly supported is dense in $\h_{\Theta,\alpha,\bar{t}}$.
\end{rmq}

\begin{rmq}
The parameter $\bar{t}$ will have to be chosen wisely with respect to the dynamical system. Indeed, we want $\h_{\Theta,\alpha,\bar{t}}$ to be a Hilbert space of anistropic ultradistributions regular in the stable direction (outside of $C_1$) and dual of regular in the unstable direction (inside $C_r$). It seems then natural to require $t_0 > 0$ and $t_r < 0$. Furthermore, we want that the linearized dynamics sends areas of high regularity into areas of lower regularity, which reads $t_0 > t_1 > \dots > t_r$ with the definition of generalized cone-hyperbolicity \ref{ch} given in \S\ref{lto}. However, some technical issues (mostly due to the ultradistributional context) will be dealt with by requiring somme additional properties of $\bar{t}$ in \S\ref{lto} (namely \eqref{sueur}).
\end{rmq}

We shall now give a Paley--Littlewood type description of our local space $\h_{\Theta,\alpha,\bar{t}}$, which will turn out to be very handy to study of the transfer operator in \S \ref{lto}. The use of Paley--Littlewood decompositions to study transfer operators is not new, see \cite{Tsu,Bal2}. They have also been used in the context of Gevrey regularity, see for instance \cite{PLgevdet}. However, we are not aware of any references in which the decomposition in thinner bands than usual proposed in \eqref{bande} is used. The main idea is that we want the weight $w_{\Theta,\alpha,\bar{t}}$ to be roughly constant on the portion of annulus that correspond to our decomposition, and since it grows very fast we cannot use the usual dyadic decomposition. Notice that this imposes to work on spaces of square integrable functions (as opposed to general $L^p$ spaces), indeed we have a uniform bound on the operator norm of the Fourier multipliers $\psi_{\Theta,n,i}\p{D}$ when acting on $L^2$. Moreover, it is convenient to work with Hilbert spaces, since it allows us to use the Lidskii trace theorem (however, this is not necessary, as explained in Remark \ref{ord0}). 

\begin{prop}\label{descrip}
Let $1 < \sigma < \alpha$. Then $ u \in \U_\sigma$ belongs to $\h_{\Theta,\alpha,\bar{t}}$ if and only if 
\begin{equation}\label{equiv}
\sum_{\p{n,i} \in \Gamma} \p{ e^{t_i n} \n{\psi_{\Theta,n,i}\p{D}u}_2}^2 < + \infty.
\end{equation}
Moreover, the square root of this quantity defines an equivalent (Hilbertian) norm on $\h_{\Theta,\alpha,\bar{t}}$.
\end{prop}

\begin{proof}
First of all, an elementary asymptotical development shows that there is a constant $C$ such that for all $\p{n,i} \in \Gamma$ and all $\xi \in \textup{supp } \psi_n$ we have
\begin{equation}\label{fine}
\frac{1}{C} e^{t_i n} \leqslant e^{t_i \b{\xi}^{\frac{1}{\alpha}}} \leqslant C e^{t_i n}
.\end{equation}
Then write
\begin{equation*}
w_{\Theta,\alpha,\bar{t}}\p{\xi} = \psi_0\p{\xi}\p{1 + \p{1- \psi_0\p{\xi}}\p{\sum_{i=0}^r e^{t_i \b{\xi}^{\frac{1}{\alpha}}} }} + \p{1- \psi_0\p{\xi}}\sum_{\substack{\p{n,i} \in \Gamma \\ n \neq 0}} \psi_{\Theta,n,i}\p{\xi}e^{t_i \b{\xi}^{\frac{1}{\alpha}}},
\end{equation*}
and use \eqref{fine} and the fact that the intersection number of the support of the $\psi_{\Theta,n,i}$ is finite to show that there is another constant $C$ such that for all $\xi \in \R^d$ we have
\begin{equation}\label{equivfour}
\frac{1}{C}w_{\Theta,\alpha,\bar{t}}\p{\xi}^2 \leqslant \sum_{\p{n,i} \in \Gamma} e^{t_i 2n} \psi_{\Theta,n,i}\p{\xi}^2 \leqslant C w_{\Theta,\alpha,\bar{t}}\p{\xi}^2.
\end{equation}

Now, if $u \in \U_\sigma$ is such that \eqref{equiv} holds then $\hat{u} \in L^2_{loc}$ (using Plancherel's formula and recalling that the intersection number of the support of the $\psi_{\Theta,n,i}$ is finite) and using \eqref{equivfour} and Plancherel's formula we get that $u \in \h_{\Theta,\alpha,\bar{t}}$ with the required estimates. The other implication is easier.
\end{proof}

We shall now define an auxiliary separable Hilbert space which will be useful in the investigation of the transfer operator. Set
\begin{equation}\label{spaceB}
\B = \set{\p{u_{n,i}}_{\p{n,i} \in \Gamma} \in \prod_{\p{n,i} \in \Gamma} L^2\p{\R^d} : \sum_{\p{n,i} \in \Gamma} \p{e^{n t_i} \n{u_{n,i}}_2}^2 < + \infty}
\end{equation}
endowed with the natural Hilbertian structure. Define the map
\begin{equation*}
\begin{array}{ccccc}
\mathcal{Q}_{\Theta} & : & \h_{\Theta,\alpha,\bar{t}} & \to & \B \\
 & & u & \mapsto & \p{\psi_{\Theta,n,i}\p{D}u}_{\p{n,i} \in \Gamma}
\end{array}.
\end{equation*}
Proposition \ref{descrip} implies that $\mathcal{Q}_\Theta$ is bounded and that its image is a closed subspace of $\B$ isomorphic to $\h_{\Theta,\alpha,\bar{t}}$. For all $\p{n,i} \in \Gamma$ define also the natural projection and inclusion
\begin{equation*}
\begin{array}{ccccc}
\pi_{n,i} & : & \B & \to &L^2\p{\R^d} \\
 & & \p{u_{\ell,j}}_{ {\ell,j} \in \Gamma} & \mapsto & u_{n,i}
\end{array}
\textrm{ and }
\begin{array}{ccccc}
\iota_{n,i} & : & L^2\p{\R^d} & \to & \B \\
 & & u & \mapsto & \p{\delta_{\p{n,i} = \p{\ell,j}} u}_{\p{\ell,j} \in \Gamma}
\end{array}.
\end{equation*}

\section{Local transfer operator}\label{lto}

In this subsection, we shall investigate the properties of "local" transfer operators acting on spaces of the type $\h_{\Theta,\alpha,\bar{t}}$ from section \S \ref{locspace}. This analysis is inspired from \cite{Tsu} with some major modifications to take advantage of the Gevrey regularity. Choose a second generalized polarization with $r$ cones $\Theta' = \p{C_i',\phi_i'}_{0 \leqslant i \leqslant r}$ and a $\sigma$-Gevrey diffeomorphism $\mathcal{T} : \R^d \to \R^d$ (with $ 1 < \sigma < \alpha - 1$) and we assume that $\mathcal{T}$ is \emph{generalized cone-hyperbolic} from $\Theta'$ to $\Theta$, that is the following conditions are fulfilled :
\begin{df}[Generalized cone-hyperbolicity]
\ 
\begin{enumerate}[label=(\roman*)]\label{ch}
\item for all $x \in \R^d$ and $i \in \set{1,\dots,r}$ we have
\begin{equation}\label{fuite}
{}^t D_x \mathcal{T} \p{C_i} \subseteq C'_{\min\p{i+2,r}};
\end{equation}
\item there is $\Lambda > 1$ such that for all $x \in \R^d$ and $\xi \in C_{r-1}$ we have
\begin{equation}\label{stronghypw}
\b{{}^t D_x \mathcal{T}\p{\xi}} \geqslant \Lambda \b{\xi};
\end{equation}
\item for the same $\Lambda > 1$, for all $x \in \R^d$ and $\xi \in \R^d$ such that ${}^t D_x \mathcal{T}\p{\xi} \notin C'_2$ we have
\begin{equation}\label{stronghypw2}
\b{{}^t D_x \mathcal{T}\p{\xi}} \leqslant \Lambda^{-1} \b{\xi}.
\end{equation}
\end{enumerate}
\end{df}

\begin{rmq}
This definition is adapted from the definition of cone-hyperbolicity given in \cite{Tsu}. Notice that in this definition the unstable dimensions $d_u$ and $d_u'$, from Definition \ref{gp}, of the cones $\Theta$ and $\Theta'$ respectively must satisfy $d_u' \geqslant d_u$. There is no reason \textit{a priori} for the equality to hold. However, in \S\ref{preuve}, we will construct a family of generalized polarizations with the same unstable dimension (it will naturally be the dimension of $E_x^u$ from Definition \ref{bashyp}, which does not depend on $x \in K$, since we require the transitivity of $\left. T \right|_K$). Nevertheless, since it looks like we do not need $d'_u = d_u$ here, we could probably work with cones of different dimensions in \S\ref{preuve} and thus remove the assumption of transitivity on $\left. T\right|_K$. However, since we need $g$ to be supported on a neighbourhood of $K$ in Theorem \ref{main}, we do not get global information on the dynamics in the absence of transitivity. For instance, if $T$ is a north-south dynamics, an orbit going from a neighbourhood of the north to a neighbourhood of the south will have to enter a region where $g$ vanishes and thus the information cannot propagate between different hyperbolic basic pieces. Although, it may be interesting to know that the transitivity hypothesis is not necessary in the Anosov case (that is when $K = M$). A study of the transfer operator for Morse-Smale (non-transitive) gradient flows may be found in the work of Dang and Rivière (see \cite{DR}).
\end{rmq}

The figure above illustrates the notion of generalized cone-hyperbolicity in dimension $2$. We have here $\Theta = \Theta'$ and $r = 4$. The cones $C_1,C_2,C_3$ and $C_4$ are delimited by dashed lines. The interiors of the cones contain the solid line marked "unstable direction" which would be in the application the image of the unstable direction (for some point in $K$) by the differential of a chart. Similarly, the other solid line would be the image of the stable direction for the same point. The black arrows depict schematically the action of ${}^t D_x \mathcal{T}$ for some $x \in \R^d$.

\definecolor{qqqqcc}{rgb}{0,0,0.8}
\definecolor{ffqqff}{rgb}{1,0,1}
\definecolor{qqccqq}{rgb}{0,0.8,0}
\definecolor{ffqqqq}{rgb}{1,0,0}
\begin{figure}
\begin{tikzpicture}[line cap=round,line join=round,>=triangle 45,x=0.5018141837478447cm,y=0.7146800799736683cm]
\clip(-10.61,-6.87) rectangle (19.28,7.12);
\draw [line width=1.2pt,dash pattern=on 2pt off 2pt,color=ffqqqq,domain=-10.61:19.28] plot(\x,{(-0.83--1.74*\x)/3.62});
\draw [line width=1.2pt,dash pattern=on 2pt off 2pt,color=ffqqqq,domain=-10.61:19.28] plot(\x,{(--6.37-1.52*\x)/4.08});
\draw [line width=1.2pt,dash pattern=on 4pt off 4pt,color=qqccqq,domain=-10.61:19.28] plot(\x,{(--2.04--0.9*\x)/5.04});
\draw [line width=1.2pt,dash pattern=on 4pt off 4pt,color=qqccqq,domain=-10.61:19.28] plot(\x,{(--3.75-0.28*\x)/4.06});
\draw [line width=1.2pt,dash pattern=on 1pt off 2pt on 4pt off 4pt,color=ffqqff,domain=-10.61:19.28] plot(\x,{(-4.91--3.54*\x)/3.24});
\draw [line width=1.2pt,dash pattern=on 1pt off 2pt on 4pt off 4pt,color=ffqqff,domain=-10.61:19.28] plot(\x,{(--6.26-1.94*\x)/2.8});
\draw [line width=1.2pt,dotted,color=qqqqcc,domain=-10.61:19.28] plot(\x,{(-3.93--2.16*\x)/0.78});
\draw [line width=1.2pt,dotted,color=qqqqcc,domain=-10.61:19.28] plot(\x,{(--3.96-1.5*\x)/1.04});
\draw [line width=2pt,domain=-10.61:19.28] plot(\x,{(--3.85--0.3*\x)/5.74});
\draw [line width=2pt,domain=-10.61:19.28] plot(\x,{(-9.09--4.1*\x)/-0.62});
\draw [line width=2.8pt] (1.82,6.36)-- (2.1,4.06);
\draw [line width=2.8pt] (1.82,3.82)-- (2.2,3.34);
\draw [line width=2.8pt] (2.2,3.34)-- (2.5,3.9);
\draw [line width=2.8pt] (8.28,0.72)-- (10.22,0.72);
\draw [line width=2.8pt] (10.22,1.1)-- (10.74,0.78);
\draw [line width=2.8pt] (10.74,0.78)-- (10.28,0.42);
\draw [shift={(4.25,1.35)},line width=2.8pt]  plot[domain=0.34:1.64,variable=\t]({1*2.51*cos(\t r)+0*2.51*sin(\t r)},{0*2.51*cos(\t r)+1*2.51*sin(\t r)});
\draw [line width=2.8pt] (6.14,2.02)-- (6.8,1.66);
\draw [line width=2.8pt] (6.8,1.66)-- (7.08,2.3);
\draw (4.23,6.32) node[anchor=north west] {$C_1$};
\draw (6.46,5.24) node[anchor=north west] {$C_2$};
\draw (8.43,3.68) node[anchor=north west] {$C_3$};
\draw (9.79,2) node[anchor=north west] {$C_4$};
\draw (1.45,7.1) node[anchor=north west] {Stable direction};
\draw (11.14,1.01) node[anchor=north west] {Unstable direction};
\end{tikzpicture}
\caption{Generalized cone-hyperbolicity.}
\end{figure}
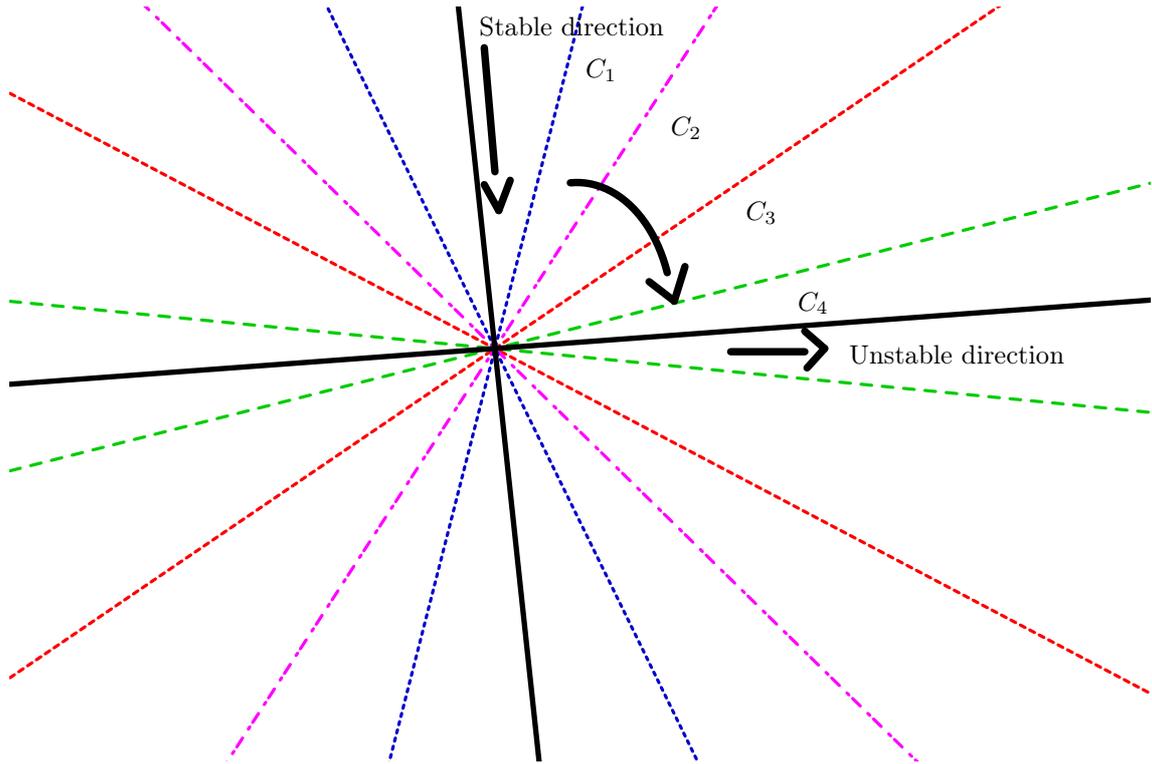

Let $G : \R^d \to \C$ be a compactly supported $\sigma$-Gevrey function. The "local" transfer operator that we are going to study is
\begin{equation*}
L : u \mapsto G. \p{u \circ \mathcal{T}},
\end{equation*}
which is well-defined as an operator from $\G_\sigma$ to itself. Fix $1 < \nu < \Lambda^{\frac{1}{\alpha}}$ and $a > 0$ such that for all $x \in \textup{supp } G$
\begin{equation}\label{defa}
a < \n{{}^t D_x \mathcal{T}^{-1}}^{-\frac{1}{\alpha}}.
\end{equation}
We choose $\bar{t} = \p{t_0,\dots,t_r} \in \R^{r+1}$ as in \S \ref{locspace}, with the additional assumptions
\begin{equation}\label{sueur}
t_0 > 0 > t_1 > \dots > t_r, \quad t_r > \nu t_{r-1} \textrm{ and } t_{i+1} < \frac{2}{a} t_i \textrm{ if } i \leqslant r-2.
\end{equation}

The aim of this section is to establish the two following propositions.

\begin{prop}\label{locrep}
The transfer operator $L$ extends to a bounded operator $\L : \h_{\Theta,\alpha,\bar{t}} \to \h_{\Theta',\alpha,\bar{t}}$. Moreover, $\L$ is nuclear. More precisely, it may be written as
\begin{equation}\label{repre}
\L = \sum_{m = 0}^{+ \infty} \lambda_m e_m \otimes l_m
\end{equation}
where the $e_m$ and $l_m$ have unit norm respectively in the Hilbert spaces $\h_{\Theta',\alpha,\bar{t}}$ and $\h_{\Theta,\alpha,\bar{t}}'$ and the $\lambda_m \in \C$ satisfy
\begin{equation}\label{reprequant}
\b{\lambda_m} \leqslant C \theta^{m^{\frac{1}{\beta}}}
\end{equation}
for all $m \in \N$, $\beta = 2 + \alpha d$ and some constants $C>0$ and $0 < \theta <1$.
\end{prop}

\begin{prop}\label{trace}
If $\Theta = \Theta'$ then 
\begin{equation}\label{fortra}
\textup{tr}\p{\L} = \sum_{x \in \R^d : \mathcal{T}x = x} \frac{G\p{x}}{\b{\det\p{I - D_x\mathcal{T} }}}
\end{equation}
\end{prop}

To carry out the proofs of Propositions \ref{locrep} and \ref{trace}, we will need $\sigma$-Gevrey functions $\tilde{\phi}_0,\dots,\tilde{\phi}_r$ such that
\begin{itemize}
\item for all $i \in \set{0,\dots,r}$ the function $\tilde{\phi}_i$ is supported in the interior of $C_i \cap S^{d-1}$ and, if in addition $i \leqslant r-2$, it vanishes on a neighborhood of $C_{i+2} \cap S^{d-1}$;
\item for all $i \in \set{0,\dots,r}$ and $x \in S^{d-1}$, if $\phi_i\p{x} \neq 0$ then $\tilde{\phi}_i\p{x} = 1$.
\end{itemize} 
Define then $\tilde{\psi}_n = \chi_{n+2} - \chi_{n-1}$ for $n \geqslant 0$, and if $\p{n,i} \in \Gamma$ set
\begin{equation*}
\tilde{\psi}_{\Theta,n,i}\p{\xi} = \left\{
\begin{array}{c}
\tilde{\psi}_n\p{\xi} \tilde{\phi}_i\p{\xi} \textrm{ if } n \geqslant 1 \\
\tilde{\psi}_0\p{\xi} \textrm{ if } n = 0.
\end{array}
\right. 
\end{equation*}
In this way $\psi_{\Theta,n,i} \p{\xi} \neq 0$ implies $\tilde{\psi}_{\Theta,n,i} \p{\xi} = 1$.

Now if $\p{n,i},\p{\ell,j} \in \Gamma$ define an operator $S_{n,i}^{\ell,j} : L^2\p{\R^d} \to L^2\p{\R^d}$ by
\begin{equation}\label{decoupew}
S_{n,i}^{\ell,j} = \psi_{\Theta',n,i}\p{D} \circ L \circ \tilde{\psi}_{\Theta,\ell,j}\p{D}.
\end{equation}
We shall see in Lemma \ref{nucmieuxw} below that the sum
\begin{equation}\label{sommew}
\sum_{\p{n,i},\p{\ell,j} \in \Gamma} \iota_{n,i} \circ S_{n,i}^{\ell,j} \circ \pi_{\ell,j}
\end{equation}
converges in trace class operator topology to an auxiliary operator $\M : \B \to \B$ closely related to $L$, where $\B$ was defined in \eqref{spaceB}.

For this, we will prove a sequence of lemmas (\ref{drw},\ref{gauw},\ref{distsw}, \ref{rcpq}, and \ref{nucmieuxw}) and we need more notation.

Without loss of generality, one may suppose\footnote{This hypothesis may be removed by adding suitable renormalization constants in the Fourier coefficients below. However, notice that it is clear from the proof of Lemma \ref{enfin} that we need no such generality to perform the proof of Theorem \ref{main}.} that $G$ is supported in $\left]-\pi,\pi\right[^d$. Then choose a $\sigma \textrm{-}$Gevrey function $\rho : \R^d \to \left[0,1\right]$ supported in $\left]-\pi,\pi\right[^d$ and such that $\rho \p{x} = 1$ if $x \in \textup{supp } G$. A $L^2$ function $u$ supported in $\left]-\pi,\pi\right[^d$ may be identified with a $2\pi \Z^d$-periodic function, and then for all $k \in \Z^d$ we can define its $k$th Fourier coefficient
\begin{equation*}
c_k\p{u} = \frac{1}{\p{2\pi}^d}\s{\left]-\pi,\pi\right[^d}{e^{-ikx}u\p{x}}{x}.
\end{equation*}
For all $k \in \Z^d$ also define $\rho_k : x \mapsto \rho\p{x} e^{ikx}$. Then if $\p{n,i},\p{\ell,j} \in \Gamma$ we can write
\begin{equation}\label{somme2w}
S_{n,i}^{\ell,j} = \sum_{k \in \Z^d} \p{ \psi_{\Theta',n,i}\p{D} \rho_k} \otimes \p{c_k \circ L \circ \tilde{\psi}_{\Theta,\ell,j}\p{D}}.
\end{equation}
This identity is merely formal right now, but we shall see that this sum actually converges in the nuclear operator topology and give explicit bound on its terms.

\begin{lm}\label{drw}
There is a constant $C >0$ such that for all $k \in \Z^d$ and $\p{n,i} \in \Gamma$ we have
\begin{equation}\label{dreqw}
\n{\psi_{\Theta',n,i}\p{D} \rho_k}_2 \leqslant C \p{n+1}^{\frac{\alpha d}{2}} \exp\p{- C^{-1} d\p{k,\textup{ supp } \psi_{\Theta',n,i}}^{\frac{1}{\sigma}}}.
\end{equation}
\end{lm}

\begin{proof}
The Fourier transform of $\psi_{\Theta',n,i}\p{D} \rho_k$ is $\psi_{\Theta',n,i}. \hat{\rho}\p{.-k}$. Thus the result follows from Lemma \ref{decroi}.
\end{proof}

\begin{rmq}\label{dr2w}
Up to enlarging $C$, the estimate \eqref{dreqw} may be improved to
\begin{equation*}
\n{\psi_{\Theta',n,i}\p{D} \rho_k}_2 \leqslant C \p{n+1}^{\frac{\alpha d}{2}} \exp\p{- C^{-1} d\p{k,\textup{ supp } \psi_{\Theta',n,s}}^{\frac{1}{\sigma}}} \exp\p{-C^{-1} \b{k}^{\frac{1}{\sigma}}}
\end{equation*}
provided that $\b{k} \geqslant C n^\alpha$.
\end{rmq}

For all $k \in \Z^d$ and $\p{\ell,j} \in \Gamma$ define
\begin{equation*}
\delta\p{k,\ell,j} = \inf_{x \in \textup{supp } G} d\p{k,{}^t D_x \mathcal{T}\p{\textup{supp } \tilde{\psi}_{\Theta,\ell,j}}}
.\end{equation*}

\begin{lm}\label{gauw}
For every $\epsilon > 0$ there is a constant $C >0$ such that, for all $k \in \Z^d$ and $\p{\ell,j} \in \Gamma$, the linear form $c_k \circ L \circ \psi_{\Theta,\ell,j}\p{D}$ belongs to the dual of $L^2\p{\R^d}$ with norm bounded by 
\begin{equation*}
C \p{\ell+1}^{\frac{\alpha d}{2}} \exp\p{-C^{-1} \delta\p{k,\ell,j}^{\frac{1}{\sigma+1}}} \textrm{ if } \delta\p{k,\ell,j} > \epsilon \ell^\alpha
\end{equation*}
and $C \p{\ell+1}^{\frac{\alpha d}{2}}$ otherwise.
\end{lm}

\begin{proof}
For all $u \in L^2\p{\R^d}$ one may write
\begin{equation*}
c_k \circ L \circ \psi_{\Theta,\ell,j}\p{D}\p{u} = \frac{1}{\p{2\pi}^d} \s{\left]-\pi,\pi\right[^d \times \R^d}{e^{i\p{\mathcal{T}\p{x}\eta - kx}} G\p{x} \tilde{\psi}_{\Theta,\ell,j}\p{\eta} \hat{u}\p{\eta}}{x}\mathrm{d}\eta.
\end{equation*}
For all $p \geqslant 0$ this may be rewritten as
\begin{equation*}
c_k \circ L \circ \psi_{\Theta,\ell,j}\p{D}\p{u} = \frac{1}{\p{2\pi}^d} \s{\left]-\pi,\pi\right[^d \times \R^d}{e^{i\p{\mathcal{T}\p{x}\eta - kx}} F_p\p{x,k,\eta} \tilde{\psi}_{\Theta,\ell,j}\p{\eta} \hat{u}\p{\eta}}{x}\mathrm{d}\eta
\end{equation*}
where $F_p$ is the sum of at most $\p{5d}^p p!$ terms of the form
\begin{equation}\label{formatw}
\p{x,k,\eta} \mapsto \pm \frac{\partial^a G\p{x}}{\p{\Phi\p{x}\p{k,\eta}}^{p+m}} \partial^{b_1} l_{j_1}\p{x}\p{k,\eta} \dots \partial^{b_p} l_{j_p}\p{x}\p{k,\eta} \partial^{\gamma_1} \Phi\p{x}\p{k,\eta} \dots \partial^{\gamma_m} \Phi\p{x}\p{k,\eta}
\end{equation}
where $m \leqslant p$ and $\b{a} + \b{b_1} + \dots + \b{b_p}  + \b{\gamma_1} + \dots + \b{\gamma_m} = p$, and $l_j\p{x}\p{k,\eta} = i\p{\partial_j \mathcal{T}\p{x} \eta - k_j}$ and $\Phi\p{x}\p{k,\eta} = \b{{}^t D_x \mathcal{T}\p{ \eta} - k}^2$. This may be proved by induction on $p$ : set $F_0\p{x,k,\eta} = G\p{x}$ and then
\begin{equation*}
F_{p+1}\p{x,.,.} = \sum_{j=1}^d \partial_j \p{\frac{l_j\p{x}F_p\p{x,.,.}}{\Phi\p{x}}},
\end{equation*}
using the formula
\begin{equation*}
\s{\R^d}{e^{i f\p{y}} g\p{y}}{y} = i \s{\R^d}{e^{i f\p{y}} \sum_{j=1}^d \partial_j \p{\frac{\partial_j f\p{y} g\p{y}}{\b{\nabla f\p{y}}^2}}}{y}.
\end{equation*}
Thus, each term of generation $p$ gives rise to at most $5d\p{p+1}$ terms of generation $p+1$.

Now, assume that $\delta\p{k,\ell,j} > \epsilon \ell^\alpha$ and notice that $\left|k\right| \leq \delta\p{k,\ell,j} + C\p{\ell+1}^\alpha$ for some $C > 0$, thus for every quadratic form $\Psi : \R^d \times \R^d \to \R$ and every linear form $l : \R^d \times \R^d \to \C$ we have
\begin{equation*}
\b{\frac{\Psi\p{k,\eta}}{\Phi\p{x}\p{k,\eta}}} \leqslant C \n{\Psi} \textrm{ and } \b{\frac{l\p{k,\eta}}{\Phi\p{x}\p{k,\eta}}} \leqslant C \n{l} d\p{k,{}^t D_x \mathcal{T} \p{\textup{supp } \tilde{\psi}_{\Theta,\ell,j}}}^{-1}
\end{equation*}
for any $x \in \R^d$ and $\eta \in \textup{supp } \tilde{\psi}_{\Theta,\ell,j}$, any choice of norms on the spaces of linear forms and quadratic forms on $\R^d \times \R^d$ and where $C$ only depends on $\epsilon$ and this choice of norms. Thus we see that each term of the form \eqref{formatw} may be bounded by
\begin{equation*}
C^{2p} \n{\partial^a G}_{\infty} \n{\partial^{b_1} l_{j_1}}_{\infty} \dots \n{\partial^{b_p} l_{j_p}}_{\infty} \n{\partial^{\gamma1} \Phi}_{\infty} \dots \n{\partial^{\gamma_m} \Phi}_{\infty} \delta\p{k,\ell,j}^{-p}.
\end{equation*}
However, the map $\Phi$ and $l_1,\dots,l_d$ are $\sigma$-Gevrey. Consequently, $F_p$ is bounded by 
\begin{equation*}
C M^p p^{\p{\sigma+1}p} \delta\p{k,l,j}^{-p}
\end{equation*}
for the relevant values of $k$ and $\eta$. This implies that
\begin{equation*}
\b{c_k \circ L \circ \tilde{\psi}_{\Theta,\ell,j}\p{D}\p{u}} \leqslant C M^p p^{\p{\sigma+1}p} \b{\textup{supp } G} \delta\p{k,\ell,j}^{-p} \p{\ell+1}^{\frac{\alpha d}{2}} \n{u}_2.
\end{equation*}
Then take $p=\max\p{0,\left\lfloor \frac{\delta\p{k,\ell,j}^{\frac{1}{\sigma+1}}}{e M^{\frac{1}{\sigma+1}}}\right\rfloor}$ to get the first estimate, the second one is easy.
\end{proof}

Notice that Lemmas \ref{drw} and \ref{gauw} (see also Remark \ref{dr2w}) imply that the right-hand side of \eqref{somme2w} converges in nuclear operator topology. The equality holds as we can see by testing it on $\mathcal{C}^\infty$ compactly supported functions. However, Lemmas \ref{drw} and \ref{gauw} alone are not enough to prove Proposition \ref{locrep}, we also need to use the generalized cone-hyperbolicity of $\mathcal{T}$.

Recall that $\nu$ has been chosen such that $1 < \nu < \Lambda^{\frac{1}{\alpha}}$, where $\Lambda$ is from \eqref{stronghypw}. Define then a relation $\hra$ on $\Gamma$ in the following way :
\begin{itemize}
\item if $i=j=0$, then $\p{\ell,j} \hra \p{n,i}$ if and only if $\ell \geqslant \nu n$;
\item if $i,j \in \set{r-1,r}$, then $\p{\ell,j} \hra \p{n,i}$ if and only if $n \geqslant \nu \ell$;
\item  if $i \geqslant j+1$ and $j \leqslant r-2$, then $\p{\ell,j} \hra \p{n,i}$ if and only if $n \geqslant \frac{a}{2} \ell$ where $a$ is from \eqref{defa}.
\end{itemize}
This relation indexes terms in \eqref{sommew} which correspond to high differentiability transitioning to low differentiability. Let us recall an elementary fact before stating and proving Lemma \ref{distsw} which is the main tool to use generalized cone-hyperbolicity of $\mathcal{T}$ in order to deal with zones of low regularity leaking into zones of high regularity.

\begin{lm}\label{distcone}
Let $C_+$ and $C_-$ be two transverse cones in $\R^d$ (that is we have $C_+ \cap C_- = \set{0}$) and write
\begin{equation*}
\mu = \min \p{d\p{C_+ \cap S^{d-1},C_-},d\p{C_- \cap S^{d-1},C_+}} > 0.
\end{equation*}
Then if $\xi \in C_+$ and $\eta \in C_-$ we have $d\p{\xi,\eta} \geqslant \mu \max\p{\b{\xi},\b{\eta}}$.
\end{lm}

\begin{proof}
If $\xi = \eta = 0$ the result is trivial. Consequently, we may suppose without loss of generality that $\b{\xi} = \max\p{\b{\xi},\b{\eta}} > 0$. Then $d\p{\frac{\xi}{\b{\xi}}, \frac{\eta}{\b{\xi}}} \geqslant \mu$, which gives the announced estimates, multiplying by $\b{\xi}$.
\end{proof}

\begin{lm}\label{distsw}
There are a constant $c>0$ and an integer $N$ such that for all $\p{n,i},\p{\ell,j} \in \Gamma$ we have : $\p{\ell,j} \hra \p{n,i}$ or $\max\p{n,\ell} \leqslant N$ or 
\begin{equation}\label{distsuppw}
d\p{\textup{supp } \psi_{\Theta',n,i}, {}^tD_x \mathcal{T} \p{\textup{supp } \tilde{\psi}_{\Theta,\ell,j}} } \geqslant c \max\p{n,\ell}^\alpha
\end{equation}
for all $x \in \textup{supp } G$.
\end{lm}

\begin{proof}
The proof is based on an investigation of figure 1. Let us distinguish the different values of $i$ and $j$.

We deal first with the case $i,j \in \set{r-1,r}$. If $\p{\ell,j} \nhra \p{n,i}$ and $\max\p{n,\ell} > N$ then we have $n < \nu \ell$. If $\xi \in \textup{supp } \psi_{\Theta',n,i}$ and $\eta \in \textup{supp } \tilde{\psi}_{\Theta,\ell,j}$ then $\eta \in C_{r-1}$ and $\ell \neq 0$ (otherwise $\max\p{n,\ell} = 0 \leqslant N$) and thus, for all $x \in \textup{supp } G$, we have $\b{{}^t D_x \mathcal{T}\p{\eta}} \geqslant \Lambda \b{\eta} \geqslant \Lambda \ell^{\alpha}$. Consequently, the distance between $\xi$ and ${}^t D_x \mathcal{T}\p{\eta}$ is larger than
\begin{equation*}
\Lambda \p{\ell-1}^\alpha - \p{\p{n+1}^\alpha + 1} \geqslant \Lambda \p{\ell-1}^\alpha - \p{\p{\nu \ell + 1}^\alpha + 1} \underset{ \ell \to + \infty}{\sim} \p{\Lambda - \nu^\alpha} \ell^\alpha.
\end{equation*}
Notice that $\max\p{n,l} \leqslant \nu \ell$ and recall that $\nu$ has been chosen so that $\Lambda - \nu^\alpha > 0$ to conclude this first case.

Let us now deal with the case $\p{i,j} = \p{0,0}$ and assume that $\max\p{n,\ell} > N$ and $\p{\ell,j} \nhra \p{n,i}$, that is we have $\ell < \nu n$. If $\xi \in \textup{supp } \psi_{\Theta',n,i}$ and $\eta \in \textup{supp } \tilde{\psi}_{\Theta,\ell,j}$ then there are two possibilities. The first one is that ${}^t D_x \mathcal{T} \p{\eta}$ belongs to $C'_2$ and may be dealt using Lemma \ref{distcone}. Indeed, since $\phi'_0$ vanishes on a neighbourhood of $C'_2 \cap S^{d-1}$, the set $\textup{supp } \psi_{\Theta',n,0} $ is contained in a closed cone transverse to $C'_2$ (and independant of $n \geqslant 1$). The second possibility is that ${}^t D_ x\mathcal{T} \p{\eta}$ does not belong to $C'_2$. Then \eqref{stronghypw2} implies that
\begin{equation*}
\b{{}^t D_x \mathcal{T} \p{\eta}} \leqslant \frac{1}{\Lambda} \b{\eta} \leqslant \frac{\p{\ell+3}^\alpha+1}{\Lambda} \leqslant \frac{\p{\nu n+3}^\alpha + 1}{\Lambda}
\end{equation*} 
and consequently the distance between $\xi$ and ${}^t D_x \mathcal{T} \p{\eta}$ is larger than
\begin{equation*}
n^\alpha - \frac{\p{\nu n+3}^\alpha + 1}{\Lambda} \underset{ n \to + \infty}{\sim} \p{1 - \frac{\nu^\alpha}{\Lambda}} n^{\alpha}.
\end{equation*}
Then, since $\frac{\nu^\alpha}{\Lambda} <1$, we get the announced minoration, provided $N$ is large enough (recall that $\ell < \nu n$ and thus we only need to consider the case of large $n$, since $\ell$ is small when $n$ is small).

We deal now with the case $i \geqslant j+1$. If $j \geqslant r-1$, then we have $i,j \in \set{r-1,r}$ and this case has already been dealt with. Thus we assume that $ j \leqslant r-2$, $\max\p{n,\ell} > N$ and $\p{\ell,j} \nhra \p{n,i}$. Consequently we have $n < \frac{a}{2} \ell$. Now, if $\xi \in \textup{supp } \psi_{\Theta',n,i}$ and $\eta \in \textup{supp } \tilde{\psi}_{\Theta,\ell,j}$ the definition of $a$ implies that the distance between $\xi$ and ${}^t D_x \mathcal{T}\p{\eta}$ is greater than
\begin{equation*}
\b{{}^t D_x \mathcal{T}\p{\eta}} - \b{\xi} \geqslant a^\alpha \b{\eta} - \p{n+1}^\alpha - 1 \geqslant a^{\alpha} \ell^\alpha - \p{\frac{a}{2} \ell + 1}^\alpha - 1 \underset{\ell \to + \infty}{\sim} a^{\alpha}\p{1 - \frac{1}{2^\alpha}} \ell^\alpha.
\end{equation*}

We are left with the case $i < j+1$. If $i \leqslant r-1$ then we must have $i,j \in \set{r-1,r}$ and this case has already been dealt with. Thus, let us assume that $i < r-1$ and $\max\p{n,\ell} > N$ (we cannot have $\p{\ell,j} \hra \p{n,i}$ in this case). Consequently, for all $x \in \textup{supp } G$, the set ${}^t D_x \mathcal{T}\p{\textup{supp } \tilde{\psi}_{\Theta,\ell,j}}$ is contained in $C_{j+2}$ by \eqref{fuite}. However, since $i \leqslant r-2$, the function $\phi_i'$ vanishes on a neighborhood of $C_{j+2} \subseteq C_{i+2}$ and consequently $\textup{supp } \psi_{\Theta',n,i}$ is contained in a closed cone transverse to $C_{j+2}$ that does not depend of $n \geqslant 1$ (the case $n=0$ is easily dealt with separately by taking $N$ large enough). Applying Lemma \ref{distcone}, this ends the proof of the lemma.
\end{proof}

We need a last lemma to construct our auxiliary operator $\M : \B \to \B$, where $\B$ was defined in \eqref{spaceB}.

\begin{lm}\label{rcpq}
Let $\p{a_m}_{m\in \N}$ be a decreasing sequence of positive real numbers such that for some constant $M >0$,$\beta > 0$ and all $\epsilon >0$ sufficiently small we have
\begin{equation}\label{estime}
\# \set{m : a_m \geqslant \epsilon} \leqslant M \ln\p{\frac{1}{\epsilon}}^{\beta}.
\end{equation}
Then there are constants $C>0$ and $0 < \theta < 1$ such that for all $m \in \N$ we have
\begin{equation*}
a_m \leqslant C \theta^{m^\frac{1}{\beta}}.
\end{equation*}
\end{lm}

\begin{proof}
Up to enlarging $C$ and $\theta$ at the end, we may remove a finite number of terms in $\p{a_m}_{m\in \N}$ and thus assume that \eqref{estime} holds for all $\epsilon < 1$ and that $a_m < \frac{1}{2}$ for all $m \in \N$. Then set $C=1$ and $\theta = \exp\p{- \frac{1}{M^{\frac{1}{\beta}}}}$ and notice that if $a_m > \theta^{m^{\frac{1}{\beta}}}$ then
\begin{equation*}
m + 1 \leqslant \# \set{k : a_k \geqslant \theta^{m^{\frac{1}{\beta}}}} \leqslant m,
\end{equation*}
which is absurd.
\end{proof}

We next exploit the facts we have gathered to construct $\M$ and to prove the key Lemma \ref{nucmieuxw} needed to show Propositions \ref{locrep} and \ref{trace}.

\begin{lm}\label{nucmieuxw}
The sum
\begin{equation*}
\sum_{\p{n,s},\p{\ell,t} \in \Gamma} i_{n,s} \circ S_{n,s}^{\ell,t} \circ \pi_{\ell,t} = \sum_{\substack{\p{n,s},\p{\ell,t} \in \Gamma \\ k \in \Z^d} } \p{i_{n,s}\p{\psi_{\Theta',n,s}\p{D} \rho_k}} \otimes \p{c_k \circ L \circ \widetilde{\psi}_{\Theta,\ell,t}\p{D} \circ \pi_{\ell,t}}
\end{equation*}
converges in nuclear operator topology to a trace class operator $\M$ that may be written as
\begin{equation*}
\M = \sum_{m \in \N} \lambda_m e_m \otimes l_m,
\end{equation*}
where the $l_m$ and $e_m$ have unit norm respectively in $\B'$ and $\B$, and $\lambda_m \in \C$, with $\b{\lambda_m} \leqslant C \theta^{m^{\frac{1}{\beta}}}$, where $\beta = 2 +\alpha d$ and for some constants $C>0$ and $0< \theta < 1$.
\end{lm}

\begin{proof}
We shall see that there is a constant $M>0$ such that for all $\epsilon > 0$ small enough $\# A_\epsilon \leqslant M \ln\p{\frac{1}{\epsilon}}^{\beta}$ where
\begin{equation*}
A_\epsilon = \set{ \p{k,\p{\ell,j},\p{n,i}}\in \N \times \Gamma \times \Gamma : \n{c_k \circ L \circ \tilde{\psi}_{\Theta,\ell,j}\p{D}\circ \pi_{\ell,j}}_{\B'} \n{\iota_{n,i}\p{\psi_{\Theta',n,i}\p{D}\rho_k}}_{\B} \geqslant \epsilon} 
\end{equation*}
and then use Lemma \ref{rcpq}. To do so we will cover this set with three smaller sets: we deal separately with the $\p{k,\p{\ell,j},\p{n,i}}$ with $\p{\ell,j} \hra \p{n,i}$ or $\max\p{n,\ell} \leqslant N$ or $\p{\ell,j} \nhra \p{n,i}$ and $\max\p{n,\ell} > N$.

\begin{itemize}
\item \textit{First case}: $\p{\ell,j} \hra \p{n,i}$. As in the proof of Lemma \ref{distsw}, we distinguish subcases that corresponds to different values of $i$ and $j$. We deal first with the case $i,j \in \set{r-1,r}$. Using Lemmas \ref{drw} and \ref{gauw}, we have
\begin{equation}\label{petw}
\n{c_k \circ L \circ \tilde{\psi}_{\Theta,\ell,j}\p{D}\circ \pi_{\ell,j}}_{\B'} \n{\iota_{n,i}\p{\psi_{\Theta',n,i}\p{D}\rho_k}}_{\B} \leqslant C e^{t_i n} \p{n+1}^{\frac{\alpha d}{2}} e^{-t_j \ell} \p{\ell+1}^{\frac{\alpha d}{2}}
\end{equation}
for some constant $C > 0$ and if in addition $\b{k} \geqslant C n^{\alpha}$ then \eqref{petw} can be improved to (see Remark \ref{dr2w})
\begin{equation}\label{grdw}
\begin{split}
& \n{c_k \circ L \circ \tilde{\psi}_{\Theta,\ell,j}\p{D}\circ \pi_{\ell,j}}_{\B'} \n{\iota_{n,i}\p{\psi_{\Theta',n,i}\p{D}\rho_k}}_{\B} \\ & \qquad \qquad \qquad \leqslant C e^{t_i n} \p{n+1}^{\frac{\alpha d}{2}} e^{-t_j \ell} \p{\ell+1}^{\frac{\alpha d}{2}} \exp\p{-C^{-1} \b{k}^{\frac{1}{\sigma}}}.
\end{split}
\end{equation}
Choose $\tilde{t}_i > t_i$ and $\tilde{t}_j < t_j$ such that\footnote{The notation is a bit misleading since we can have $i=j$ but in this case we will never have $\tilde{t}_i = \tilde{t}_j$.} $\tilde{t}_i - \frac{\tilde{t}_j}{\nu} < 0 $ (where $\nu >1$ is from the definition of $\hra$). This is possible since $\nu > 1$, and $t_{r-1} > t_r > \nu t_{r-1}$ by assumption. Up to enlarging $C$, we can replace $e^{t_i n}\p{n+1}^{\frac{\alpha d}{2}}$ by $e^{\tilde{t}_i n}$ and $e^{- t_j \ell}\p{\ell+1}^{\frac{\alpha d}{2}}$ by $e^{- \tilde{t}_j \ell}$ in \eqref{petw} and \eqref{grdw}. Consequently, if $\p{k,\p{\ell,j},\p{n,i}}$ belongs to $A_\epsilon$ we have
\begin{equation*}
\ln C + n \p{\tilde{t}_i - \frac{\tilde{t}_j}{\nu}} \geqslant \ln C + n \tilde{t}_i - \ell \tilde{t}_j \geqslant \ln \epsilon
\end{equation*}
and thus
\begin{equation*}
n \leqslant \p{\frac{\tilde{t}_j}{\nu} - \tilde{t}_i}^{-1} \p{\ln\p{\frac{1}{\epsilon}} + \ln C}.
\end{equation*}
If $\b{k} \geqslant C \max\p{n,l}^\alpha$ we must in addition have
\begin{equation*}
\ln C - C^{-1} \b{k}^{\frac{1}{\sigma}} \geqslant \ln \epsilon
\end{equation*}
and thus
\begin{equation}\label{lesk}
\b{k} \leqslant C\p{\ln\p{\frac{1}{\epsilon}} + \ln C}^\sigma
.\end{equation}
Consequently, recalling that $n \geqslant \nu l$, there are at most $M \ln\p{\frac{1}{\epsilon}}^{2 + \alpha d} $ elements of $A_\epsilon$ such that $\p{\ell,j} \hra \p{n,i}$ and $i,j \in \set{r-1,r} $, for some constant $M$ and $\epsilon$ small enough (notice that $d \geqslant 2$). 

The case $i=j=0$ is similar, interverting $n$ and $\ell$.

We are left with the case $i \geqslant j+1$ and $j \leqslant r-2$. In this case $\p{\ell,j} \hra \p{n,i}$ implies that $n \geqslant \frac{a}{2} \ell$. Choose $\tilde{t}_i > t_i$ and $\tilde{t}_j < t_j$ such that $\tilde{t}_i - \frac{2}{a} \tilde{t}_j < 0$ (this is possible since $t_i \leqslant t_{j+1} < \frac{2}{a} t_j$ by assumption). Then, up to enlarging $C$ we may replace $e^{t_i n}\p{n+1}^{\frac{\alpha d}{2}}$ by $e^{\tilde{t}_i n}$ and $e^{- t_j \ell}\p{\ell+1}^{\frac{\alpha d}{2}}$ by $e^{- \tilde{t}_j \ell}$ in \eqref{petw} and \eqref{grdw}. And then, if $\p{k,\p{\ell,j},\p{n,i}}$ belongs to $A_\epsilon$ we have as before
\begin{equation*}
n \leqslant \frac{ \p{\ln\p{\frac{1}{\epsilon}} + \ln C}}{\tau}
\end{equation*}
where $\tau = \frac{2}{a} \tilde{t}_j - \tilde{t}_i$ if $j \neq 0$ and $\tau = - \tilde{t}_i$ if $j=0$. Furthermore, \eqref{lesk} remains true and thus we get the announced estimate as in the previous cases.

\item \textit{Second case}: $\max\p{n,\ell} \leqslant N$. There is a finite number of possible values for $n$ and $\ell$ in this case, and consequently we are only interested in $k$. Working as in the previous cases, we then see that there are at most $M \ln\p{\frac{1}{\epsilon}}^{\sigma d}$ elements in $A_\epsilon$ such that $\max\p{n,\ell} \leqslant N$ in this case.

\item \textit{Third case}: $\p{\ell,i} \nhra \p{n,i}$ \textit{and} $\max\p{n,\ell} > N$. In this case, where $c$ is from Lemma \ref{distsw}, for all $x \in \textup{supp } G$ we have
\begin{equation*}
d\p{k,\textup{supp } \psi_{\Theta',n,s}} + d\p{k,{}^t D_x \mathcal{T}\p{\textup{supp } \tilde{\psi}_{\Theta,\ell,t}}} \geqslant d\p{\textup{supp } \psi_{\Theta',n,s},{}^t D_x \mathcal{T}\p{\textup{supp } \tilde{\psi}_{\Theta,\ell,t}}} \geqslant c \max\p{n,l}^\alpha
\end{equation*}
and thus we have
\begin{equation*}
d\p{k,\textup{supp } \psi_{\Theta',n,s}} \geqslant \frac{c}{2} \max\p{n,\ell}^\alpha
\end{equation*}
or
\begin{equation*}
\delta\p{k,\ell,j}  \geqslant \frac{c}{2} \max\p{n,\ell}^\alpha.
\end{equation*}
This implies using Lemmas \ref{drw} and \ref{gauw} that (for some new constant $C$ and $b$ large enough) we have
\begin{equation}\label{trucw}
\n{c_k \circ L \circ \tilde{\psi}_{\Theta,\ell,j}\p{D}\circ \pi_{\ell,j}}_{\B'} \n{\iota_{n,i}\p{\psi_{\Theta',n,i}\p{D}\rho_k}}_{\B} \leqslant C b^{n+\ell} \exp\p{-C^{-1}\max\p{n,\ell}^{\frac{\alpha}{\sigma+1}}},
\end{equation}
which can be improved when $\b{k} \geqslant Cn^\alpha$ in
\begin{equation}\label{machinw}
\n{c_k \circ L \circ \tilde{\psi}_{\Theta,\ell,j}\p{D}\circ \pi_{\ell,j}}_{\B'} \n{\iota_{n,i}\p{\psi_{\Theta',l,t}\p{D}\rho_k}}_{\B} \leqslant C b^{n+\ell} \exp\p{-C^{-1}\p{\max\p{n,\ell}^{\frac{\alpha}{\sigma+1}} + \b{k}^{\frac{1}{\sigma}}}}.
\end{equation}
Provided that $\max\p{n,l}$ is large enough, which can be performed by making $N$ larger, \eqref{trucw} implies that $\max\p{n,\ell}$ is smaller than $M \ln\p{\frac{1}{\epsilon}}^{\frac{\sigma+1}{\alpha}}$ for some constants $M$ and provided that $\epsilon$ is small enough (we need $\alpha > \sigma + 1$ to get this estimate). And then \eqref{machinw} ensures that $\b{k} \leqslant C n^\alpha$ or $\b{k} \leqslant M \ln\p{\frac{1}{\epsilon}}^\sigma$. At the end, there are at most $M \ln \p{\frac{1}{\epsilon}}^{2 \frac{\sigma+1}{\alpha} + \p{\sigma+1}d}$ elements in $A_\epsilon$ in this third case (and $2 \frac{\sigma+1}{\alpha} + \p{\sigma + 1}d < \beta$).
\end{itemize}
\end{proof}

We have now all the information we need to prove Propositions \ref{locrep} and \ref{trace}.

\begin{proof}[Proof of Proposition \ref{locrep}]
Take $\p{u_{\ell,t}}_{\p{\ell,t} \in \Gamma} \in \B$ with finite support and such that $u_{\ell,t} \in \G_\sigma$, for all $\p{\ell,t} \in \Gamma$. Write $u= \sum_{\p{\ell,t} \in \Gamma} \tilde{\psi}_{\Theta,\ell,t}\p{D} u_{\ell,t} \in \G_\sigma$ and notice that $L u \in \G_\sigma \subseteq \h_{\Theta',\alpha,\bar{t}}$ and that $\mathcal{Q}_{\Theta'} Lu = \M \p{u_{\ell,t}}_{\p{\ell,t} \in \Gamma}$, where $\M$ is defined in Lemma \ref{nucmieuxw}. But the set of such sequences is easily seen to be dense in $\B$ and thus $\M$ sends $\B$ into $\mathcal{Q}_{\Theta'} \p{\h_{\Theta',\alpha,\bar{t}}}$.

We can then set $\L = \mathcal{Q}_{\Theta'}^{-1} \circ p \circ \M \circ \mathcal{Q}_{\Theta} $, where $p$ is the orthogonal projection on $\mathcal{Q}_{\Theta'} \p{\h_{\Theta',\alpha,\bar{t}}}$ (which is a closed subspace of $\B$ by Proposition \ref{descrip}). That $\L$ extends $L$ is a consequence of the calculation above (since the elements of $\G_\sigma$ whose Fourier transform is compactly supported are dense in $\h_{\Theta,\alpha,\bar{t}}$, recall Remark \ref{densité}). The representation \eqref{repre} is then an immediate consequence of Lemma \ref{nucmieuxw}, replacing the vector $e_m$ by $\mathcal{Q}_{\Theta'}^{-1} \circ p \p{e_m}$, and the linear form $l_m$ by $l_m \circ \mathcal{Q}_{\Theta}$, and normalizing.
\end{proof}

\begin{proof}[Proof of Proposition \ref{trace}]
From the construction of $\L$ in the proof of Proposition \ref{locrep}, we can deduce that $\M$ on $\B$ and $\L$ on $\h_{\Theta,\alpha,\bar{t}}$ have the same non-zero spectrum and thus, by the Lidskii trace theorem, the same trace. Now, $\M$ may be written as
\begin{equation*}
\M = \sum_{\p{n,i} \in \Gamma} \sum_{\p{\ell,j} \in \Gamma} \sum_{k \in \Z^d}  \p{\iota_{n,i}\p{\psi_{\Theta,n,i}\p{D} \rho_k}} \otimes \p{c_k \circ L \circ \tilde{\psi}_{\Theta,\ell,j}\p{D} \circ \pi_{\ell,j}}
\end{equation*}
where the sum converges in nuclear operator topology (see Lemma \ref{nucmieuxw}). Thus we have
\begin{align*}
\textrm{tr}\p{\M} & = \sum_{\p{n,i} \in \Gamma} \sum_{\p{\ell,j} \in \Gamma} \sum_{k \in \Z^d} c_k \circ L \circ \tilde{\psi}_{\Theta,\ell,j}\p{D} \circ \pi_{\ell,j} \circ \iota_{n,i}\circ \psi_{\Theta,n,i}\p{D} \p{\rho_k} \\
   & = \sum_{\p{n,i} \in \Gamma} \sum_{k \in \Z^d} c_k \circ L \circ \psi_{\Theta,n,i}\p{D} \p{\rho_k},
\end{align*}
where we used that $\iota_{n,i} \circ \pi_{\ell,j} = 0$ if $\p{n,i} \neq \p{\ell,j}$ and that $\tilde{\psi}_{\Theta,n,i}\p{D} \circ \psi_{\Theta,n,i}\p{D} = \psi_{\Theta,n,i}\p{D}$. Notice that
\begin{equation*}
c_k \circ L \circ \psi_{\Theta,n,i}\p{D} \p{\rho_k} = \frac{1}{\p{2\pi}^{2d}} \s{\R^d\times\R^d}{e^{i\p{\mathcal{T}\p{x} \eta - k x}} G\p{x} \psi_{\Theta,n,i}\p{\eta} \hat{\rho}\p{\eta - k}}{x} \mathrm{d}\eta.
\end{equation*}
Using dominated convergence ($\hat{\rho}$ is rapidly decaying), we can sum over $\p{n,i} \in \Gamma$ and get
\begin{equation*}
\textrm{tr}\p{\M} = \sum_{k \in \Z^d} \frac{1}{\p{2\pi}^{2d}} \s{\R^d \times \R^d}{e^{i\p{\mathcal{T}\p{x} \eta - k x}} G\p{x} \hat{\rho}\p{\eta - k}}{x} \mathrm{d}\eta.
\end{equation*}
Dominated convergence also provides
\begin{align*}
\frac{1}{\p{2\pi}^{2d}} \s{\R^d \times \R^d}{e^{i\p{\mathcal{T}\p{x} \eta - k x}}  g\p{x} & \hat{\rho}\p{\eta - k}}{x}  \mathrm{d}\eta \\ & = \lim_{\epsilon \to 0} \frac{1}{\p{2\pi}^{2d}} \s{\p{\R^d}^3}{e^{i\p{\mathcal{T}\p{x} \eta - k x + k y}} G\p{x} \rho\p{y}e^{- \epsilon \b{\eta}^2}}{x}  \mathrm{d}y \mathrm{d}\eta \\
       & = \lim_{\epsilon \to 0} \frac{1}{\p{2\pi}^{2d}} \s{\p{\R^d}^3}{e^{i\p{\mathcal{T}\p{x} \eta - k x + k y}} G\p{x} \rho\p{y}e^{-\epsilon \b{\eta}^2}}{x}  \mathrm{d}y \mathrm{d}\eta \\
       & = \lim_{\epsilon \to 0} \frac{1}{\p{2\pi}^{d}} \frac{1}{\p{2\sqrt{\pi\epsilon}}^d} \s{\p{\R^d}^2}{e^{ik\p{ y - x}} G\p{x} \rho\p{y}e^{-\frac{\p{\mathcal{T}\p{x} - y}^2}{4\epsilon}}}{x} \mathrm{d}y \\
       & = \frac{1}{\p{2\pi}^d} \s{\R^d}{e^{-ik\p{x - \mathcal{T}\p{x}}} \rho\p{\mathcal{T}\p{x}} G\p{x}}{x}.
\end{align*}
Generalized cone-hyperbolicity of $\mathcal{T}$ implies that $1$ is never an eigenvalue for $D_x \mathcal{T}$, and thus $x \mapsto \mathcal{T}\p{x} - x$ is a local diffeomorphism. Consequently, one can cover the support of $G$ by a finite number of open sets $\p{O_s}_{s \in \Sigma}$ such that for all $s \in \Sigma$ the restriction of $ x \mapsto \mathcal{T}\p{x}- x$ to $O_j$ can be extended to a diffeomorphism $H_s$ from $\R^d$ to itself (a similar construction will be carried out with more details in the proof of Lemma \ref{enfin}, see formula \eqref{exten}). Choose then a finite family $\p{\theta_s}_{s \in \Sigma}$ such that for all $s \in \Sigma$ the function $\theta_s : \R^d \to \left[0,1\right]$ is $\mathcal{C}^\infty$ and supported in $O_s$ and for all $x \in \textup{supp } G$ we have $\sum_{ s \in \Sigma} \theta_s\p{x} = 1$. Then we have, performing the change of variables "$u=H_s\p{x}$",
\begin{align*}
\frac{1}{\p{2\pi}^d} \s{\R^d}{e^{-ik\p{x - \mathcal{T}\p{x}}} \rho\p{\mathcal{T}\p{x}} G\p{x}}{x} & = \sum_{s \in \Sigma} \frac{1}{\p{2\pi}^d} \s{\R^d}{e^{-iku} A_s\p{u}}{u} = \frac{1}{\p{2 \pi}^d}\sum_{s \in \Sigma} \widehat{A}_s\p{k}
\end{align*}
where
\begin{equation*}
A_s\p{u} = \rho \circ \mathcal{T} \circ H_s^{-1}\p{u}  \theta_j \circ H_s^{-1}\p{u} G \circ H_s^{-1}\p{u} \b{\det D_u H_s^{-1}}.
\end{equation*}
Notice that $A_s$ is smooth and compactly supported, thus the Poisson summation formula provides
\begin{equation*}
\frac{1}{\p{2\pi}^d} \sum_{k \in \Z^d} \widehat{A}_s\p{k} = \sum_{k \in \Z^d} A_s\p{2 \pi k}.
\end{equation*}
However, if $k \neq 0$ then $H_s^{-1}\p{2 \pi k}$ and $\mathcal{T} \circ H_s^{-1}\p{2 \pi k} =  H_s^{-1}\p{2\pi k} - 2 \pi k$ cannot be both in $\left]-\pi,\pi,\right[^d$ and thus $A_s\p{2\pi k} = 0$ (since $\rho$ and $G$ are supported in $\left]-\pi,\pi\right[^d$). Consequently, we have
\begin{equation*}
\textrm{tr}\p{\M} = \sum_{s \in \Sigma} A_s\p{0}.
\end{equation*}
Furthermore, if $\mathcal{T}$ has a fixed point $x^*$ in $O_s$, it has only one and it is $H_s^{-1}\p{0}$, thus we can write $A_s\p{0} = \frac{\theta_j\p{x^*} G\p{x^*}}{\b{\det\p{I-D_{x^*} \mathcal{T}}}}$. Otherwise, $H_s^{-1}\p{0}$ does not belong to the support of $\theta_s$ and consequently $A_s\p{0} = 0$. We have then
\begin{equation*}
\sum_{s \in \Sigma} A_s\p{0} = \sum_{\mathcal{T}x = x} \frac{ G\p{x}}{\b{\det\p{I-D_{x} \mathcal{T}}}}
\end{equation*}
where we used that for all $x \in \textup{supp } G$ we have $\sum_{ s \in \Sigma} \theta_s\p{x} = 1$.
\end{proof}

\begin{rmq}
Nuclearity of the local transfer operator $\L$ and formula \eqref{fortra} for its trace may also be proven by giving an integral representation of the operators $S_{n,i}^{\ell,j}$. This is maybe a more natural way to work than the one exposed here. However, our method gives the representation \eqref{repre} with the estimate \eqref{reprequant}, which makes the Theorem \ref{main} much stronger thanks to Lemma \ref{abstrait}.
\end{rmq}

\section{Global space and proof of Theorem \ref{main}}\label{preuve}

Let us now prove Theorem \ref{main}. Recall that $T$ and $g$ are $\sigma$-Gevrey for some $\sigma>1$ which is now fixed. In order to define the global space $\h$, we shall glue together the local spaces $\h_{\Theta,\alpha,\bar{t}}$ from \S \ref{locspace}. To do so, we begin by proving the following lemma.

\begin{lm}\label{enfin}
There are a compact isolating neighbourhood $V$ for $K$, a finite set $\Omega$ and an integer $r \geqslant 2$ such that the following holds. For every $\omega \in \Omega$ there are a $\sigma$-Gevrey chart $\kappa_\omega : U_\omega \to V_\omega$ and a generalized polarization with $r$ cones $\Theta_\omega = \p{C_{i,\omega},\phi_{i,\omega}}_{0 \leqslant i \leqslant r}$ such that $V \subseteq \bigcup_{\omega \in \Omega} U_\omega$ and for every $\omega,\omega' \in \Omega$, if $T\p{U_\omega}\cap U_{\omega'} \neq \emptyset$, then the map
\begin{equation*}
\kappa_{\omega'} \circ T \circ \kappa_\omega^{-1} : \kappa_\omega\p{U_\omega \cap T^{-1}\p{U_{\omega'}}} \to V_{\omega'}
\end{equation*}
admits an extension $\mathcal{T}_{\omega,\omega'} : \R^d \to \R^d$ which is a $\sigma$-Gevrey generalized cone-hyperbolic diffeomorphism from $\Theta_\omega$ to $\Theta_\omega'$.
\end{lm}

\begin{proof}
Choose $\tilde{\lambda} < \lambda$ such that $\tilde{\lambda} > 1$. We may define a Mather metric $\b{\cdot}_x$ on a neighborhood of $K$. For this metric, the spaces $E_x^u$ and $E_x^s$ from Definition \ref{bashyp} are orthogonal and the restriction of $D_x T$ to these spaces is respectively a dilatation or a contraction by a factor $\tilde{\lambda}$, respectively $\tilde{\lambda}^{-1}$. If $x \in K$ and $v \in T_xM$, we write $v = v_u+v_s$ for the decomposition of $v$ with respect to $T_x M = E_x^u \oplus E_x^s$. Choose $\gamma_2 > 0$ large enough so that $\tilde{\lambda} > \sqrt{1+ \frac{1}{\gamma_2^2}} $ and define for all $i \geqslant 1$ the number $\gamma_i = \tilde{\lambda}^{1-\frac{i}{2}} \gamma_1$. Fix $r \geqslant 2$ large enough so that $\tilde{\lambda} > \sqrt{1 + \gamma_{r-1}^2}$.

Then for all $x \in K$ define closed cones $\p{C_i\p{x}}_{i \geqslant 1}$ depending Hölder continuously on $x$ by
\begin{equation*}
C_i\p{x} = \set{ v \in T_x M : \b{v_s}_x \leqslant \gamma_i \b{v_u}_x},
\end{equation*}
for $i \in \set{1,\dots,r+1}$ and set $C_0\p{x} = T_x M$.

Let us prove that if $x \in K$ and $i \in \set{1,\dots,r+1}$ then ${}^t D_x T\p{C_i\p{Tx}} \subseteq C_{i+4}\p{x}$. Indeed, with respect to the orthogonal decomposition $T_x M = E_x^u \oplus E_x^s$, the matrix of $D_x T$ has the form
\begin{equation*}
D_x T = \left[
\begin{array}{cc}
A_x & 0 \\
0 & B_x
\end{array}
\right]
\end{equation*}
with $A_x$ and $B_x$ respectively a dilatation of factor $\tilde{\lambda}$ and a contraction of factor $\tilde{\lambda}^{-1}$. Now, if $v = v_u + v_s \in C_i\p{Tx}$ then we have ${}^t D_x T\p{v} = {}^t A_x \p{v_u} + {}^t B_x \p{v_s}$ and
\begin{equation*}
\b{{}^t B_x \p{v_s}}_x \leqslant \frac{1}{\tilde{\lambda}}\b{v_s}_{Tx} \leqslant \frac{\gamma_i}{ \tilde{\lambda}} \b{v_u}_{Tx}\leqslant \frac{\gamma_i}{\tilde{\lambda}^2} \b{{}^t A_x v_u}_{x},
\end{equation*}
but $\gamma_{i+4} = \frac{\gamma_i}{\tilde{\lambda}^2}$ by definition of the $\gamma_i$.

We shall now prove that if $x \in K$ then
\begin{equation*}
\Lambda \coloneqq \min\p{ \inf_{\substack{ \b{\xi} = 1 \\ \xi \in C_{r-1}\p{Tx}}} \b{{}^t D_x T \p{\xi}}_x, \p{\sup_{\substack{\b{\xi} = 1 \\ {}^t D_x T\p{\xi} \in T_x M \setminus C_2\p{x}}} \b{{}^t D_x T\p{\xi}}_x}^{-1}} > 1.
\end{equation*}
Let $\xi  = \xi_u + \xi_s \in C_{r-1}\p{Tx}$ then
\begin{equation*}
\b{{}^t D_x T\p{\xi}}_x \geqslant \b{{}^t A_x \xi_u}_x \geqslant \tilde{\lambda} \b{\xi_u}_{Tx},
\end{equation*}
but $\b{\xi_s}_{Tx} \leqslant \gamma_{r-1} \b{\xi_u}_{Tx}$ and thus $\b{\xi}_{Tx} \leqslant \sqrt{1 + \gamma_{r-1}^2} \b{\xi_u}_{Tx}$, and consequently we get
\begin{equation*}
\b{{}^t D_x T\p{\xi}}_x \geqslant \frac{\tilde{\lambda}}{\sqrt{1 + \gamma_{r-1}^2}} \b{\xi}_{Tx}
\end{equation*}
and $r$ has be chosen so that $\frac{\tilde{\lambda}}{\sqrt{1 + \gamma_{r-1}^2}} >1$. The converse inequality is similar.

Now, for all $x \in K$ choose a $\sigma$-Gevrey chart $\kappa_x : U_x \to V_x$ such that $x \in U_x$, $U_x \subseteq W$ and $D_x \kappa_x : T_x M \to \R^d$ is an isometry. Then choose a polarization $\Theta_x = \p{C_{i,x},\phi_{i,x}}_{0 \leqslant i \leqslant r}$ such that $D_x \kappa_x \p{C_{i+1}\p{x}} \sqsubset C_{i,x} \sqsubset D_x \kappa_x \p{C_i\p{x}} $ for $i \in \set{1,\dots,r}$. Since the $C_i\p{y}$ depend continuously on $y$, we can find an open ball $B_x$ centered at $x$ such that $\bar{B_x} \subseteq U_x $, and for all $y \in \bar{B_x}$ we have for $i \in \set{1,\dots,r}$
\begin{equation*}
\p{{}^t D_y \kappa_x}^{-1} C_{i+1}\p{y} \sqsubset C_{i,x} \sqsubset \p{{}^t D_y \kappa_x}^{-1} C_{i}\p{y}
\end{equation*}
and
\begin{equation*}
\n{D_y \kappa_x} \leqslant 1 + \epsilon \textup{ and } \n{\p{D_y \kappa_x}^{-1}}^{-1} \geqslant 1 - \epsilon,
\end{equation*}
where $\epsilon > 0$ is small enough so that
\begin{equation*}
\frac{1 - \epsilon}{1 + \epsilon} \Lambda > 1.
\end{equation*}
Since $K$ is compact, we can find a finite number of point $x_1,\dots,x_m$ in $K$ such that $K \subseteq \bigcup_{k=1}^m \kappa_{x_k}^{-1}\p{B_{x_k}} $.

Choose $z \in \bigcup_{k=1}^m \kappa_{x_k}^{-1}\p{\bar{B_{x_k}}}$, and let $j,k \in \set{1,\dots,m}$ be such that $z \in \kappa_{x_j}^{-1}\p{\bar{B_{x_j}}}$ and $Tz \in \kappa_{x_j}^{-1}\p{\bar{B_{x_j}}}$. We consider the map
\begin{equation*}
T_{z,j,k} = \kappa_{x_k} \circ T \circ \kappa_{x_j}^{-1} : V_{x_j}\cap \kappa_{x_j}^{-1}\p{T^{-1} U_{x_k}} \to \R^d.
\end{equation*} 
Write $z_j = \kappa_{x_j}\p{z}$ and notice that for $i \in \set{1,\dots,r}$ we have
\begin{align*}
{}^t D_{z_j} T_{z,j,k} \p{C_{i,x_k}} & \sqsubset \p{{}^t D_{z_j} \kappa_{x_j}}^{-1} \circ {}^t D_z T \p{C_i\p{Tz}} \subseteq \p{{}^t D_{z_j} \kappa_{x_j}}^{-1} \p{C_{i+4}\p{z}} \subseteq C_{\min\p{i+2,r},x_j}
\end{align*}
and
\begin{equation*}
\min\p{ \inf_{\substack{ \b{\xi} = 1 \\ \xi \in C_{r-1,x_k}}} \b{{}^t D_{z_j} T_{z,j,k} \p{\xi}}, \p{\sup_{\substack{\b{\xi} = 1 \\ {}^t D_{z_j} T_{z,j,k}\p{\xi} \in \R^d \setminus C_{1,x_j}}} \b{{}^t D_{z_j} T_{z,j,k}\p{\xi}}}^{-1}} \geqslant \frac{1-\epsilon}{1 + \epsilon} \Lambda > 1.
\end{equation*}
This is due to the construction of $B_{x_j}$ and $B_{x_k}$. Choose $\rho : \R^d \to \left[0,1\right]$ a $\sigma$-Gevrey function such that $\rho\p{y} = 1$ if $\b{y} \leqslant \frac{1}{2}$, and $\rho\p{y} = 0$ if $\b{y} \geqslant 1$. For $\epsilon = \epsilon\p{z,j,k}>0$ that we shall take small enough (in particular so that \eqref{exten} makes sense), define
\begin{equation}\label{exten}
\mathcal{T}_{z,j,k}: y \mapsto \rho\p{\frac{y-z_j}{\epsilon}}T_{z,j,k}\p{y} + \p{1-\rho\p{\frac{y-z_j}{\epsilon}}}\p{T_{z,j,k}\p{z_j} + D_{z_j} T_{z,j,k}\p{y-z_j}}.
\end{equation}
By taking $\epsilon$ small enough $\mathcal{T}_{z,j,k}$, can be made arbitrarily close in the $\mathcal{C}^1$ topology to the affine map $y \mapsto T_{z,j,k}\p{z_j} + D_{z_j} T_{z,j,k}\p{y-z_j}$. In particular, it is a diffeomorphism from $\R^d$ to itself (use for instance Hadamard--Levy's theorem), generalized cone-hyperbolic from $\Theta_{x_j}$ to $\Theta_{x_k}$. Furthermore, there is a neighbourhood $W_{z,j,k} \subseteq U_{x_j}$ of $z$ such that $T_{z,j,k}$ and $\mathcal{T}_{z,j,k}$ coincide on $\kappa_{x_j}\p{W_{z,j,k}}$. If $j$ and $k$ are such that $z \notin \kappa_{x_j}^{-1}\p{\bar{B_{x_j}}}$ or $Tz \notin \kappa_{x_k}^{-1}\p{\bar{B_{x_k}}}$, just set $W_{z,j,k} = M$.

Set now $\mathcal{W}_z = \bigcap_{j,k=1}^m W_{z,j,k}$ for all $z \in K$, and choose a neighbourhood $E_z \subseteq \mathcal{W}_z$ of $z$ in $M$ small enough so that if $T\p{E_z} \cap \kappa_{x_j}\p{B_{x_j}} \neq \emptyset$ for some $j \in \set{1,\dots,m}$ then $Tz \in \kappa_{x_j}\p{\bar{B_{x_j}}}$. By compacity of $K$, one can find a finite number of points $z_1,\dots,z_n$ such that $K \subseteq \bigcup_{k=1}^n E_{z_k}$. Finally, we define
\begin{equation*}
\Omega = \set{ \omega=\p{j,k} : j \in \set{1,\dots,m}, k \in \set{1,\dots,n} \textrm{ and } U_\omega := E_{z_k} \cap \kappa_{x_j}\p{B_{x_j}} \neq \emptyset}.
\end{equation*}
If $\omega = \p{j,k} \in \Omega$ the associated chart is $\kappa_\omega = \left.\kappa_{x_j}\right|_{U_\omega} : U_{\omega} \to \kappa_{x_j}\p{U_\omega}$ and the generalized polarization is $\Theta_\omega = \Theta_{x_j}$. If $\omega' = \p{j',k'}$ is such that $T\p{U_\omega} \cap U_{\omega'} \neq \emptyset$ then the map $\kappa_{\omega'} \circ T \circ \kappa_{\omega}^{-1}$ may be extended by $\mathcal{T}_{\omega,\omega'} = \mathcal{T}_{z_k,j,j'}$. Finally, let $V$ be any compact isolating neighbourhood of $K$ that is contained in $\bigcup_{\omega \in \Omega} U_\omega$.
\end{proof}

We next define the space $\h$ announced in Theorem \ref{main}. We need more notation.

Choose a $\sigma$-Gevrey partition of unity $\p{\theta_\omega}_{\omega \in \Omega}$ on $V$ subordinated to $\p{U_\omega}_{\omega \in \Omega}$, that is, for every $\omega \in \Omega$ the function $\theta_\omega : M \to \left[0,1\right]$ is $\sigma$-Gevrey and supported in $U_\omega$, and, in addition, for all $x \in V$, we have $\sum_{\omega \in \Omega} \theta_\omega\p{x} = 1$. Up to reducing $V$, we may suppose that $\bigcup_{\omega \in \Omega} U_\omega$ is isolating for $K$. Notice that if $\omega \in \Omega$ and $u \in \U_\sigma\p{V}$ then $\p{\theta_\omega u} \circ \kappa_\omega^{-1}$ is a well-defined compactly supported element of $\U_\sigma$ (see Remark~\ref{ose}). Consequently, we can define a map
\begin{equation}\label{defPhi}
\begin{array}{ccccc}
\Phi & : & \U_\sigma\p{V} & \to & \oplus_{\omega \in \Omega} \U_\sigma \\
 & & u & \mapsto & \p{\p{\theta_\omega u} \circ \kappa_\omega^{-1}}_{\omega \in \Omega}
\end{array}
\end{equation}
as well as
\begin{equation}\label{defS}
\begin{array}{ccccc}
S & : & \oplus_{\omega \in \Omega} \U_\sigma & \to & \U_\sigma\p{M} \\
& & \p{u_\omega}_{\omega \in \Omega} & \mapsto & \sum_{\omega \in \Omega} \p{h_\omega u_\omega} \circ \kappa_\omega
\end{array}
\end{equation}
where for all $\omega \in \Omega$ the function $h_\omega$ is $\sigma$-Gevrey and compactly supported in $V_\omega$, and $h_\omega\p{x} = 1$ for all $x \in \kappa_\omega\p{\textup{supp } \theta_\omega}$. Notice that $S \circ \Phi$ is the inclusion of $\U_\sigma\p{V}$ into $\U_\sigma\p{M}$.

Fix $\alpha > \sigma +1$ and notice that for all $\omega,\omega' \in \Omega$ the map $\mathcal{T}_{\omega,\omega'}$, if defined, is cone hyperbolic from $\Theta_\omega$ to $\Theta_{\omega'}$ and thus provides a $ \Lambda_{\omega,\omega'}$ (from Definition \ref{ch}). We choose $\bar{t} = \p{t_0,\dots,t_r} \in \R^{r+1}$ such that \eqref{sueur} holds with $\nu > 1$ strictly smaller than the $\Lambda_{\omega,\omega'}^{\frac{1}{\alpha}}$'s and $a$ that satisfies \eqref{defa} with $\mathcal{T}$ replaced by $\mathcal{T}_{\omega,\omega'}$ (when defined) for all $x \in \R^d$ (this is possible since the differential of $\mathcal{T}_{\omega,\omega'}$ is constant outside from a compact subset of $\R^d$). Then define a Hilbert space $\h_{\Omega,\alpha,\bar{t}} = \oplus_{\omega \in \Omega} \h_{\Theta_\omega,\alpha,\bar{t}}$, where the $\h_{\Theta_\omega,\alpha,\bar{t}}$ are the Hilbert spaces defined in \S \ref{locspace}. Define the auxiliary space $ \widetilde{\h} = \set{u \in \U_\sigma\p{V} : \Phi u \in \h_{\Omega,\alpha,\bar{t}} } $ that we endow with the Hermitian norm $\n{u}_{\widetilde{\h}} = \n{\Phi u}_{\h_{\Omega,\alpha,\bar{t}}}$.

\begin{lm}\label{verif}
$\widetilde{\h}$ is a separable Hilbert space (or equivalently $\Phi\p{\widetilde{\h}}$ is a closed subspace of $\h_{\Omega,\alpha,\bar{t}}$). The inclusion of $\widetilde{\h}$ in $\U_\sigma\p{V}$ is continuous, and $\G_\sigma$ is contained in $\widetilde{\h}$ the inclusion being continuous.
\end{lm}

\begin{proof}
The proof of the second assertion is easy and left to the reader (use $\Phi$ and $S$ from \eqref{defPhi} and \eqref{defS}, and the analogous statement for the $\h_{\Theta_\omega,\alpha,\bar{t}}$). The definition of $\widetilde{\h}$ readily implies that
\begin{equation*}
\Phi\p{\widetilde{\h}} = \Phi\p{\U_\sigma\p{V}} \cap \h_{\Omega,\alpha,\bar{t}}.
\end{equation*}
However, from Proposition \ref{loc}, we know that the injection of $\h_{\Omega,\alpha,\bar{t}}$ in $\oplus_{\omega \in \Omega} \U_\sigma$ is continuous. Thus we only need to prove that $\Phi\p{\U_\sigma\p{V}}$ is closed in $\oplus_{\omega \in \Omega} \U_\sigma$. However, one can see that
\begin{equation*}
\Phi\p{\U_\sigma\p{V}} = \set{u \in \oplus_{\omega \in \Omega} \U_\sigma : \Phi S u = u \textrm{ and } Su \textrm{ is supported in } V}
\end{equation*}
and consequently $\Phi\p{\U_\sigma\p{V}}$ is closed in $\oplus_{\omega \in \Omega} \U_\sigma$ (see Remark \ref{ose}).
\end{proof}

We can now define the space $\h$ announced in Theorem \ref{main} : $\h$ is the closure of $\G_\sigma\p{V}$ in $\widetilde{\h}$.

\begin{proof}[Proof of Theorem \ref{main}]
Notice that with the definition of $\h$ and Lemma \ref{verif}, the first two points of Theorem \ref{main} are satisfied. Recall that the transfer operator that we want to extend is defined on $\G_\sigma\p{V}$ by
\begin{equation*}
Lu\p{x} = g\p{x} u \circ T\p{x}.
\end{equation*}
If $\omega,\omega'$ are such that $T\p{U_\omega} \cap U_{\omega'} \neq \emptyset$, define the auxiliary operator $L_{\omega,\omega'}$ on $\G_\sigma$ by
\begin{equation*}
L_{\omega,\omega'}u\p{x} = G_{\omega,\omega'}\p{x} u \circ \mathcal{T}_{\omega,\omega'}\p{x}
\end{equation*}
where
\begin{equation*}
G_{\omega,\omega'}\p{x} = \theta_{\omega} \circ \kappa_{\omega}^{-1}\p{x} g \circ \kappa_{\omega}^{-1} \p{x} h_{\omega'} \circ \mathcal{T}_{\omega,\omega'}\p{x}
\end{equation*}
properly extended by zero. If $T\p{U_\omega} \cap U_{\omega'} = \emptyset$ just set $L_{\omega,\omega'} = 0$. We can then define an operator $\widetilde{L}$ that acts on $\oplus_{\omega \in \Omega} \G_\sigma$ as the matrix $\p{L_{\omega,\omega'}}_{\omega,\omega' \in \Omega}$, that is 
\begin{equation}\label{matrice}
\widetilde{L} \p{u_\omega}_{\omega \in \Omega} = \p{\sum_{\omega' \in \Omega} L_{\omega,\omega'} u_{\omega'} }_{\omega \in \Omega}.
\end{equation}
For every $\omega,\omega' \in \Omega$, Proposition \ref{locrep} provides a trace class operator $\L_{\omega,\omega'} : \h_{\Theta_{\omega'},\alpha,\bar{t}} \to \h_{\Theta_\omega,\alpha,\bar{t}}$ that extends $L_{\omega,\omega'}$, and thus we have with \eqref{matrice} a trace class operator $\widetilde{\L} : \h_{\Omega,\alpha,\bar{t}} \to \h_{\Omega,\alpha,\bar{t}}$ that extends $\widetilde{L}$. But notice that $\widetilde{L}$ sends $\oplus_{\omega \in \Omega} \G_\sigma$ into $\Phi\p{\G_{\sigma}\p{V}}$ (the operator $\widetilde{L}$ is defined by the formula $\widetilde{L} = \Phi \circ L \circ S$, and $g$ is supported in $V$). Thus $\widetilde{\L}$ sends $\h_{\Omega,\alpha,\bar{t}}$ into $\Phi\p{\h}$. Since $\Phi$ induces an isometry $\Psi$ between $\h$ and $\Phi\p{\h}$, we may define $\L$ as  $\L = \Psi^{-1} \circ \widetilde{\L} \circ \Psi$. With this definition, $\L$ extends $L$ (since $\widetilde{L} = \Phi \circ L \circ S$) and is trace class since it is conjugated with the operator induced by $\widetilde{\L}$ on $\Phi\p{\h}$. Moreover, $\L$ and $\tilde{\L}$ have the same non-zero spectrum and consequently we have $\textrm{tr}\p{\L} = \textrm{tr}\p{\widetilde{\L}}$. Using the Lidskii trace theorem, it is then easy to see that
\begin{equation*}
\textrm{tr}\p{\L} = \sum_{\omega \in \Omega} \textrm{tr}\p{\L_{\omega,\omega}}
\end{equation*}
which gives by Proposition \ref{trace}
\begin{align*}
\textrm{tr}\p{\L} & = \sum_{\omega \in \Omega} \sum_{\mathcal{T}_{\omega,\omega} x = x} \frac{G_{\omega}\p{x}}{\b{\det\p{I-D_x\mathcal{T}_{\omega,\omega}}}} = \sum_{\omega \in \Omega}\sum_{\mathcal{T}_{\omega,\omega} x = x}\frac{\theta_\omega \circ \kappa_\omega^{-1}\p{x} g \circ \kappa_{\omega}^{-1}\p{x}}{\b{\det\p{I-D_x\mathcal{T}_{\omega,\omega}}}}\\
& = \sum_{\omega \in \Omega} \sum_{\substack{y \in M \\ Ty=y}} \frac{\theta_\omega\p{y}g\p{y}}{\b{\det\p{I-D_y T}}} = \sum_{\substack{ y \in K \\ Ty=y}} \frac{g\p{y}}{\b{\det\p{I-D_y T}}} =\tf{\L},
\end{align*}
where we used that any fixed point of $T$ in $U = \bigcup_{\omega \in \Omega} U_\omega$ is in fact in $K$, since $U$ is an isolating neighbourhood for $K$. We could prove that $\textrm{tr}\p{\L^m} = \tf{\L^m}$ for $m \geqslant 2$ by giving a similar decomposition of $\widetilde{\L}^m$ but there is a faster way to show this. From Lemma \ref{equiva} and the Lidskii trace theorem, we know that $\textrm{tr}\p{\L}$, and thus $\tf{\L}$, is the sum of the Ruelle resonances for $\p{T,g}$. But then, we can apply this result to the system $\p{T^m,g^{\p{m}}}$ and it appears that $\tf{\L^m}$ is equal to the sum of the $m$th powers of Ruelle resonances for $\p{T,g}$, which also are the non-zero eigenvalues of $\L^m$ using Lemma \ref{equiva} again, and finally the Lidskii trace theorem implies that $\tf{\L^m} = \textrm{tr}\p{\L^m}$.

To prove points \ref{better} and \ref{butter}, we need to establish a representation of the type \eqref{volrep} with the estimate \eqref{dif} for $\L$, which will allow us to use Lemma \ref{abstrait}. However, from Proposition \ref{locrep}, we know that the $\L_{\omega,\omega'}$ admit such a representation, thus $\widetilde{\L}$ does too, and $\L$ as well (it is conjugated with the operator induced by $\widetilde{\L}$ on a closed subspace of $\h_{\Omega,\alpha,\bar{t}}$). This proves points \ref{better} and \ref{butter} for $\beta = 2 + \alpha d$. Recalling that $\alpha$ may be chosen to be any real number strictly greater than $\sigma + 1$, the proof of Theorem \ref{main} is complete.
\end{proof}

\appendix

\section*{Appendix}

\section{Proof of Lemma \ref{equiva} }\label{proofequiva}

We shall apply \cite[Lemma A.1]{Tsu}. To do so, let $\widetilde{\B}$ be a Banach space given by Theorem \ref{fond}. Notice that we have
\begin{equation*}
\G_\sigma\p{V} \subseteq \mathcal{C}^{\infty}\p{V} \subseteq \widetilde{\B} \subseteq \mathcal{D}'\p{V} \subseteq \U_\sigma\p{V} 
\end{equation*}
with all inclusions continuous. But $\G_\sigma\p{V}$ is dense in $\mathcal{C}^\infty\p{V}$ and thus is dense in $\widetilde{\B}$. Thus, \cite[Lemma A.1]{Tsu} ensures that outside of the closed disc of radius $\epsilon$, the spectrum of $\L$ acting on $\B$ and $\B'$ is the same. \qed

\section{Proof of Lemma \ref{abstrait}}\label{proofabstrait}

We will need two preparatory technical lemmas.

\begin{lm}\label{preced}
For all $\beta \geqslant 0$ there is a constant $C > 0$ such that for all $r \geqslant 2$ we have
\begin{equation}\label{fr}
r \int_{\log r}^{+ \infty} e^{-u} u^\beta \mathrm{d}u \leqslant C \p{\log r}^{\beta}
\end{equation}
\end{lm}

\begin{proof}
The logarithm of the left-hand side of \eqref{fr} is convex as a function of $\beta$ and thus we only need to prove \eqref{fr} when $\beta$ is an integer. But then performing $\beta$ integrations by parts, we see that the left-hand side of \eqref{fr} is a polynomial of degree $\beta$ in $\log r$.
\end{proof}

\begin{lm}\label{ruse}
Let $\beta >0$ and $0 < \theta < 1$ then there are constants $M,D >0$ such that for all $n \geqslant 1$ we have
\begin{equation*}
\sum_{m_1 < \dots < m_n} \theta^{\sum_{j=1}^n m_j^{\frac{1}{\beta}}} \leqslant M \exp\p{-Dn^{1 + \frac{1}{\beta}}}.
\end{equation*}
\end{lm}

\begin{proof}
Notice that the series $\sum_{m \geqslant 1}\theta^{m^{\frac{1}{\beta}}}$ converges and thus the infinite product $\prod_{m \geqslant 1} \p{1 + \theta^{m^{\frac{1}{\beta}}} z}$ converges uniformly on all compact of $\C$ to a holomorphic function $f$ that we may write $f : z \mapsto \sum_{n \geqslant 0} a_n z_n$ with $a_n = \sum_{m_1 < \dots < m_n} \theta^{\sum_{j=1}^n m_j^{\frac{1}{\beta}}} $ when $n \geqslant 1$. Notice that $f$ has genus zero and the number $n\p{r}$ of zeroes of $f$ of modulus smaller than $r$ is bounded by $C \p{\log_+ r}^\beta$ for some $C > 0$. Thus \cite[3.5.1]{Boas} implies that for $r$ large enough (with the change of variable "$u = \log t$" and Lemma \ref{preced})
\begin{equation*}
\sup_{\b{z} \leqslant r} \log \b{f\p{z}} \leqslant C \int_0^r \frac{\p{\log_+ t}^\beta}{t} \mathrm{d}t + r\int_r^{+ \infty} \frac{\p{\log_+ t}^\beta}{t^2} \mathrm{d}t \leqslant \widetilde{C} \p{\log r}^{\beta + 1}.
\end{equation*}
Now, using Cauchy's formula, we get for all $r$ large enough and $n \in \N$
\begin{equation*}
a_n \leqslant \frac{e^{\widetilde{C} \p{\log r}^{\beta +1}}}{r^n}
\end{equation*}
and the result follows by taking $r = \exp\p{\p{\frac{n}{\widetilde{C}\p{\beta + 1}}}^{\frac{1}{\beta}}}$ (which is large when $n$ is large).
\end{proof}

\begin{proof}[Proof of Lemma \ref{abstrait}]
We have (see page 17 of the second part of \cite{Groth}, the $a_n$ are defined by \eqref{defan})
\begin{equation*}
a_n = \p{-1}^n \sum_{0 \leqslant m_1 < \dots < m_n} \p{\prod_{j=1}^n \lambda_{m_j}} \det\p{\p{l_{m_i}\p{e_{m_j}}}_{1 \leqslant i,j \leqslant n}}
\end{equation*}
then using Hadamard's theorem, \eqref{dif} and Lemma \ref{ruse} we get
\begin{equation*}
\b{a_n} \leqslant C^n M \exp\p{-D n^{1 + \frac{1}{\beta}}} n^{\frac{n}{2}} 
\end{equation*}
and we can remove the factor $C^n n^{\frac{n}{2}}$ by making $M$ larger and $D$ smaller. 

Write $\theta = \exp\p{-D}$. Let $z \in \C$ be large enough and choose $n_0 \in \N$ then we have
\begin{equation*}
\b{\det\p{I+z \L}} \leqslant M n_0 \b{z}^{n_0} + M \sum_{n = n_0}^{+ \infty} \p{\theta^{n_0^{\frac{1}{\beta}}}\b{z}}^n \leqslant n_0 \b{z}^{n_0} + M \b{z}^{n_0} \frac{\theta^{n_0^{1 + \frac{1}{\beta}}}}{1 - \b{z} \theta^{n_0^{\frac{1}{\beta}}}}
\end{equation*}
and we get \eqref{crois} by taking $n_0 = \left\lfloor \p{\frac{- \ln 2 - \ln \b{z}}{\ln \theta}}^\beta \right\rfloor$.

Finally notice that ($N$ is decreasing)
\begin{equation*}
N\p{r} \leqslant \frac{1}{\ln 2} \int_{\frac{1}{r}}^{\frac{2}{r}} \frac{N\p{\frac{1}{t}}}{t} \mathrm{d}t \leqslant \frac{2}{\ln 2} \sup_{\b{z} \leqslant 2r} \log \b{\det\p{I + z \L}}
\end{equation*}
where we used Jensen's formula (see \cite[1.2.1 p.2]{Boas} for instance). The estimate \eqref{nombre} follows.
\end{proof}

\section*{Acknowledgements}
I would like to thank Sebastien Gouëzel who suggested to consider Gevrey dynamics in the context of trace formulae. I am also grateful to Maciej Zworski and Semyon Dyatlov for useful discussions and to Viviane Baladi for numerous suggestions and careful reading of the different versions of this work. Finally, I would like to thank the organizers of the summer school "Analytical aspects of hyperbolics flows" helds in the university of Nantes in July 2017, during which I started this work, and the Mittag--Leffler, where I finished the first version of that paper during the workshop "Fractals and dimension" in December 2017. This research is supported by the European Research Council (ERC) under the European Union’s Horizon 2020 research and innovation programme (grant agreement No 787304).

\bibliographystyle{plain}
\bibliography{biblio}

\begin{thebibliography}{10}

\bibitem{adam}
Alexander Adam.
\newblock Generic non-trivial resonances for {A}nosov diffeomorphisms.
\newblock {\em Nonlinearity}, 30(3):1146--1164, 2017.

\bibitem{port}
Viviane Baladi.
\newblock Optimality of {R}uelle's bound for the domain of meromorphy of
  generalized zeta functions.
\newblock {\em Portugal. Math.}, 49(1):69--83, 1992.

\bibitem{linresp}
Viviane Baladi.
\newblock Linear response, or else.
\newblock In {\em Proceedings of the {I}nternational {C}ongress of
  {M}athematicians---{S}eoul 2014. {V}ol. {III}}, pages 525--545. Kyung Moon
  Sa, Seoul, 2014.

\bibitem{Bal2}
Viviane Baladi.
\newblock {\em Dynamical {Z}eta {F}unctions and {D}ynamical {D}eterminants for
  {H}yperbolic {M}aps}, volume~68 of {\em Ergebnisse}.
\newblock Springer, 2018.

\bibitem{Tsu}
Viviane Baladi and Masato Tsujii.
\newblock Dynamical determinants and spectrum for hyperbolic diffemorphisms.
\newblock {\em Geometric and probabilistic structures in dynamics}, Amer. Math.
  Soc., Providence, RI(469):29--68, 2008.

\bibitem{BanNaud}
Oscar Bandtlow and Frédéric Naud.
\newblock Lower bounds for the {R}uelle spectrum of analytic expanding circle
  maps.
\newblock {\em arXiv : 1605.06247}, 2016.
\newblock (to appear ETDS).

\bibitem{Boas}
Ralph Boas.
\newblock {\em Entire functions}.
\newblock Academic {P}ress, 1954.

\bibitem{whitnqa}
José Bonet, Rüdiger~W. Braun, Reinhold~G. Meise, and B.~Alan Taylor.
\newblock Whitney's extension theorem for nonquasianalytic classes of
  ultradifferentiable functions.
\newblock {\em Studia Math.}, 99(2):155--184, 1991.

\bibitem{sabowen}
Rufus Bowen.
\newblock One-dimensional hyperbolic sets for flows.
\newblock {\em J. Differential Equations}, 12:173--179, 1972.

\bibitem{pilipoconv}
Richard~D. Carmichael and Stevan Pilipovi\'c.
\newblock On the convolution and the {L}aplace transformation in the space of
  {B}eurling-{G}evrey tempered ultradistributions.
\newblock {\em Math. Nachr.}, 158:119--131, 1992.

\bibitem{PLgevdet}
Hua Chen and Luigi Rodino.
\newblock Nonlinear microlocal analysis and applications in {G}evrey classes.
\newblock In {\em Differential equations, asymptotic analysis, and mathematical
  physics ({P}otsdam, 1996)}, volume 100 of {\em Math. Res.}, pages 47--53.
  Akademie Verlag, Berlin, 1997.

\bibitem{DR}
Nguyen~Viet Dang and Gabriel Rivière.
\newblock Spectral analysis of {M}orse-{S}male gradient flows.
\newblock
  http://math.univ-lyon1.fr/homes-www/dang/Dang-Riviere-gradient-flow-2018.pdf(to
  appear Annales de l'ENS).

\bibitem{DZdet}
Semyon Dyatlov and Maciej Zworski.
\newblock Dynamical zeta functions for {A}nosov flows via microlocal analysis.
\newblock {\em Ann. Sci. \'Ec. Norm. Sup\'er. (4)}, 49(3):543--577, 2016.

\bibitem{fried}
David Fried.
\newblock The zeta functions of {S}elberg and {R}uelle. {I}.
\newblock {\em Ann. Sci. \'Ecole Norm. Sup.}, 19(4):491--517, 1986.

\bibitem{friedzeta}
David Fried.
\newblock Meromorphic zeta functions for analytic flows.
\newblock {\em Comm. Math. Phys.}, 174(1):161--190, 1995.

\bibitem{gevorg}
Maurice Gevrey.
\newblock Sur la nature analytique des solutions des \'equations aux
  d\'eriv\'ees partielles. {P}remier m\'emoire.
\newblock {\em Ann. Sci. \'Ecole Norm. Sup. (3)}, 35:129--190, 1918.

\bibitem{Gohb}
Israel Gohberg, Seymour Goldberg, and Nahum Krupnik.
\newblock {\em {T}races and {D}eterminants of {L}inear {O}perators}.
\newblock Operator theory : Advances and applications, 116. Birkhaüser Verlag,
  Basel-Boston-Berlin, 2000.

\bibitem{GouLiv2}
S\'ebastien Gou\"ezel and Carlangelo Liverani.
\newblock Compact locally maximal hyperbolic sets for smooth maps: fine
  statistical properties.
\newblock {\em J. Differential Geom.}, 79(3):433--477, 2008.

\bibitem{GLK}
Sebastien Gouëzel and Carlangelo Liverani.
\newblock Banach spaces adapted to {A}nosov systems.
\newblock {\em Ergod. Th. and Dyn. Sys.}, 26:189--217, 2006.

\bibitem{Groth}
Alexandre Grothendieck.
\newblock {\em Produits tensoriels topologiques et espaces nucléaires}.
\newblock American Mathematical Society, 1955.

\bibitem{jenk}
Oliver Jenkinson and Mark Pollicott.
\newblock Rigorous effective bounds on the {H}ausdorff dimension of continued
  fraction {C}antor sets: a hundred decimal digits for the dimension of
  {$E_2$}.
\newblock {\em Adv. Math.}, 325:87--115, 2018.

\bibitem{ltfzwor}
Long Jin and Maciej Zworski.
\newblock A local trace formula for {A}nosov flows.
\newblock {\em Ann. Henri Poincar\'e}, 18(1):1--35, 2017.
\newblock With appendices by Fr\'ed\'eric Naud.

\bibitem{Ko1}
Hikosaburo Komatsu.
\newblock Ultradistributions. {I}. {S}tructure theorems and a characterization.
\newblock {\em J. Fac. Sci. Univ. Tokyo Sect. IA Math.}, 20:25--105, 1973.

\bibitem{Ko2}
Hikosaburo Komatsu.
\newblock Ultradistributions. {II}. {T}he kernel theorem and ultradistributions
  with support in a submanifold.
\newblock {\em J. Fac. Sci. Univ. Tokyo Sect. IA Math.}, 24(3):607--628, 1977.

\bibitem{Ko3}
Hikosaburo Komatsu.
\newblock Ultradistributions. {III}. {V}ector-valued ultradistributions and the
  theory of kernels.
\newblock {\em J. Fac. Sci. Univ. Tokyo Sect. IA Math.}, 29(3):653--717, 1982.

\bibitem{nqa}
Andreas Kriegl, Peter~W. Michor, and Armin Rainer.
\newblock The convenient setting for non-quasianalytic {D}enjoy-{C}arleman
  differentiable mappings.
\newblock {\em J. Funct. Anal.}, 256(11):3510--3544, 2009.

\bibitem{livdet}
Carlangelo Liverani.
\newblock Fredholm determinants, {A}nosov maps and {R}uelle resonances.
\newblock {\em Discrete Contin. Dyn. Syst.}, 13(5):1203--1215, 2005.

\bibitem{LivTsu}
Carlangelo Liverani and Masato Tsujii.
\newblock Zeta functions and dynamical systems.
\newblock {\em Nonlinearity}, 19(10):2467--2473, 2006.

\bibitem{MunozMarco}
Vicente Mu\~noz and Ricardo P\'erez~Marco.
\newblock On the genus of meromorphic functions.
\newblock {\em Proc. Amer. Math. Soc.}, 143(1):341--351, 2015.

\bibitem{MunozMarco2}
Vicente Mu\~noz and Ricardo P\'erez~Marco.
\newblock Unified treatment of explicit and trace formulas via
  {P}oisson--{N}ewton formula.
\newblock {\em Comm. Math. Phys.}, 336(3):1201--1230, 2015.

\bibitem{Palais}
Richard~S. Palais.
\newblock Local triviality of the restriction map for embeddings.
\newblock {\em Comment. Math. Helv.}, 34:305--312, 1960.

\bibitem{PP}
William Parry and Mark Pollicott.
\newblock Zeta functions and the periodic orbit structure of hyperbolic
  dynamics.
\newblock {\em Ast\'erisque}, (187-188):268, 1990.

\bibitem{pilipotemp}
Stevan Pilipovi\'c.
\newblock Tempered ultradistributions.
\newblock {\em Boll. Un. Mat. Ital. B (7)}, 2(2):235--251, 1988.

\bibitem{Poll}
Mark Pollicott and Polina Vytovna.
\newblock Linear response and periodic points.
\newblock {\em Nonlinearity}, 29(10):3047--3066, 2016.

\bibitem{Ruelle1}
David Ruelle.
\newblock {Z}eta-{F}unctions for {E}xpanding {M}aps and {A}nosov {F}lows.
\newblock {\em Inventiones Math.}, 34:231--242, 1976.

\bibitem{R1}
Hans~Henrik Rugh.
\newblock The correlation spectrum for hyperbolic analytic maps.
\newblock {\em Nonlinearity}, 5:1237--1263, 1992.

\bibitem{R2}
Hans~Henrik Rugh.
\newblock {G}eneralized {F}redholm determinants and {S}elberg zeta functions
  for {A}xiom {A} dynamical systems.
\newblock {\em Ergod. Th. and Dyn. Sys.}, 16:805--819, 1996.

\bibitem{whitney}
Hassler Whitney.
\newblock Analytic extensions of differentiable functions defined in closed
  sets.
\newblock {\em Trans. Amer. Math. Soc.}, 36(1):63--89, 1934.

\end{thebibliography}

\end{document}